\documentclass{amsart}

\usepackage{url}
\DeclareUnicodeCharacter{221E}{infinity}
\usepackage[T1]{fontenc}

\usepackage{amsmath,amssymb,amstext,amsthm}
\usepackage{mathrsfs}
\usepackage{mathtools}
\usepackage{todonotes}
\usepackage{xcolor}
\usepackage{esint}
\usepackage[shortlabels]{enumitem}
\usepackage{tikz-cd}
\usepackage[all,cmtip]{xy}
\usepackage{hyperref}
\usepackage{graphicx}
\usepackage{wrapfig}
\usepackage{microtype}
\usepackage[weather]{ifsym}
\usepackage{wasysym}

\usepackage{tensor}
\usepackage{faktor}
\usepackage{bbm}	

\usepackage[margin=1.5in]{geometry}
\newcommand{\Day}{\mathbin{\textnormal{\sun}}} 

\usepackage[noabbrev,nameinlink,capitalise]{cleveref}		

\theoremstyle{definition} 
\numberwithin{equation}{section}

\newtheorem{introtheorem}{Theorem}[section]

\newtheorem{theorem}{Theorem}[section]
\newtheorem*{theorem*}{Theorem}
\newtheorem{lemma}[theorem]{Lemma}
\newtheorem{cor}[theorem]{Corollary}
\newtheorem{prop}[theorem]{Proposition}
\newtheorem*{prop*}{Proposition}

\newtheorem{constr}[theorem]{Construction}
\newtheorem{ex}[theorem]{Example}

\newtheorem{defin}[theorem]{Definition}
\newtheorem{defin*}[theorem]{Definition}
\newtheorem{obs}[theorem]{Observation}
\newtheorem{reminder}[theorem]{Reminder}
\newtheorem{notat}[theorem]{Notation}

\newtheorem{rem}[theorem]{Remark}
\newtheorem{warning}[theorem]{Warning}

\Crefname{rem}{Remark}{Remarks}
\Crefname{reminder}{Reminder}{Reminders}
\Crefname{obs}{Observation}{Observations}
\Crefname{cor}{Corollary}{Corollaries}
\Crefname{prop}{Proposition}{Propositions}
\Crefname{defn}{Definition}{Definitions}
\Crefname{thm}{Theorem}{Theorems}
\Crefname{constr}{Construction}{Contstructions}
\Crefname{notat}{Notation}{Notations}
\Crefname{introtheorem}{Theorem}{Theorems}
\Crefname{subappendix}{Section}{Sections}
\Crefname{corollary}{Corollary}{Corollaries}
\Crefname{axioms}{Axiom}{Axioms}
\Crefname{exercise}{Exercise}{Exercises}
\Crefname{exercisenum}{Exercise}{Exercises}
\Crefname{construction}{Construction}{Constructions}
\Crefname{problem}{Problem}{Problems}
\Crefname{theorem}{Theorem}{Theorems}
\Crefname{defin}{Definition}{Definitions}
\Crefname{proposition}{Proposition}{Propositions}
\Crefname{lemma}{Lemma}{Lemmas}
\Crefname{lem}{Lemma}{Lemmas}
\Crefname{remark}{Remark}{Remarks}
\Crefname{example}{Example}{Examples}
\Crefname{quest}{Question}{Questions}
\Crefname{examplealph}{Example}{Examples}
\Crefname{section}{Section}{Sections}
\Crefname{subsection}{Subsection}{Subsections}
\Crefname{summary}{Summary}{Summaries}
\Crefname{warning}{Warning}{Warnings}
\Crefname{part}{Part}{Parts}
\Crefname{conjecture}{Conjecture}{Conjectures}

\crefformat{nul}{(#2#1#3)} 
\crefname{nul}{}{}
\Crefformat{nul}{(#2#1#3)}
\Crefname{nul}{}{}

\newcommand{\N}{\mathbb{N}}
\newcommand{\C}{\mathbb{C}}
\newcommand{\E}{\mathbb{E}}

\newcommand{\unicodeinfty}{∞}
\newcommand{\Hom}{\operatorname{Hom}}
\newcommand{\Endo}{\operatorname{End}}
\newcommand{\Fun}{\operatorname{Fun}}

\newcommand{\LinEnd}{\Endo^\rL_{\mathcal{V}}}
\newcommand{\Nat}{\operatorname{Nat}}
\newcommand{\Arr}{\operatorname{Arr}}
\newcommand{\Ind}{\operatorname{Ind}}
\newcommand{\lInd}{\Ind_{\tav}}

\newcommand{\iHom}{\underline{\Hom}}

\newcommand{\Env}{\operatorname{Env}}
\newcommand{\vEnv}{\operatorname{vEnv}}
\newcommand{\Sym}{\operatorname{Sym}}

\newcommand{\RMod}{\operatorname{RMod}}

\newcommand{\Spc}{\mathcal{S}}
\newcommand{\Sp}{\mathcal{S} {p}}
\newcommand{\Spcn}{\Sp^{\mathrm{cn}}}
\newcommand{\Spaces}{\Spc}

\newcommand{\Ab}{\operatorname{Ab}}

\newcommand{\Set}{\operatorname{Set}}

\newcommand{\id}{\operatorname{id}}

\newcommand{\Cat}{\mathop{\mathcal{C}at}\nolimits}

\newcommand{\BMod}[2]{{}_{#1 \!} \operatorname{Bimod}_{#2}}
\newcommand{\Catcolim}{{\Cat} {}^{\mathrm{colim}}} 

\newcommand{\Ass}{\operatorname{Ass}}
\newcommand{\Alg}{\operatorname{Alg}}
\newcommand{\CAlg}{\operatorname{CAlg}}
\newcommand{\CMon}{\operatorname{CMon}}

\newcommand{\RM}{\operatorname{RM}}

\newcommand{\Op}{\mathop{\mathcal{O}\! \mathit{p}}\nolimits}

\newcommand{\vOp}{\mathop{\mathit{v}\mathcal{O}\! \mathit{p}}\nolimits}

\newcommand{\Fin}{\operatorname{Fin}_*}
\newcommand{\fset}[1]{{\underline{#1}_{+}}}

\newcommand{\PrV}{\Pr_{\mathcal{V}}}

\newcommand{\PrW}{\Pr_{\mathcal{W}}}

\renewcommand{\Pr}{\operatorname{Pr}\nolimits}

\newcommand{\IOp}{\mathop{{\mathbb{O} \mathrm{p}}}}

\newcommand{\Map}{\operatorname{Map}}

\newcommand{\EM}{\mathcal{EM}}
\newcommand{\Kl}{\mathrm{Kleisli}}

\newcommand{\PShO}{\PSh^{\otimes}_\cV}

\newcommand{\PSh}{{\mathcal{P}}} 
\newcommand{\lPSh}{\hat{\PSh}}

\newcommand{\PShV}{{\PSh_\mathcal{V}}}
\newcommand{\coPShV}{\mathop{\PSh^{\vee\!}_{\cV}}\nolimits}

\newcommand{\LMod}{\operatorname{LMod}}

\newcommand{\Mul}{\operatorname{Mul}}

\newcommand{\kerodon}[1]{\cite[\href{https://kerodon.net/tag/#1}{Tag #1}]{kerodon}}

\newcommand{\CatV}{\Cat(\mathcal{V})}

\newcommand{\vCat}{{\mathop{v\mathcal{C}at}\nolimits}}

\newcommand{\lCat}{\widehat{\Cat}}

\newcommand{\lSpaces}{\hat{\Spaces}}
\newcommand{\vCatV}{\vCat(\mathcal{V})}

\newcommand{\vCatXV}{\vCat_X(\mathcal{V})}

\newcommand{\vEnr}{\mathop{v\mathcal{E}nr}\nolimits}

\newcommand{\FOp}{{\mathop{\!\mathcal{FO}p}\nolimits}}

\newcommand{\ColV}{\mathrm{Col}_{\mathcal{V}}}

\newcommand{\ob}{ob}
\newcommand{\col}{col}
\newcommand{\FreeV}{\operatorname{Free}\nolimits^{\mathcal{V}}_!}

\newcommand{\Free}{\mathrm{Free}}

\newcommand{\coCart}{\mathrm{coCart}}

\newcommand{\op}{\mathrm{op}}

\newcommand{\CIm}{\operatorname{CIm}}

\newcommand{\ICat}{\mathbb{C}\mathrm{at}}
\newcommand{\IlCat}{\widehat{\mathbb{C}\mathrm{at}}}

\newcommand{\ICatcolim}{\mathbb{C}\mathrm{at}^{\mathrm{colim}}}

\newcommand{\IPr}{\mathbb{P}\mathrm{r}}

\newcommand{\IPrV}{\IPr_{\cV}}

\newcommand{\Morita}{\mathbb{M}\mathrm{orita}}

\newcommand{\iNat}{\underline{\Nat}}

\DeclareMathOperator*\colim{colim}

\newcommand{\oArr}[2]{\overunderset{#1}{#2}{\downarrow}}	

\makeatletter
\newcommand{\otriang}{\mathbin{\mathpalette\make@circled{\triangleright}}}
\newcommand{\ostar}{\mathbin{\mathpalette\make@circled{\star}}}
\newcommand{\ootimes}{\mathbin{\mathpalette\make@bigcircled\otimes}}
\newcommand{\make@circled}[2]{%
	\ooalign{$\m@th#1\smallbigcirc{#1}$\cr\hidewidth$\m@th#1#2$\hidewidth\cr}%
}
\newcommand{\smallbigcirc}[1]{%
	\vcenter{\hbox{\scalebox{0.77778}{$\m@th#1\bigcirc$}}}%
}
\newcommand{\make@bigcircled}[2]{%
	\ooalign{$\m@th#1\bigsmallbigcirc{#1}$\cr\hidewidth$\m@th#1#2$\hidewidth\cr}%
}
\newcommand{\bigsmallbigcirc}[1]{%
	\vcenter{\hbox{\scalebox{1}{$\m@th#1\bigcirc$}}}%
}
\makeatother

\makeatletter
\newcommand{\oeert}{\mathbin{\mathpalette\eert@h\relax}}
\newcommand{\otree}{\mathbin{\mathpalette\tree@h\relax}}
\newcommand{\tree@h}[2]{%
	\ooalign{%
		\hfil$\m@th#1\ominus$\hfil\cr
		\hfil\scalebox{0.5}[1]{$\m@th#1<$}\hfil\cr
	}%
}
\newcommand{\eert@h}[2]{%
	\rotatebox[origin = c]{180}{\ooalign{%
			\hfil$\m@th#1\ominus$\hfil\cr
			\hfil\scalebox{0.5}[1]{$\m@th#1<$}\hfil\cr
		}%
	}
}
\makeatother

\makeatletter												
\newcommand{\doublehat}[1]{%
	\begingroup%
	\let\macc@kerna\z@%
	\let\macc@kernb\z@%
	\let\macc@nucleus\@empty%
	\hat{\raisebox{.2ex}{\vphantom{\ensuremath{#1}}}\smash{\hat{#1}}}%
	\endgroup%
}
\makeatother

\makeatletter
\newcommand{\oset}[3][0ex]{%
	\mathrel{\mathop{#3}\limits^{
			\vbox to#1{\kern-2\ex@
				\hbox{$\scriptstyle#2$}\vss}}}}
\makeatother

\DeclareSymbolFont{bbold}{U}{bbold}{m}{n}
\DeclareSymbolFontAlphabet{\mathbbold}{bbold}

\usepackage{eucal}                  

\DeclareFontFamily{T1}{cbgreek}{}
\DeclareFontShape{T1}{cbgreek}{m}{n}{<-6>  grmn0500 <6-7> grmn0600 <7-8> grmn0700 <8-9> grmn0800 <9-10> grmn0900 <10-12> grmn1000 <12-17> grmn1200 <17-> grmn1728}{}
\DeclareSymbolFont{quadratics}{T1}{cbgreek}{m}{n}
\DeclareMathSymbol{\qoppa}{\mathord}{quadratics}{19}
\DeclareMathSymbol{\Qoppa}{\mathord}{quadratics}{21}

\DeclareFontFamily{U}{rcjhbltx}{}	
\DeclareFontShape{U}{rcjhbltx}{m}{n}{<->rcjhbltx}{}
\DeclareSymbolFont{hebrewletters}{U}{rcjhbltx}{m}{n}

\DeclareMathSymbol{\lamed}{\mathord}{hebrewletters}{108}
\DeclareMathSymbol{\mem}{\mathord}{hebrewletters}{109}
\DeclareMathSymbol{\ayin}{\mathord}{hebrewletters}{96}
\DeclareMathSymbol{\tsadi}{\mathord}{hebrewletters}{118}
\DeclareMathSymbol{\qof}{\mathord}{hebrewletters}{113}
\DeclareMathSymbol{\shin}{\mathord}{hebrewletters}{152}
\DeclareMathSymbol{\tav}{\mathord}{hebrewletters}{116}

\DeclareFontFamily{U}{dmjhira}{}
\DeclareFontShape{U}{dmjhira}{m}{n}{ <-> dmjhira }{}

\DeclareRobustCommand{\yo}{\text{\usefont{U}{dmjhira}{m}{n}\symbol{"48}}}

\def\cC{\mathcal C}\def\cD{\mathcal D}

\def\cK{\mathcal K}
\def\cM{\mathcal M}\def\cN{\mathcal N}\def\cO{\mathcal O}\def\cP{\mathcal P}

\def\cV{\mathcal V}\def\cW{\mathcal W}\def\cX{\mathcal X}

\newcommand{\bC}{{\mathbb C}}
\newcommand{\bD}{{\mathbb D}}

\newcommand{\rB}{{\mathrm B}}
\newcommand{\rC}{{\mathrm C}}
\newcommand{\rD}{{\mathrm D}}
\newcommand{\rE}{{\mathrm E}}

\newcommand{\rI}{{\mathrm I}}

\newcommand{\rK}{{\mathrm K}}
\newcommand{\rL}{{\mathrm L}}
\newcommand{\rM}{{\mathrm M}}
\newcommand{\rN}{{\mathrm N}}
\newcommand{\rO}{{\mathrm O}}
\newcommand{\rP}{{\mathrm P}}
\newcommand{\rQ}{{\mathrm Q}}
\newcommand{\rR}{{\mathrm R}}

\newcommand{\rV}{{\mathrm V}}
\newcommand{\rW}{{\mathrm W}}

\newcommand{\HAsec}[1]{\href{http://www.math.ias.edu/~lurie/papers/HA.pdf\#section.#1}{\S #1}}

\newcommand{\HAsubsec}[1]{\href{http://www.math.ias.edu/~lurie/papers/HA.pdf\#subsection.#1}{\S #1}}

\newcommand{\HTTthm}[2]{\href{http://www.math.ias.edu/~lurie/papers/HTT.pdf\#theorem.#2}{#1~#2}}
\newcommand{\HAthm}[2]{\href{http://www.math.ias.edu/~lurie/papers/HA.pdf\#theorem.#2}{#1~#2}}

\newcommand{\SAGthm}[2]{\href{http://www.math.ias.edu/~lurie/papers/SAG-rootfile.pdf\#theorem.#2}{#1~#2}}

\newcommand{\HTT}[2]{\cite[\HTTthm{#1}{#2}]{HTT}}
\newcommand{\HA}[2]{\cite[\HAthm{#1}{#2}]{HA}}

\newcommand{\SAG}[2]{\cite[\SAGthm{#1}{#2}]{SAG}}

\newcommand{\markedthm}[2]{\href{http://www.markus-zetto.com/marked-rootfile.pdf\#theorem.#2}{#1~#2}}
\newcommand{\marked}[2]{\cite[\markedthm{#1}{#2}]{marked}}

\title{Enriched \texorpdfstring{$\infty$}{\unicodeinfty}-operads as marked algebras}

\author{Markus Zetto}
\address{Fachbereich Mathematik, Universit\"at Hamburg}
\email{markus.zetto@uni-hamburg.de}
\urladdr{https://www.markus-zetto.com}

\date{\today{}, Hamburg}

\begin{document}
	
	\pagenumbering{gobble}
	\clearpage
	\hypersetup{pageanchor=false}

	\begin{abstract}
		

		We show that an enriched $\infty$-operad is completely determined by its category of right modules together with a `marking' of the representable modules. More precisely, for any presentably monoidal $\infty$-category $\cV$ we construct an equivalence between the category of colored $\cV$-enriched $\infty$-operads and a certain full subcategory of the category of presentably symmetric monoidal $\cV$-module $\infty$-categories equipped with a functor from an $\infty$-groupoid. This effectively allows us to reduce many aspects of enriched $\infty$-operad theory to the theory of presentably symmetric monoidal $\infty$-categories.
		
		As an application, we describe a notion of univalence (or Rezk-completeness) for enriched $\infty$-operads, and directly construct an equivalence between univalent $\Spaces$-enriched $\infty$-operads in our sense and Lurie's model of $\infty$-operads in \cite{HA}. We study envelopes and categories of algebras for enriched $\infty$-operads and show that, in the $\Spaces$-enriched case, the resulting notions agree in both models.

		

	\end{abstract}
	\maketitle
	\hypersetup{pageanchor=true}
	
	\setcounter{tocdepth}{1}
	\tableofcontents

	
	\pagenumbering{roman}
	
	\section{Introduction}
	\label{sec:intro}

	The notion of a \emph{(colored) operad}, or multicategory, generalizes that of a category by allowing for \emph{$k$-ary multimorphisms} with $k$ sources and a single target, for any $k \geq 0$. Operads were originally introduced in algebraic topology to encode the algebraic structure carried by iterated loop spaces \cite{boardman2006homotopy}, \cite{may2006geometry}. The \emph{little cubes operads}  $\E_n$ introduced for this purpose already indicated the need for a homotopy coherent notion of operads, as their multimorphisms spaces are given by configuration spaces and not merely sets. Incorporating this is one of the main achievements of the notion of $\infty$-operads introduced in \cite{HA}.
	
	Many operads in nature however, most prominently the Lie operad and its derived or spectral variants, do not only have spaces of multimorphisms but rather they are \emph{enriched}: their multimorphisms form objects in some symmetric monoidal category, for instance vector spaces or derived categories. Specifically the latter requires the enrichment category to be homotopy coherent as well, indicating the need for \emph{enriched $\infty$-operads}. One of their main applications is Koszul duality which requires semiadditive enrichment to be defined \cite[Corollary 2.27]{heinemilnor} and enhances to an equivalence for stable enrichment \cite{ching2012bar}, \cite[Thm. 3.4]{heuts2024koszul}, \cite[Prop. 3.6]{heinemilnor}. They also arise in mathematical physics (related to factorization algebras and chiral algebras), geometric representation theory, deformation theory and algebraic geometry; see the introduction of \cite{haugsengsseq} for further pointers.

	\subsection{Enriched \texorpdfstring{$\infty$}{\unicodeinfty}-operads}

	Several definitions of enriched multicolored $\infty$-operads were already given in \cite{haugsengchu} and \cite{haugsengsseq}. We here follow an approach prioneered by \cite{baez1998higher} for classical multicolored operads 
	and extended to single-colored $\infty$-operads in \cite{brantner2017lubin}. See §1.2 of \cite{haugsengsseq} for a more detailed historical account.
	
	Fix a presentably symmetric monoidal $\infty$-category $\cV \in \CAlg(\Pr)$ as our enrichment category. Given a space $X$, let $\Sym X = \coprod_{k \geq 0} X^{\times k}_{h \Sigma_k}$ be the coproduct over its symmetric powers, i.e.\ the free symmetric monoidal $\infty$-category on $X$; and $\Fun(\Sym X, \cV)$ equipped with Day convolution the free presentably symmetric monoidal $\cV$-linear $\infty$-category on it. 
	This means that there is an equivalence
	\[
	\Endo^{\rL, \otimes}_\cV( \Fun(\Sym X, \cV) ) \simeq \Fun(X, \Fun(\Sym X, \cV)) \simeq \Fun(\Sym X \times X, \cV) \; .
	\]
	We show in \cref{sec:appstar} that presentably symmetric monoidal $\cV$-linear $\infty$-categories form a $2$-category $\CAlg(\RMod_\cV(\IPr))$, where we write $\IPr$ instead of $\Pr$ to indicate the $2$-categorical structure. Transporting the composition operation on endofunctors along this equivalence hence induces a monoidal structure $\oeert$ on the right, (the reverse of) which is often referred to as the \emph{composition product} or \emph{plethysm product} on the category of \emph{$X$-colored symmetric sequences}. Explicitly, given $A, B \in \Fun(\Sym X \times X, \cV)$, we show in \cref{cor:compositionproduct} that for $(x_1, \dots, x_n) \in \Sym X$ and $z \in X$ the product $(A \oeert B)(x_1, \dots, x_n; z)$ can be expressed as 
	\[
	\bigsqcup_{\underset{k \geq 0}{I_1 \sqcup \dots \sqcup I_k = \{1, \dots, n\}}} \colim_{(y_1, \dots, y_k) \in X^{\times k}} \left( A((x_{i_1})_{i_1 \in I_1}; y_1) \otimes \dots \otimes A((x_{i_k})_{i_k \in I_k};  y_k) \right) \otimes_{\Sigma_k} B(y_1 , \dots , y_k; z)
	\]
	where we take a coproduct over the partitions of $\{1, \dots, n\}$, and the $\Sigma_k$-action permutes the blocks $(x_{i_1})_{i_1 \in I_1}, \dots, (x_{i_k})_{i_k \in I_k}$ partitioning $(x_1, \dots, x_n)$, as well as the $y_1, \dots, y_k$. 
	
	An object $\Mul_\cO \in \Fun(\Sym X \times X, \cV)$ assigns to a finite symmetric tuple of source colors in $X$ and a single target color an object of $\cV$, which is precisely the data of multimorphisms in a $\cV$-enriched operad. Equipping $\Mul_\cO$ with an algebra structure for $\oeert$ encodes homotopy-coherent composition and identities. This motivates the following definition:
	
	\begin{defin}
		A \emph{$\cV$-enriched operad with space of colors} $X$ is a monad on $\Fun(\Sym X, \cV)$ in the $(\infty,2)$-category $\CAlg(\RMod_\cV(\IPr))$. In other words, it is an algebra for the composition product $\oeert$ on $\Fun( \Sym X \times X, \cV)$. We denote their category by 
		\[ \vOp_X(\cV) := \Alg(\Endo^{\rL, \otimes}_\cV( \Fun(\Sym X, \cV) )) \; . \]
	\end{defin}

	\subsection{Marked algebras}

	Given that we have defined an enriched $\infty$-operad  as a monad $\cO\in \Alg( \Endo^{\rL, \otimes}_\cV( \Fun(\Sym X, \cV) ))$, we define an \emph{operadic presheaf}\footnote{This is usually referred to as a \emph{right $\cO$-module}; we have opted for a different name since it is more suggestive, and since we model them as left modules due to our composition product being reverse.} to be an algebra over this monad. The \emph{operadic presheaf category} $\PShO(\cO)$ of $\cO$ is hence defined as its Eilenberg-Moore category in $\CAlg(\RMod_\cV(\IPr))$, and thereby part of a monadic adjunction
	\[
	\Fun(\Sym X, \cV) \rightleftarrows \LMod_\cO(\Fun(\Sym X, \cV)) =: \PShO(\cO) \; .
	\]
	Unwinding this definition using the explicit expression for the phlethysm product from above, we see that an operadic presheaf consists of a functor $W: \Sym(X) \to \cV$ equipped with an action morphism $\cO \oeert W \to W$ in $\cV$, i.e.\ a morphism
	\[
	\left( \cO((x_{i_1})_{i_1 \in I_1}; y_1) \otimes \dots \otimes \cO((x_{i_k})_{i_k \in I_k};  y_k) \right) \otimes_{\Sigma_k} W(y_1 , \dots , y_k) \to W(x_1, \dots, x_n)
	\]
	for every partition $I_1 \sqcup \dots \sqcup I_k = \{1, \dots, n\}$ with $k \geq 0$ and any $y_1, \dots, y_k \in X$, satisfying coherence relations. For a $\Spaces$-enriched operad, we show in \cref{obs:operadicpshspaces} that the operadic presheaf category can be expressed as $\PSh^\otimes_{\Spaces}(\cO) \simeq \PSh \Env (\cO)$, where the \emph{envelope} $\Env(\cO)$ is the free symmetric monoidal category generated by the $\infty$-operad $\cO$, i.e.\ left adjoint to the forgetful functor $\CAlg(\Cat) \to \Op$. For enriched operads, we define an enriched envelope functor such that this holds \cref{defin:envelopefun}.
	
	An an example, since $\cO$ is an algebra, the functor $\cO(-; x)$ for any $x \in X$ admits an action as above. We refer to it as the \emph{representable operadic presheaf} $\yo_x$ on $x$, and define the \emph{operadic Yoneda embedding} as the composition $X \subseteq \Sym X \subseteq \Fun(\Sym (X), \Spaces) \to \Fun(\Sym (X), \cV) \to \PShO(\cO)$.

	In $\cV$-enriched category theory, representable presheaves $\yo_c$ have the property that they are \emph{atomic}, i.e.\ the functor
	\[
	\iHom_{\PShV(\cC)}(\yo_c, -) \simeq \mathrm{ev}_c : \PShV(\cC) \to \cV
	\]
	preserves colimits and $\cV$-tensorings (equivalently, weighted colimits), where $\iHom$ denotes the internal hom for the $\cV$-tensoring. In the operadic setting, this is replaced by the following condition:
	
	\begin{defin}[{\cref{obs:hereditary}}]
		Let $\cM \in \CAlg(\RMod_\cV(\Pr))$ be a presentably symmetric monoidal $\cV$-module $\infty$-category, and $X \in \Spaces$ by a space. Then a functor $y: X \to \cM$ is called a \emph{$\otimes$-atomic marking} if 
		for any $x_1, \dots, x_n \in X$ with $n \geq 0$,
		\begin{itemize}
			\item The product $y (x_1) \otimes \dots \otimes y (x_n)$ is an atomic object in $\cM$,
			\item The unitor induces isomorphisms $\emptyset_\cV \simeq \iHom_\cM(y(x_1) \otimes \dots \otimes y(x_n), 1_\cM)$ unless $n=0$, in which case $1_\cV \simeq \iHom_\cM(1_\cM, 1_\cM)$,
			\item For any $m_1, m_2 \in \cM$, the multiplication map
			\[
			\bigsqcup_{\underline{n} = S \sqcup T} \iHom_\cM \left( \bigotimes_{s \in S} y (x_s), m_1 \right) \otimes \iHom_\cM \left( \bigotimes_{t \in T} y (x_t), m_2 \right) \overset{\otimes}{\to} \iHom_\cM \left( \bigotimes_{k=1}^n y (x_k) , m_1 \otimes m_2 \right)
			\]
			is an isomorphism. This is known as the \emph{hereditary condition}, c.f.\ \cite[Prop. 2.4.6, Rem. 2.4.7]{haugsengenv}.
		\end{itemize}
		An object $m \in \cM$ is called \emph{$\otimes$-atomic} if the functor $* \to \cM$ hitting $m$ is a $\otimes$-atomic marking.
	\end{defin}
	This completely characterizes those maps $X \to \cM$ which arise as the operadic Yoneda embeddings of enriched operads:

	\begin{introtheorem}[{\cref{thm:markedalgebras}}] \label{introthm:markedalgebras}
	For a space $X$, the assignment of the operadic Yoneda embedding $\yo : X \to \PShO(\cO)$ to a $\cV$-enriched operad $\cO$ with space of objects $X$ induces a fully faithful embedding
		\[
		\vOp_X(\cV) = \Alg\left(\Endo^{\rL, \otimes}_\cV(\Fun(\Sym X, \cV))\right) \hookrightarrow \CAlg (\RMod_\cV(\Pr))_{X/}
		\]
		whose image consists those functors $y: X \to \cM$ with $\cM \in \CAlg (\RMod_\cV(\Pr))$ such that:
		\begin{itemize}
			\item $y$ is a $\otimes$-atomic marking, 
			\item The image of $y$ generates $\cM$ under colimits, symmetric monoidal structure and $\cV$-tensoring.
		\end{itemize}
		We refer to functors $y: X \to \cM$ in the essential image of this embedding as \emph{marked $\cV$-algebras}. 
	\end{introtheorem}

	This closely follows the main idea of \cite{marked}. Indeed, $y: X \to \cM$ being a $\otimes$-atomic marking is equivalent to the induced cocontinuous symmetric monoidal functor $\Fun(\Sym X, \cV) \to \cM$ being a left adjoint $1$-morphism in the $(\infty,2)$-category $\CAlg(\RMod_\cV(\IPr))$. The second conditon of \cref{introthm:markedalgebras} is the further assertion that this adjunction is monadic. Since $\PShO(\cO)$ is the Eilenberg-Moore object of this adjunction, \cref{introthm:markedalgebras} then follows from general statements about monads.

	\subsection{Fissile categories}
	
	In stark contrast to the situation with enriched $\infty$-categories in \cite{marked}, the condition of being a $\otimes$-atomic marking cannot be checked one object at a time, i.e.\ $y: X \to \cM$ being $\otimes$-atomic is \emph{not} equivalent to (but does imply) $y(x)$ being $\otimes$-atomic for all $x \in X$. This is however true under the following additional assumption on $\cM$: 
	
	\begin{defin}[{\cref{def:fissile}}]
		A presentably symmetric monoidal $\cV$-module $\infty$-category $\cM$ is called \emph{fissile} if the unit $\cV \to \cM$ and the multiplication map $\otimes: \cM \otimes_\cV \cM \to \cM$ are left adjoint $1$-morphisms in $\CAlg(\RMod_\cV(\IPr))$.
	\end{defin}
	
	\begin{rem}
		This is very similar, but somewhat orthogonal to the notion of \emph{rigidity} in \cite[Chapter 1, §9]{gr}, where the multiplication is required to admit an adjoint as a $\cV$-$\cV$-bimodule map.
	\end{rem}
	
	We show in \cref{prop:iLfissile} that given a fissile category $\cM$, a functor $y: X \to \cM$ is a $\otimes$-atomic marking iff for every $x \in X$, the object $y(x) \in \cM$ is $\otimes$-atomic. A key example of a fissile category is $\Fun(\Sym X, \cV)$, since its multiplication is obtained by applying the free functor $\PSh \Sym (-) \otimes \cV$ to the fold map $\nabla: X \sqcup X \to X$. This generalizes to the following theorem: 
	
	\begin{introtheorem}[{\cref{thm:markedalgfissile}}]
	For every $\cV$-enriched operad $\cO$, the operadic presheaf category $\PShO(\cO)$ is fissile. In particular, the image of the functor from \cref{introthm:markedalgebras} may equivalently be characterized as consisting of those functors $y: X \to \cM$ such that \begin{itemize}
		\item $\cM$ is fissile,
		\item Every $x \in X$ is sent to a $\otimes$-atomic object in $\cM$,
		\item The image of $y$ generates $\cM$ under colimits, $\cV$-tensoring and symmetric monoidal structure.
	\end{itemize} 
	\end{introtheorem}
	
	\begin{rem}
		The notion of fissile categories is not technically necessary for most results in this paper. The reason why it is useful to separate from the notion of being $\otimes$-atomically generated is that it only refers to $\cM$, and not the existence of some functor $X \to \cM$. Where we do require it is for our discussion of colimits over enriched $\infty$-operads in \cref{prop:operadslimcolim}, and in future work \cite{cauchyoperads} to define the \emph{Cauchy-completion} of an enriched $\infty$-operad, see \cref{rem:cauchy}.
	\end{rem}

\subsection{Univalence and comparison to Lurie's \texorpdfstring{$\infty$}{\unicodeinfty}-operads}
	
Building on \cref{introthm:markedalgebras} we define the category of \emph{valent $\cV$-enriched operads} $\vOp(\cV)$ as the full subcategory of the pullback $\Fun([1], \lCat) \times_{\lCat} \CAlg(\RMod_\cV(\Pr))$ on the marked $\cV$-algebras. The terminology \emph{valent} as opposed to \emph{univalent} indicates that the space of objects $X$ that we begin with need not coincide with the actual underlying space of colors of our operad. To recover the correct underlying space, as suggested by the Yoneda lemma, one must require that $X$ embeds as a subcategory of $\PShO(\cO)$:
	
	\begin{defin*}
		A valent $\cV$-enriched operad $\cO$ is called \emph{univalent} if its associated marked algebra $\yo: \col \cO \to \PShO(\cO)$ is a subcategory inclusion, in other words $\col \cO \simeq \mathrm{Im}(\yo)^\simeq$. Denote by $\Op(\cV) \subseteq \vOp(\cV)$ the full subcategory on the univalent $\cV$-operads.
	\end{defin*}
	
	\begin{introtheorem}[{\cref{thm:vOpS},\cref{cor:OpS}}]
		The category $\vOp(\Spaces)$ of valent $\Spaces$-enriched operads is equivalent to the category of \emph{flagged $\infty$-operads}, i.e.\ pairs consisting of an $\infty$-operad $\rO$ in the sense of Lurie together with a surjective functor $X \to \underline{\rO}$ from some space $X$ to the underlying category of $\rO$.
		
		Moreover, the full subcategory $\Op(\Spaces)$ of univalent $\Spaces$-enriched operads is equivalent to Lurie's category $\Op$ of $\infty$-operads.
	\end{introtheorem}

	Let us sketch how this comparison works. We first explain how to associate a monad in $\CAlg(\Pr)$ to a flagged $\infty$-operad $X \to \rO$. Regard $X$ as an $\infty$-operad $\mathrm{Triv}_X$ with only unary operations, so that $\Env \mathrm{Triv}_X \simeq \Sym X \in \CAlg(\Cat)$. Applying the envelope functor $\Env$ and then the presheaf functor $\PSh$ to the map of operads $i: \mathrm{Triv}_X \to \rO$ yields an adjunction
	\[
	\begin{tikzcd}[ampersand replacement=\&]
		\PSh \Sym X  \ar[rr, bend left, "\Env(i)_!", "\otimes \text{, L}"'] \& \& \PSh \Env \cO \ar[ll, bend left, "\Env(i)^*", "\otimes \text{, L}"']
	\end{tikzcd}
	\]
	in which the left adjoint is symmetric monoidal and the right adjoint preserves colimits. In \cref{thm:dayweaklydisj} we show, using an argument due to Jan Steinebrunner, that the induced lax symmetric monoidal structure on the right adjoint is actually strong. Hence the above is an adjunction in the $2$-category $\CAlg(\IPr)$, so we obtain a symmetric monoidal colimit-preserving monad
	\[
	(\Env f)^* (\Env f)_! \in \Alg \left( \Endo^{\rL, \otimes} (\PSh \Sym X) \right) \simeq \Alg( \Fun(\Sym X \times X, \Spaces) ) \; .
	\]
	The latter equivalence precomposes with the Yoneda embedding $o \mapsto \yo_{(o)}$, so we calculate
	\[
	(\Env f)^* (\Env f)_! (\yo_{(o)}) = (\Env f)^* \yo_{(o)} = \Map_{\Env(\cO)}( - , (o)) \circ \Env (f)
	\]
	which sends a symmetric tuple $(o_1, \dots, o_n) \in \Sym (\col \cO)$ to
	\[
	\Map_{\Env(\cO)}( (o_1, \dots, o_n) , (o)) \simeq \Mul_\cO( o_1, \dots, o_n ; o) \; .
	\]
	Thus, the multimorphism object $\Mul_\cO$ is promoted to an algebra in $\Fun(\Sym X \times X, \Spaces)$, i.e.\ a $\Spaces$-enriched operad. By construction, the marked algebra associated to this $\Spaces$-operad is $X \to \PSh \Env \rO$, because the induced functor $\PSh \Sym X \to \PSh \Env \rO$ is indeed a monadic $1$-morphism in $\CAlg(\IPr)$. Hence the argument identifies $\PSh \Env \rO$ with the operadic presheaf category of $\rO$. We deduce that the enriched envelope we define in \cref{defin:envelopefun}, and the categories of $\cO$-algebras we define in \cref{defin:alg}, agree with Lurie's definitions \cref{rem:envagrees}, \cref{prop:algspaces}.
	
	\begin{rem}
		In this paper, we only consider symmetric enriched operads. In principle, our approach immediately extends to non-symmetric operads, or $\rO$-operads for any Lurie-operad $\rO$, by replacing the $(\infty, 2)$-category $\CAlg(\RMod_\cV(\IPr))$ with $\Alg_{\rO}(\RMod_\cV(\IPr))$, the symmetric algebra functor $\Sym$ with the free $\rO$-algebra functor $\mathrm{Free}_{\rO}$, and so on. The main difficulty arises in \cref{sec:appstar}, since there is no $\star$-product on $\Alg_\rO(\Cat)$ so the $(\infty,2)$-categorical structure on $\Alg_{\rO}(\RMod_\cV(\IPr))$ must be described via different means, e.g.\ by embedding it into $\IOp_{/\rO}$ or $\Fun(\rO^\otimes, \ICat)$. Also, our definition of fissile categories only makes sense in the symmetric situation.
	\end{rem}
	
	\begin{rem}
		While we only describe $\cV$-enriched operads for presentably symmetric monoidal enrichment category $\cV \in \CAlg(\RMod_\cV(\Pr))$, our discussion extends to enrichment in any $\rV \in \CAlg(\Cat)$ since we can symmetric monoidally embed $\rV \subseteq \PSh (\rV) \in \CAlg(\RMod_\cV(\Pr))$ into its presheaf category, so $\vOp(\rV) \subseteq \vOp(\PSh(\rV))$ can be defined as the full subcategory of $\PSh(\rV)$-enriched operads whose multimorphism objects all lie in $\rV \subseteq \PSh(\rV)$. Compare \marked{Def.}{9.4}. In fact, we may even enrich over a Lurie-operad $\rO \in \Op$ by embedding it into $\PSh \Env \rO \in \CAlg(\Pr)$.
	\end{rem}

\subsection{Organization of the paper}

In \cref{sec:enrichedoperads}, we define enriched $\infty$-operads as monads in the $(\infty,2)$-category $\CAlg(\RMod_\cV(\IPr))$, and develop some basic intuition and constructions. \cref{sec:tensoratomics} is concerned with studying internally left adjoint $1$-morphisms as well as monadic $1$-morphisms in $\CAlg(\RMod_\cV(\IPr))$, including criteria for characterizing and methods to construct them. 
We use this theory in \cref{sec:markedalgebras} to give an alternative description of enriched $\infty$-operads using marked algebras, and introduce the notion of fissile categories. In \cref{sec:comparison} we compare $\Spaces$-enriched $\infty$-operads with Lurie's notion of $\infty$-operads, which involves the notion of univalence. \cref{sec:envelopes} is about enriched envelopes, and \cref{sec:algebras} defines categories of algebras and for instance examines limits, colimits and free algebras. Finally, \cref{sec:appstar} constructs the $(\infty,2)$-category $\CAlg(\RMod_\cV(\IPr))$ we make use of throughout the paper, establishing several of its properties we use in the main text like the existence of Eilenberg-Moore objects.

\subsection{Relation to other work}

In \cite{haugsengmonads}, Haugseng proves that $\infty$-operads in the sense of Lurie are equivalent to algebras in $\Spaces$-enriched symmetric sequences. However, his construction of the composition product differs from the one obtained by identifying symmetric sequences with an endomorphism object in $\CAlg(\RMod_\cV(\IPr))$, and he does not supply a comparison between the two. As he notes in §1.2, this left open the question of whether this notion of enriched $\infty$-operads agrees with this more widely used one. During the preparation of this work, both monoidal structures were compared in the single-colored case by \cite{arakawa}. 

There is further unpublished work on our comparison result for $\Spaces$-enriched $\infty$-operads in the single-colored case: Lukas Brantner and Gijs Heuts \cite{brantnerheuts} are preparing a proof using the theory of analytic monads. Also Thomas Blom, Connor Malin and Niall Taggart \cite{blomkoszul} give a proof that is closely related to ours and describe Koszul duality in their formalism. Finally,  Shaul Barkan and Jan Steinebrunner were considering a similar framework that extends to colored enriched $\infty$-properads, see also the outlook of \cite{equifib}.

\subsection{Acknowledgements}

I would like to thank Fernando Abellán, Nikolaus Betker, Thomas Blom, Adam Dauser, Jonte Gödicke and Jonas Linßen for valuable discussions about the contents of this paper, as well as Tobias Dyckerhoff for the opportunity to talk about it in his Seminar. Particular thanks to Jan Steinebrunner for telling me about \cref{thm:dayweaklydisj} and further helpful comments, and to my PhD supervisor David Reutter for extensive feedback on a preliminary draft. The author was supported by the Deutsche Forschungsgemeinschaft via the Emmy Noether program – 493608176.

\subsection{Notation and Conventions}
	Throughout this paper we freely use language and statements of $\infty$-category theory as developed in \cite{joyal2002quasi}, \cite{HTT}, \cite{HA} and \cite{kerodon}. We also make extensive use of $(\infty,2)$-categories and enriched $\infty$-categories, for which we mostly refer to \cite{bienriched}, \cite{heinemonadicity} and \cite{marked}.

\begin{itemize}
	\item We write $\Spaces$ for the $\infty$-category of `spaces' or `$\infty$-groupoids'.
	\item By a category, $2$-category or $\cV$-category we always refer to an $(\infty, 1)$-category, $(\infty, 2)$-category or a $\cV$-enriched $\infty$-category. All categorical constructions, like functors, limits and algebra objects, are homotopy coherent.
	
	\item We fix uncountable inaccessible cardinals $\tav < \hat{\tav} < \doublehat{\tav}$, and call them the universes of \emph{small}, \emph{large} and \emph{very large} sets, respectively\footnote{$\tav$ (\emph{tav} or \emph{taw}) is the last letter in the Hebrew alphabet, used by Cantor to denote the ``absolute infinite''.}.
	
	\item Unless stated otherwise (e.g.\ by calling them presentable) categories are always small, and by having or preserving all colimits we refer to small colimits. The large analog of a construction is then denoted by a hat -- e.g.\ $\Cat$ denotes the large category of small categories, while $\widehat{\Cat}$ denotes the very large category of large categories.
	\item The mapping space in a category $\rC$ is denoted $\Map_\rC$, the morphism object in an enriched category $\cC$ by $\Hom_\cC$, and the internal Hom in a module category $\cM$ by $\iHom_\cM$.
	\item We denote by $\rC \hookrightarrow \rD$ a (not necessarily full) subcategory inclusion, i.e.\ a functor inducing a monomorphism $\rC^\simeq \hookrightarrow \rD^\simeq$ on maximal subspaces, and monomorphisms $\Map_\rC(c, c') \hookrightarrow \Map_\rD(Fc, Fc')$ for all $c, c' \in \rC$. Equivalently, these are the monomorphisms in $\Cat$ by \kerodon{04W5}.
	\item We work with categories \emph{univalently} and model-independently, meaning that we do not make use of evil notions. For instance, (full) subcategories are closed under isomorphisms, by the image or \emph{full image} of a functor we mean the essential image, and by a \emph{surjective} functor we refer to what is usually known as an essentially surjective functor.
	\item We write $\Arr(\rC) := \Fun([1], \rC)$ for the \emph{arrow category} of a category $\rC$, where $[1]$ is the walking arrow. Given a property $P$ that morphisms of $\rC$ can possess, we write $\Arr^P(\rC) \subseteq \Arr(\rC)$ for the respective full subcategory of the arrow category. Similarly, we write $\rC_{/^{P} c} \subseteq \rC_{/ c}, \rC_{c/^{P}} \subseteq \rC_{c/}$ for the respective \emph{full} subcategories on arrows satisfying $P$.
	\item A category $\rC$ is called \emph{weakly contractible} if its geometric realization $|\rC| \in \Spaces$, i.e.\ its localization at all of its morphisms, is isomorphic to the point $*$. A functor $f: \rC \to \rD$ is called \emph{colimit (limit) cofinal} if precomposing with it preserves colimits (limits).
	\item Given a class of colimit shapes $\cK \subseteq \Cat$, a functor $F: \rC \to \rD$ \emph{reflects $\cK$-shaped colimits} if for any cone $\bar{p} : \rK^\triangleright \to \rC$ with $\rK \in \cK$ such that $F \circ \bar{p}$ is a colimit cone in $\rD$, already $\bar{p}$ was a colimit cone in $\rC$. Further $F$ \emph{creates $\cK$-shaped colimits} if it both preserves and reflects them. Similarly for limits.
	\item We write $F: \rC \rightleftarrows \rD : G$ or $F \dashv G$ for functors $F: \rC \to \rD$ and $G: \rD \to \rC$ such that $F$ is left adjoint to $G$.
	\item For $\rC$ a small category denote by $\PSh(\rC) := \Fun(\rC^{\op}, \Spaces)$ its presheaf category, which is the recipient of the fully faithful Yoneda embedding $\yo : \rC \hookrightarrow \PSh(\rC)$ by \HTT{Prop.}{5.1.3.1}, as well as the free cocompletion of $\rC$ under small colimits by \HTT{Thm.}{5.1.5.6}. Similarly for $\cM$ a large category, we write $\lPSh(\cM) := \Fun(\cM^{\op}, \hat{\Spaces})$ which is its free cocompletion under large colimits, and contains the free cocompletion $\PSh^{\tav\text{-rex}}(\cM)$ of $\cM$ under small colimits as its full subcategory on $\tav$-compact objects. 
	Further, $\lPSh(\cM)$ contains the Ind-$\tav$-completion $\lInd(\cM)$, as the full subcategory spanned by those presheaves $\cM^{\op} \to \lSpaces$ that preserve small limits.
	\item We denote the large category of presentable categories and colimit-preserving functors by $\Pr$, and the tensor product of presentable categories $\cM, \cN \in \Pr$ by $\cM \otimes \cN$. 
	\item We let $\mathrm{Fin}$ be the category of finite sets, whose objects we denote as $\underline{n} := \{1, \dots, n\}$. Similarly $\Fin := \mathrm{Fin}_{\underline{1} /}$ is the category of finite pointed sets whose objects we denote $\underline{n}_+ := \{*, 1, \dots, n\}$; and $\mathrm{Fin}^\simeq \hookrightarrow \mathrm{Fin}$ is the wide subcategory of finite sets and bijections.
	\item By a (Lurie-)operad $\rO$ we mean a small symmetric colored $\infty$-operad, unless explicitly stated otherwise. We denote its \emph{category of operations}, the associated fibration over $\Fin$, by $\rO^\otimes \to \Fin$. 
	Also we refer to fiber over $\underline{1}_+ \in \Fin$ as its \emph{underlying category} $\underline{\rO} := \rO^\otimes_{\underline{1}_+}$, and refer to the maximal subspace $\col \rO := \underline{\rO}^\simeq$ as its \emph{space of colors}. We call the morphisms covering the terminal map $\underline{1}_+ \to \underline{1}_+$ in $\Fin$ the \emph{$n$-ary multimorphisms} or \emph{$n$-ary operations} or $\rO$. Write $\Op$ for the category of operads. 
	\item We refer to symmetric monoidal categories by their underlying category which we usually call $\rV$, write $V^\otimes \to \Fin$ for the associated category of operations and $\underline{\rV}: \Fin \to \Cat$ for the associated commutative monoid or Segal object in $\Cat$. 
	\item We refer to $\cV \in \Alg(\Pr)$ as a \emph{presentably monoidal category}; explicitly this unwinds to a monoidal category that is presentable and (left and right) closed. Similarly, $\cM \in \RMod_\cV (\Pr) =: \PrV$ will be called a \emph{presentable (right) $\cV$-module category}; i.e.\ a presentable category equipped with a right $\cV$-action that admits both an internal hom $\iHom_\cM(-, -) : \cM^{\op} \times \cM \to \cV$ and a cotensoring $(-)^{(-)}: \cV^{\op} \times \cM \to \cM$.
\end{itemize}

\pagenumbering{arabic}


\section{Enriched Operads}
\label{sec:enrichedoperads}

For $\cV \in \CAlg(\Pr)$ a presentably symmetric monoidal category, in order to define $\cV$-enriched operads we make use of the $2$-category $\CAlg(\IPrV)$ whose
\begin{itemize}
	\item underlying $1$-category is the category $\CAlg(\PrV) := \CAlg(\RMod_\cV(\Pr))$ of presentably symmetric monoidal $\cV$-module categories,
	\item morphism categories $\Fun^{\rL, \otimes}_\cV(\cM, \cN)$ between $\cM, \cN \in \CAlg(\PrV)$ consist of symmetric monoidal $\cV$-linear colimit-preserving functors.
\end{itemize}
We construct and study this $2$-category in \cref{sec:appstar}, establishing the following properties:
\begin{itemize}
	\item By \cref{cor:freeforgetfulfunctors} the forgetful $2$-functors
	\[
	\CAlg(\IPrV) \to \CAlg(\IPr) \to \CAlg(\widehat{\ICat}) \to \widehat{\ICat}
	\]
	admit (partially) left adjoint $2$-functors 
	\[
	\ICat \overset{\Sym}{\longrightarrow} \CAlg(\ICat) \overset{\PSh}{\longrightarrow} \CAlg(\IPr) \overset{- \otimes \cV}{\longrightarrow} \CAlg(\IPrV)
	\]
	by which we mean in particular that for a given $\rC \in \Cat$ and $\cM\in \CAlg(\PrV)$, precomposing with the canonical functor $\rC \subseteq \Sym(\rC) \subseteq \PSh \Sym(\rC) \to \PSh\Sym (\rC) \otimes \cV$ induces an equivalence
	\[ \Fun^{\rL, \otimes}_\cV \left( \PSh(\Sym \rC) \otimes \cV, \cM \right) \simeq \Fun(\rC, \cM) \; . \]
	Since $\PSh, \Sym$ and $- \otimes \cV$ merely enhance the respective $1$-functors, we calculate
	\[ \PSh(\Sym \rC) \otimes \cV \simeq \Fun(\Sym \rC ^{\op}, \cV) \simeq \prod_{n \geq 0} \Fun\left( (\rC^\op)^{\times n}_{h \Sigma_n}, \cV \right) \; . \]
	\item For $\cM, \cN, \cP \in \CAlg(\PrV)$, the composition functor \[ \circ : \Fun^{\rL, \otimes}_\cV(\cM, \cN) \times \Fun^{\rL, \otimes}_\cV (\cN, \cP) \to \Fun^{\rL, \otimes}_\cV (\cM, \cP) \] preserves sifted colimits in the left argument, and all colimits in the right argument by \cref{prop:CAlgPrVclosed}.
	\item By \cref{thm:EMPrV} the $2$-category $\CAlg(\IPrV)$ admits 
	Eilenberg-Moore-objects $\LMod_T(\cM)$ for any monad $T \in \Alg(\Endo^{\otimes, \rL}_\cV(\cM))$, which are created by all of the above forgetful $2$-functors and hence calculated in $\IlCat$. Also they agree with the respective Kleisli objects, meaning that we have the following equivalences:
	\begin{align*}
		\Fun^{\rL, \otimes}_\cV(\cN, \LMod_T(\cM)) &\simeq \LMod_T(\Fun^{\rL, \otimes}_\cV(\cN, \cM)) \\  \Fun^{\rL, \otimes}_\cV(\LMod_T(\cM), \cN) &\simeq \RMod_T(\Fun^{\rL, \otimes}_\cV(\cM, \cN))
	\end{align*}
\end{itemize}

\begin{defin}
	\label{def:vOpX}
	Given a space $X \in \Spaces$, a \emph{$\cV$-enriched operad with colors $X$} is a monad on $\PSh(\Sym X) \otimes \cV \simeq \Fun ( \Sym X^{\op} , \cV) $ in the $2$-category $\CAlg (\RMod_\cV (\IPr))$. Using the above adjunction 
	and that $X^\op \simeq X$, this unwinds to an algebra in
	\[
	\Endo^{\rL, \otimes}_\cV ( \PSh ( \Sym X ) \otimes \cV ) \simeq 
	\Fun( X, \Fun( \Sym X^{\op} ,  \cV ) ) \simeq \Fun \left( \Sym X \times X, \cV \right) \, .
	\]
	We refer to $\Fun \left( \Sym X \times X, \cV \right)$, equipped with the composition monoidal structure $\oeert$ it obtains as an endomorphism object, as the category of \emph{$X$-colored symmetric sequences} in $\cV$.
	
	Given a $\cV$-enriched operad $\cO$, we call its underlying functor $\Mul_\cO : \Sym X \times X \to \cV$ the \emph{multigraph} of $\cO$, and denote its evaluation on a tuple $((x_1, \dots, x_n), x)$ by $\Mul_\cO(x_1, \dots, x_n ; x) \in \cV$. Its points $1_\cV \to \Mul_\cO(x_1, \dots, x_n ; x)$ are called \emph{$n$-ary morphisms} in $\cO$.
\end{defin}

\begin{warning}
		From the explicit formula in \cref{cor:compositionproduct} we will see that the monoidal structure $\oeert$ on $\Fun \left( \Sym X \times X, \cV \right)$ is reverse to what is commonly called the composition product on symmetric sequences
		, which we denote $\otree$. Their categories of algebras $\vOp_X(\cV)$ are of course equivalent by passing to the reverse algebra.
\end{warning}

\begin{notat}
	We write $\underline{x}$ for a symmetric tuple $(x_1, \dots, x_n) \in \Sym X$. If $X = *$ we say $\cO\in \Fun \left( \Sym (*) \times *, \cV \right) \simeq \Fun(\bigsqcup_{k \geq 0} B \Sigma_k, \cV)$ is \emph{single-colored}, and write $\cO(n)$ for $\Mul_\cO( * , \dots, *; *) \in \cV$ evaluating on the unique $n$-tuple.
\end{notat}

\begin{lemma}
	\label{lem:dayconvolution}
	Let $X \in \Cat$ and $W_1, \dots, W_k \in \PSh \Sym X \otimes \cV$. Then, their Day convolution product can be calculated as
	\[
		W_1 \Day \dots \Day W_k  (x_1, \dots, x_n) \simeq \bigsqcup_{I_1 \sqcup \dots \sqcup I_k = \underline{n}} W_1(\underline{x}|_{I_1}) \otimes \dots \otimes W_k(\underline{x}|_{I_k})
	\]
	where the second coproduct is indexed by the set of partitions of $\underline{n} = \{1, \dots , n\}$ into $k$ disjoint subsets $I_1, \dots, I_k$, and $\underline{x}|_{I_1}$ is the symmetric tuple consisting of those $x_j$ with $j \in I_1$.
\end{lemma}
\begin{proof}
	By \HA{Ex.}{2.2.6.17} the Day convolution product can be calculated as the colimit
	\[
	\colim_{\underline{{x}}^{(1)} \otimes \dots \otimes \underline{{x}}^{(k)} \to \underline{x}} W_1(\underline{{x}}^{(1)}) \otimes \dots \otimes W_k(\underline{{x}}^{(k)})
	\]
	over the indexing category $((\Sym X)^{\times k}) \times_{\Sym X} (\Sym X)_{/\underline{x}}$. Now, the terminal map $X \to *$ induces a forgetful functor $\Sym X \to \Sym(*) \simeq \mathrm{Fin}^{\simeq}$ into the category $\mathrm{Fin}^{\simeq}$ of finite sets and bijections, so we obtain the following functor on slice categories:
	\[
	\Phi: ((\Sym X)^{\times k}) \times_{\Sym X} (\Sym X)_{/\underline{x}} \to (\mathrm{Fin}^{\simeq, \times k}) \times_{\mathrm{Fin}^{\simeq}} (\mathrm{Fin}^{\simeq})_{/\underline{x}} \; .
	\]
	The right side, which a priori is a $(1,1)$-category, is in fact the discrete set of partitions of $\underline{n}$. Further, the fiber of $\Phi$ over a fixed partition $I_1 \sqcup \dots \sqcup I_k = \underline{n}$ is the product
	\[
		\prod_{i=1}^{k} (\Sym X)_{/\underline{x}^{(k)}}
	\]
	since $(\Sym X)^\otimes \to \Fin$ is a coCartesian fibration, as is its pullback to $\mathrm{Fin}^{\simeq}$, so we may apply \HTT{Def.}{2.4.1.1}. 
	This product admits a terminal object since each factor has terminal object $\id_{\underline{x}^{(k)}}$. Hence, $\Phi$ admits a right adjoint $\Phi^\rR$ sending each partition to the respective terminal object in the fiber, and since right adjoints are colimit-cofinal the colimit in question agrees with the coproduct over partitions over the evaluations at the respective terminal objects.
\end{proof}

\begin{lemma}
	\label{lem:inducedfreefunctor}
	Let $X \in \Cat$ and $\cM \in \CAlg(\PrV)$. Under the equivalence $\Fun(X, \cM) \simeq \Fun^{\rL, \otimes}_\cV (\PSh \Sym X \otimes \cV, \cM)$ from \cref{cor:freeforgetfulfunctors}, a functor $A: X \to \cM$ is sent to the morphism $\bar{A}: \Fun(\Sym X^{\op} , \cV) \to \cM$ in $\CAlg(\PrV)$ determined by the coend
	\[
	\bar{A} : (W: \Sym X^{\op} \to \cV) \mapsto \oint^{\underline{x} \in \Sym X} A(x_1) \otimes \dots \otimes A(x_n) \otimes W(\underline{x}) \; .
	\]
\end{lemma}
\begin{proof}
	As a first step, $A$ extends to a unique symmetric monoidal functor $\Sym X \to \cM$ sending $(x_1, \dots, x_n) \mapsto A(x_1) \otimes \dots \otimes A(x_n)$. From here, the formula for the extension to $\PSh(\Sym X) \otimes \cV$ follows from \marked{Lem.}{2.13}.
\end{proof}

\begin{cor}
	\label{cor:compositionproduct}
	The monoidal structure $\oeert$ on $\Fun \left( \Sym X \times X, \cV \right)$ sends $A, B : \Sym X \times X \to \cV$ to
	\[
	A \oeert B ( \underline{x}, z ) = \bigsqcup_{k \geq 0} \bigsqcup_{I_1 \sqcup \dots \sqcup I_k = \underline{n}} \oint^{\underline{y} \in \Sym^k X} A(\underline{{x}}^{(1)} ; y_1) \otimes \dots \otimes A(\underline{{x}}^{(k)}; y_k) \otimes B(\underline{y}; z) \; .
	\]
	If additionally $X$ is a space, this agrees with
	\[
	\bigsqcup_{k \geq 0} \bigsqcup_{I_1 \sqcup \dots \sqcup I_k = \underline{n}} \colim_{(y_1, \dots, y_k) \in X^{\times k}} \left( A(\underline{x}|_{I_1}; y_1) \otimes \dots \otimes A( \underline{x}|_{I_k}; y_k) \right) \otimes_{\Sigma_k} B(y_1 , \dots , y_k; z)
	\]
	where the $\Sigma_k$-action permutes the blocks $\underline{x}|_{I_1}, \dots, \underline{x}|_{I_k}$ partitioning $\underline{x}$ and the entries of $\underline{y}$. Specifically for $X = *$ we obtain
	\[
	\bigsqcup_{k \geq 0} \bigsqcup_{I_1 \sqcup \dots \sqcup I_k = \underline{n}} (A( I_1) \otimes \dots \otimes A( I_k)) \otimes_{\Sigma_k} B(k) \; . 
	\]
\end{cor}
\begin{proof}
	If we denote by $\bar{A}, \bar{B}$ the endofunctors of $\PSh(\Sym X) \otimes \cV$ associated to $A$ and $B$, then by definition we can write $A \oeert B (  (x_i) ; z ) = \bar{A} \circ \bar{B}(\yo_z \otimes 1_\cV) (x_1, \dots, x_n) = \bar{A}(B(- ; z))(x_1, \dots, x_n)$ so the first expression follows from \cref{lem:inducedfreefunctor}, noting that $\mathrm{Tw}(\Sym X) \simeq \bigsqcup_{k \geq 0} \mathrm{Tw} (\Sym^k X)$ to expand the coend, and the formula for the Day convolution product in \cref{lem:dayconvolution}. 
	
	By \kerodon{048H} and \kerodon{048L}, for a space $X$ we have $\mathrm{Tw}(X) \simeq X$ so coends agree with colimits over $X$. Further, the functor $B \Sigma_k \to \Spaces$ encoding the $\Sigma_k$-action on $X^{\times k}$ unstraightens into a coCartesian fibration $X^{\times k}_{h \Sigma_k} \to B \Sigma_k$ with fiber $X^{\times k}$. Hence, $\colim_{X^{\times k}_{h \Sigma_k}} \simeq \colim_{B\Sigma_k} \colim_{X^{\times k}}$, and the colimit over $B \Sigma_k$ turns the tensor product into a relative tensor product $\otimes_{\Sigma_k}$.	
\end{proof}

\begin{reminder}
	In \marked{Def.}{2.30}, we have defined a \emph{valent $\cV$-enriched category} with space of objectss $X$ in a similar manner, namely as a monad on $\PSh(X) \otimes \cV$ in $\IPrV$. We denote their category by
	\[ \vCatXV := \Alg(\Endo^\rL_\cV(\PSh(X) \otimes \cV)) \simeq \Alg(\Fun(X^{\op} \times X, \cV)) \; . \]
\end{reminder}

\begin{constr}
	\label{constr:modulesvsalgebras}
	By \cref{prop:freeforgetfulfunctors} the symmetric algebra functor 
	\[ \Sym^{\otimes_\cV}: \IPrV \to \CAlg(\IPrV) \]
	is a $2$-functor, so it induces a monoidal functor between endomorphism objects
	\[
	\Fun(X \times X , \cV) \simeq \LinEnd(\PSh(X) \otimes \cV) \rightarrow \Endo^{\rL, \otimes}_\cV(\PSh \Sym X \otimes \cV) \simeq \Fun(\Sym X \times X, \cV)
	\]
	explicitly acting by left Kan extension along the full inclusion of the summand $X \subseteq \Sym X$, i.e.\ extension by the initial object $\emptyset_\cV$. In particular the latter is fully faithful and preserves colimits, so it admits a lax monoidal right adjoint. We obtain an adjunction between algebra objects
	\[
	\mathrm{Triv}_\cV : \vCatXV \rightleftarrows \vOp_X(\cV) : \ColV \; ,
	\]
	where the fully faithful left adjoint $\mathrm{Triv}_\cV$ sends a $\cV$-enriched category $\cC$ to the \emph{trivial $\cV$-operad} with colors $\cC$ and multimorphisms
	\[
	\Mul_{\mathrm{Triv}_\cV(\cC)}(c_1, \dots, c_n ; c) = \left\lbrace \begin{matrix}\Hom_\cC(c_1, c) \text{ if } n=1 \\ \emptyset_\cV \text{ otherwise}\end{matrix} \right. \; .
	\]
	The right adjoint $\ColV$ sends a $\cV$-operad to its \emph{$\cV$-category of colors} with $\Hom_{\ColV(\cO)}(o_1, o_2) \simeq \Mul_\cO(o_1; o_2)$.
\end{constr}

\begin{constr}
	\label{constr:env}
	Similarly, the forgetful $2$-functor $\CAlg(\IPrV) \to \IPrV$ induces on endomorphism categories a monoidal functor
	\[
	\Fun(\Sym X \times X, \cV) \simeq \Endo^{\rL, \otimes}_\cV(\PSh \Sym X \otimes \cV) \to \Endo^{\rL}_\cV(\PSh \Sym X \otimes \cV) \simeq \Fun(\Sym X \times \Sym X, \cV)
	\]
	explicitly given by extending symmetric monoidally in the first argument. On algebra objects we obtain a functor
	\[
	\vEnv_\cV : \vOp_X(\cV) \to \vCat_{\Sym X}(\cV) \; 
	\]
	which we call the \emph{(valent) $\cV$-enriched envelope}. Explicitly,
	\[
	\Hom_{\vEnv_\cV(\cO)}((o_1, \dots, o_n), (p_1, \dots , p_m)) \simeq \bigotimes_{i=1}^{m} \Mul_{\cO}(o_1, \dots, o_n; p_i) \; .
	\]
	We will see in \cref{defin:envelopefun} that this enhances to a functor $\vEnv_\cV: \vOp(\cV) \to \CAlg(\vCatV)$.
\end{constr}

\begin{constr}
	\label{constr:coespace}
	A morphism $f: \cV \to \cW$ in $\CAlg(\PrV)$ induces by \cref{prop:eosstar} a $2$-functor $f_{\mathrm{ext}} : \CAlg(\IPrV) \to \CAlg(\IPr_\cW)$. As in \cref{constr:modulesvsalgebras}, we obtain a \emph{change-of-enrichment} adjunction
	\[
	f_! : \vOp_X(\cV) \rightleftarrows \vOp_X(\cW) : f_!^{\rR}
	\]
	that acts on multigraphs by postcomposing $\Mul_\cO : \Sym X \times X \to \cV$ with $f$, or postcomposing with $f^\rR$ in the other direction.
\end{constr}

\begin{defin}
	\label{defin:operadicpsh}
	Given a $\cV$-operad $\cO \in \vOp_X(\cV)$, we define its \emph{operadic (enriched) presheaf category} as the Eilenberg-Moore object
	\[ \PShO(\cO) := \LMod_\cO(\PSh \Sym (X) \otimes \cV) \in  \CAlg(\IPr) \; .\]
	It is part of an Eilenberg-Moore adjunction
	\[
	\PSh \Sym (X) \otimes \cV \rightleftarrows \PShO(\cO)
	\]
	between free and forgetful functor. We define the \emph{operadic Yoneda functor} $\yo_\cO : X \to \PSh \Sym (X) \otimes \cV \to \PShO(\cO)$ of $\cO$ by composing the unit of the free-forgetful adjunction $\lCat \rightleftarrows \CAlg(\RMod_\cV(\Catcolim))$ with the free module functor.
\end{defin}
\begin{rem}
	What we call $\PShO(\cO)$ is usually referred to as the category of \emph{right modules} over $\cO$, see e.g.\ \cite{heinemilnor}. The unfortunate change in direction comes from the fact that our composition product is reverse to the classical convention.
\end{rem}

\begin{prop}[Operadic Yoneda Lemma]
	\label{prop:operadicyoneda}
	For $\cO\in \vOp_X(\cV)$ and colors $x_1, \dots, x_n, x \in X$, there is a canonical equivalence
	\[   \iHom_{\PShO(\cO)}(\yo(x_1) \otimes \dots \otimes \yo(x_n), \yo(x)) \simeq  \Mul_\cO(x_1, \dots, x_n; x)  \; .\]
\end{prop}
\begin{proof}
	We obtain the monad $\cO \in \Alg(\Endo^{\rL, \otimes}_\cV(\PSh \Sym X \otimes \cV))$ as the composition of the forgetful and free functor in the Eilenberg-Moore-adjunction $Y: \PSh \Sym (X) \otimes \cV \rightleftarrows \PShO(\cO) : Y^\rR$. To obtain the associated $\Mul_\cO \in \Fun(\Sym X \times X, \cV)$, we must precompose $Y$ with $X \to \PSh \Sym X \otimes \cV$ obtaining $\yo$. Thus we can write
	\[
	\Mul_\cO(x_1, \dots, x_n; x) \simeq \mathrm{ev}_{(x_1, \dots, x_n)} Y^\rR  \yo(x) \; .
	\]
	But the evaluation $\mathrm{ev}_{(x_1, \dots, x_n)}$ is right adjoint to the unique functor $\cV \to \PSh \Sym X \otimes 1_\cV$ sending $1_\cV \mapsto \yo_{(x_1, \dots, x_n)}$, to its composition with $Y^\rR$ is right adjoint to the unique functor $\cV \to \PShO(\cO)$ sending $1_\cV \mapsto Y(\yo_{(x_1, \dots, x_n)})$, and thereby agrees with $\iHom_{\PShO(\cO)}(Y(\yo_{(x_1, \dots, x_n)}), -)$. We conclude since $Y$ is symmetric monoidal.
\end{proof}

\section{\texorpdfstring{$\otimes$}{Tensor}-atomic markings}
\label{sec:tensoratomics}

Similarly to atomic objects and atomically generated categories in the world of $\IPrV$ as studied in \cite{ramzi2024dualizable} or \cite[§3]{marked}, we introduce $\otimes$-atomic objects and $\otimes$-atomically generated categories in $\CAlg(\IPrV)$. Surprisingly, many statements are no longer true in the symmetric monoidal world or require further assumptions.

\subsection{Internal left adjoints} 

We first study left adjoint $1$-morphisms internally to the $2$-category $\CAlg(\IPrV)$, allowing us to define $\otimes$-atomic objects and markings.

\begin{prop}
	A $1$-morphism $F: \cM \to \cN$ in the $2$-category $\CAlg(\IPrV)$ is internally left adjoint if its right adjoint functor $F^\rR$ preserves colimits, and the canonical lax symmetric monoidal lax $\cV$-linear structure on $F^\rR$ is strong symmetric monoidal and strong $\cV$-linear.
\end{prop}

\begin{proof}
	\cref{obs:CAlgPrVemb} exhibits $\CAlg(\IPrV)$ as a sub-$2$-category of $\Fun(\Fin, \RMod_\cV(\ICat))$, so adjunctions can be detected in the latter. Adjunctions in a functor $2$-category are given by pointwise adjunctions satisfying a Beck-Chevalley condition, see for instance \cite[Cor. 5.15]{heinemonadicity}. In our case, for any morphism $\alpha : \underline{n}_+ \to \underline{m}_+$ in $\Fin$ the diagram
	\[
	\begin{tikzcd}
		\cM^{\times n} \arrow[r, "F^{\times n}"] \arrow[d, "\alpha_!"] & \cN^{\times n} \arrow[d, "\alpha_!"] \\
		\cM^{\times m} \arrow[r, "F^{\times n}"] & \cN^{\times m}
	\end{tikzcd}
	\]
	must be horizontally right adjointable, yielding the strong symmetric monoidality condition for $F^\rR$. It suffices to check the pointwise adjointness condition in $\PrV$ on the component $F^{\times 1} : \cM^{\times 1} \to \cN^{\times 1}$, where it unwinds into strong $\cV$-linearity and cocontinuity of $F^\rR$ by a similar argument. Compare \cite[\S~4]{ben2024naturality} for this.
\end{proof}

\begin{defin}
	For $\cM\in \CAlg(\PrV)$, an object $m \in \cM$ is called \emph{$\otimes$-atomic} if the unique map $\Fun(\bigsqcup_{k \geq 0} B \Sigma_k, \cV) \simeq \PSh \Sym(*) \otimes \cV \to \cM$ sending $1_\cV \mapsto m$ is internally left adjoint. In other words, the \emph{internal multi-hom} functor
	\[
	\iHom^\otimes_{\cM}(m , -) : \cM \to \prod_{k \geq 0} \Fun\left(B \Sigma_k, \cV\right)
	\]
	right adjoint to the above map 
	preserves colimits, and its canonical lax symmetric monoidal lax $\cV$-linear structure is strong. Denote by $\mathcal{M}^{\otimes\mathrm{at}} \subseteq \mathcal{M}$ the full subcategory on the $\otimes$-atomic objects.
	
	Given a small category $\rC$, a functor $\rC \to \cM$ is called a \emph{$\otimes$-atomic marking} if its unique extension $\PSh \Sym \rC \to \cM$ in $\CAlg(\PrV)$ is internally left adjoint. Specifically a sequence of objects $(m_1, \dots, m_n) \in \cM$ is called \emph{$\otimes$-atomic tuple} if the functor $\{1, \dots, n\} \to \cM$ marking them is a $\otimes$-atomic marking.
\end{defin}

\begin{rem}
	This naming is chosen to resemble the case of $\cM \in \IPrV$, where an object $m \in \cM$ is called \emph{$\cV$-atomic} or simply \emph{atomic} if the $1$-morphism $\PSh(*) \otimes \cV \simeq \cV \to \cM$ in $\IPrV$ sending $* \mapsto m$ is internally left adjoint. Compare \cite{ben2024naturality}, \cite{ramzi2024dualizable}, \marked{Def.}{3.20}.
\end{rem}

\begin{obs}
	\label{obs:iLpresatomic}
	Internally left adjoint functors in $\CAlg(\PrV)$ are closed under composition; in particular they preserve $\otimes$-atomic markings and $\otimes$-atomic objects. We will prove partial converses in \cref{lem:iLcdomdiagram} and \cref{prop:iLonsequ}.
\end{obs}

\begin{obs}
	\label{obs:tensoratomic}
	Let us work out explicitly what it means for an object $m \in \cM$ to be $\otimes$-atomic in $\cM \in \CAlg(\PrV)$. The induced functor
	\[
	\PSh \Sym(*) \otimes \cV \simeq \prod_{k \geq 0} \Fun(B \Sigma_k , \cV) \to \cM
	\]
	sends a tuple $( v_0, v_1, v_2, \dots)$ of objects $v_k \in \cV$ with $\Sigma_k$-action to
	\[
	\bigsqcup_{k \geq 0} m^{\otimes k} \otimes_{\Sigma_k} v_k \in \cM \;
	\]
	where we set $m^{\otimes 0} := 1_\cV$, and $\Sigma_k$ acts on $m^{\otimes k}$ by permuting the factors.  
	So its right adjoint $\iHom_\cM^{\otimes}(m ,-): \cM \to \prod_{k \geq 0} \Fun(B \Sigma_k , \cV)$ sends $m' \in \cM$ to the tuple
	\begin{align*}
		\left( \iHom_\cM( m^{\otimes k} , m' ) \in \Fun(B \Sigma_k, \cV)  \right)_{k \geq 0}
	\end{align*}
	with $\Sigma_k$-action permuting the factors of $m$.
	This means that $m \in \cM$ is $\otimes$-atomic iff for any $n \in \N_0$,
	\begin{itemize}
		\item the functor $\iHom_\cM( m^{\otimes n} , -)$ preserves colimits and $\cV$-tensoring; in other words $m^{\otimes n}$ is $\cV$-atomic,
		\item The unitor induces isomorphisms $\emptyset_\cV \simeq \iHom_\cM(m^{\otimes n}, 1_\cM)$ unless $n=0$, where $1_\cV \simeq \iHom_\cM(1_\cM, 1_\cM)$,
		\item For any $m_1, m_2 \in \cM$, the multiplication map 
		\[
		\bigsqcup_{\underline{n} = S \sqcup T} \iHom_\cM(m^{\otimes  S }, m_1) \otimes \iHom_\cM(m^{\otimes  T }, m_2) \overset{\otimes}{\to} \iHom_\cM(m^{\otimes n}, m_1 \otimes m_2)
		\]
		is an isomorphism.
	\end{itemize}
\end{obs}

\begin{obs}
	\label{obs:hereditary}
	A similar discussion shows that a marking $y: \rC \to \cM$ is $\otimes$-atomic if for any finite sequence of objects $c_0, \dots, c_n \in \rC$ where $n \geq 0$,
	\begin{itemize}
		\item The product $y (x_1) \otimes \dots \otimes y (x_n)$ is an atomic object in $\cM$,
		\item The unitor induces isomorphisms $\emptyset_\cV \simeq \iHom_\cM(y(x_1) \otimes \dots \otimes y(x_n), 1_\cM)$ unless $n=0$, in which case $1_\cV \simeq \iHom_\cM(1_\cM, 1_\cM)$,
		\item For any $m_1, m_2 \in \cM$, the multiplication map
		\[
			\bigsqcup_{\underline{n} = S \sqcup T} \iHom_\cM \left( \bigotimes_{s \in S} y (x_s), m_1 \right) \otimes \iHom_\cM \left( \bigotimes_{t \in T} y (x_t), m_2 \right) \overset{\otimes}{\to} \iHom_\cM \left( \bigotimes_{k=1}^n y (x_k) , m_1 \otimes m_2 \right)
		\]
		is an isomorphism. This is known as the \emph{hereditary condition}, c.f.\ \cite[Prop. 2.4.6, Rem. 2.4.7]{haugsengenv}.
	\end{itemize}
\end{obs}


\begin{obs}
	We learn that for any $\otimes$-atomic objects or $\otimes$-atomic marking to exist in $\cM \in \CAlg(\PrV)$, the unit $1_\cM \in \cM$ must be $\cV$-atomic and the lax symmetric monoidal structure on the forgetful functor $\iHom_\cM(1_\cM, -): \cM \to \cV$ must be strong.
\end{obs}

\begin{rem}
	\label{rem:indecomp}
	Let $\cM \in \CAlg(\Pr)$ so $\cV = \Spaces$, and $m_1, m_2 \in \cM$. If the product $m_1 \otimes m_2$ is $\otimes$-atomic, then in particular the multiplication map must exhibit
	\[
	\left( \Map_\cM(m_1 \otimes m_2, m_1) \times \Map_\cM(1, m_2) \right) \sqcup \left( \Map_\cM(1, m_1) \times \Map_\cM(m_1 \otimes m_2, m_2) \right)
	\]
	as isomorphic to $\Map_\cM(m_1 \otimes m_2, m_1 \otimes m_2)$, in particular the identity $\id_{m_1 \otimes m_2}$ factors through one of the summands. This means that there exist $i: m_1 \otimes m_2 \to m_1$ and $r: 1 \to m_2$ such that
	\[
	\id_{m_1 \otimes m_2} \simeq r \otimes i \simeq (m_1 \otimes m_2 \overset{i}{\to} m_1 \simeq m_1 \otimes 1 \overset{m_1 \otimes r}{\to} m_1 \otimes m_2)
	\]
	exhibiting $m_1 \otimes m_2$ as a retract of $m_1$, or the other way around\footnote{Conversely retracts of $\otimes$-atomic objects are always $\otimes$-atomic, since colimit-preserving $\cV$-linear symmetric monoidal functors are closed under retracts in lax $\cV$-linear lax symmetric monoidal functors.}. 
	Informally, it is rare for the tensor product of $\otimes$-atomic objects to remain $\otimes$-atomic; they should be imagined as singletts or indecomposables. 
	
	If we instead choose $\cV = \Ab$, then $\id_{m_1 \otimes m_2}$ must be in the image of
	\[ \left( \Map_\cM(m_1 \otimes m_2, m_1) \otimes \Map_\cM(1, m_2) \right) \oplus \left( \Map_\cM(1, m_1) \otimes \Map_\cM(m_1 \otimes m_2, m_2) \right) \]
	under the multiplication map, meaning that it can be written as a direct sum whose summands are retracts of either $m_1$ or $m_2$. In future work \cite{cauchyoperads}, we will see that for general enrichment, $\otimes$-atomic objects behave like singletts up to \emph{absolute colimits}.
\end{rem}

\begin{lemma}
	\label{lem:PSymisiL}
	Given any functor $\rC \to \rD$, the induced map $\PSh \Sym (\rC) \otimes \cV \to \PSh \Sym (\rD) \otimes \cV$ in $\CAlg(\IPrV)$ is internally left adjoint.
\end{lemma}
\begin{proof}
	We apply $\Sym^{\otimes}$ and $- \otimes \cV$, which are $2$-functors by \cref{prop:freeforgetfulfunctors}, to $\PSh (\rC) \to \PSh (\rD)$ which is internally left adjoint in $\IPr$.
\end{proof}

\begin{obs}
	\label{obs:otimesatomicmarkingobj}
	If $y: \rC \to \cM$ is a $\otimes$-atomic marking, then for any $c \in \rC$ the object $y(c) \in \cM$ is $\otimes$-atomic, since the functor $\PSh \Sym(*) \otimes \cV \to \PSh \Sym(\rC) \otimes \cV \to \cM$ marking $y(c)$ is a composition of internally left adjoint functors by \cref{lem:PSymisiL}.
\end{obs}

\begin{warning}
	\label{warning:atomicpair}
	The converse to \cref{obs:otimesatomicmarkingobj} is wrong, already for $\rC = * \sqcup *$, i.e.\ $\otimes$-atomic pairs. The product in $\CAlg(\Pr)$ is created by the forgetful functor to $\Pr$, so it agrees with the pointwise symmetric monoidal structure induced by the Cartesian (which agrees with the coCartesian) symmetric monoidal structure on $\Pr$. In particular given $F: \cM_1 \to \cN_1$ and $G: \cM_2 \to \cN_2$ in $\CAlg(\Pr)$, the product $F \times G: \cM_1 \times \cM_2 \to \cN_1 \times \cN_2$ enhances to a morphism in $\CAlg(\Pr)$; and if $F$ and $G$ are internally left adjoint then so is $F \times G$ with right adjoint $F^\rR \times G^\rR$. 
	
	Consider $\PSh \Sym(*)^{\times 2} = \Fun(* \sqcup *, \PSh\Sym(*)) \in \CAlg(\Pr)$. Since the pointwise symmetric monoidal structure is functorial in the inclusions $* \to * \sqcup *$, both inclusions into and projections out of this biproduct in $\Pr$ are maps in $\CAlg(\Pr)$. The objects $(*, 1), (1, *) \in \PSh \Sym(*)^{\times 2}$ are both $\otimes$-atomic: The functor $\PSh\Sym(*) \to \PSh \Sym(*)^{\times 2}$ marking $(*,1)$ agrees with the left inclusion $i_1$, which is internally left adjoint to the left projection $\mathrm{pr}_1$, and similarly for $(1, *)$.
	However they do not form a $\otimes$-atomic pair, since otherwise their product $(*, *) \in \PSh \Sym (*)^{\times 2} \simeq \PSh(\Sym (*) \sqcup \Sym (*))$ would be $\Spaces$-atomic by \cref{obs:tensoratomic}. This is impossible since $\Sym (*) \sqcup \Sym (*)$ is a groupoid and thereby idempotent complete, so $(*,*)$ would have to be a representable presheaf which it is not.
\end{warning}

\begin{rem}
	We will however see in \cref{prop:iLonsequ} that to check that $y: \rC \to \cM$ is a $\otimes$-atomic marking, it is sufficient to verify that the restriction $\{c_1, \dots, c_n\} \to \cM$ is a $\otimes$-atomic marking for any finite set of objects $c_1, \dots, c_n$ in $\rC$. 
	
	Also in \cref{def:fissile} we introduce \emph{fissile categories}, and prove in \cref{prop:iLfissile} that if $\cM$ is fissile, the converse of \cref{obs:otimesatomicmarkingobj} \emph{is} in fact true, i.e.\ $y: \rC \to \cM$ is a $\otimes$-atomic marking iff it factors through the $\otimes$-atomic objects.
\end{rem}

\begin{obs}
	\label{obs:relationinternalhoms}
	By definition of the internal multihom, we have 
	\[ \Map_{\Fun(B \Sigma, \cV)}(1_{\Fun(B \Sigma, \cV)}, \iHom^\otimes_{\cM}(m , -)) \simeq \Map_\cM(m \oeert 1_{\Fun(B \Sigma, \cV)} , -) \simeq \Map_\cM(m , -) \; .\]
	This enhances to $\iHom_{\Fun(B \Sigma, \cV)}(1_{\Fun(B \Sigma, \cV)}, \iHom^\otimes_{\cM}(m , -)) \simeq \iHom_\cM(m , -)$ since
	\begin{align*}
	\Map_\cV &\left( v, \iHom_{\Fun(B \Sigma, \cV)}(1_{\Fun(B \Sigma, \cV)}, \iHom^\otimes_{\cM}(m , -)) \right) \simeq \\ 
	&\simeq \Map_{\Fun(B \Sigma, \cV)} \left( 1_{\Fun(B \Sigma, \cV)} \otimes v,  \iHom^\otimes_{\cM}(m , -)) \right) \simeq \\
	&\simeq \Map_\cM(m \oeert 1_{\Fun(B \Sigma, \cV)} \otimes v, -) \simeq \Map_\cM(m \otimes v, -) \; .
	\end{align*}
\end{obs}

\begin{prop}
	\label{prop:atomicarecpt}
	If $\kappa$ is a regular cardinal such that the unit $1_\mathcal{V} \in \mathcal{V}$ is $\kappa$-compact, then every $\otimes$-atomic object in $\mathcal{M} \in \CAlg(\PrV)$ is also $\kappa$-compact. In particular, the full subcategory $\mathcal{M}^{\otimes\mathrm{at}} \subseteq \mathcal{M}$ on the $\otimes$-atomic objects is always small.
\end{prop}
\begin{proof}
	By \cref{obs:tensoratomic}, $\otimes$-atomic objects are in particular $\cV$-atomic, so this follows from {\cite[Prop.\ 5.9]{ben2024naturality}}, \marked{Prop.}{3.25}.
\end{proof}

\subsection{\texorpdfstring{$\otimes$}{Tensor}-atomically generated categories}

Next, we introduce an analogue of the atomically generated categories from \marked{Def.}{3.13}, and point out some subtle differences such as \cref{warning:atomicpair}.

\begin{defin}
	For $\rC$ a small category and $\cM \in \CAlg(\PrV)$, a functor $\rC \to \cM$ is called a \emph{$\otimes$-atomically generating marking} if it is a $\otimes$-atomic marking, and the smallest full subcategory of $\cM$ containing its image, that is closed under symmetric monoidal structure, $\cV$-tensoring and colimits, is $\cM$ itself. A full subcategory $\cM_0 \subseteq \cM$ is called a \emph{system of $\otimes$-atomic generators} if its inclusion is a $\otimes$-atomically generating marking.  We call $\cM$ \emph{$\otimes$-atomically generated} if it admits a $\otimes$-atomically generating marking (or equivalently by \cref{cor:univisag} a system of $\otimes$-atomic generators).
\end{defin}

\begin{ex}
	For $\cO \in \vOp_X(\cV)$ any $\cV$-operad, the operadic Yoneda functor $\yo_\cO: X \to \PSh \Sym X \otimes \cV \to \PShO(\cO) = \LMod_\cO(\PSh \Sym X \otimes \cV )$ is a $\otimes$-atomically generating marking: The free functor $\PSh \Sym X \otimes \cV \to \LMod_\cO(\PSh \Sym X \otimes \cV )$ is internally left adjoint as part of a monadic adjunction, and colimit-dominant since any module is a geometric realization of free modules. Conversely, we will show in \cref{thm:markedalgebras} that any $\otimes$-atomically generating marking arises like this.
\end{ex}

\begin{warning}
	\label{warning:notag}
	While a system of $\otimes$-atomic generators $\cM_0 \subseteq \cM$ always consists of $\otimes$-atomic objects by \cref{obs:otimesatomicmarkingobj}, this is not enough to ensure that its inclusion is a $\otimes$-atomic marking. For instance, the category $\PSh \Sym(*)^{\times 2}$ from \cref{warning:atomicpair} is not $\otimes$-atomically generated: Combining \cref{warning:atomicpair} and \cref{rem:indecomp}, the only $\otimes$-atomic objects in it are $(1,*)$ and $(*,1)$, and as we had seen the full subcategory spanned by them does not constitute a system of $\otimes$-atomic generators.
\end{warning}

\begin{notat}
	Given a functor $F: \rC \to \rD$ where $\rD$ is cocomplete, its \emph{colimit-closed image} $\CIm(F) \subseteq \rD$ is the smallest full subcategory of $\rD$ containing the image of $F$ that is closed under colimits in $\rD$.
\end{notat}

\begin{lemma}
	\label{lem:CImclosed}
	Let $F: \cM\to \cN$ be a morphism in $\CAlg(\PrV)$. Then its colimit-closed image $\CIm(F) \subseteq \cN$ is closed under $\cV$-tensoring and the symmetric monoidal structure.
\end{lemma}
\begin{proof}
	Let $\cM_1$ denote the full subcategory of $\cM$ on those $m$ such that $m \otimes m_0 \in \CIm(F)$ for all $m_0$ in the full image of $F$. This contains the full image of $F$ and is closed under colimits, hence it must agree with $\CIm(F)$. Now, let $\cM_2$ be the full subcategory on those $m \in \cM$ such that $m \otimes v \in \CIm(F)$ for all $v \in \cV$, and $m \otimes m' \in \CIm(F)$ for all $m' \in \CIm(F)$. By the previous argument this contains the full image of $F$, also it is closed under colimits, so $\cM_2$ agrees with $\CIm(F)$ and we are done.
\end{proof}

\begin{lemma}
	\label{lem:cdomgeneration}
	Let $\rC$ be a small category, $\cM \in \CAlg(\PrV)$ and $y: \rC \to \cM$ a functor. Then the following are equivalent:
	\begin{enumerate}[(1)]
		\item The image of $y$ generates $\cM$ under colimits, $\cV$-tensoring and the symmetric monoidal structure,
		\item The unique extension $Y : \PSh \Sym (\rC) \otimes \cV \to \cM$ to a morphism in $\CAlg(\PrV)$ is colimit-dominant,
		\item The right adjoint $Y^\rR : \cM\to \PSh \Sym (\rC) \otimes \cV$, which sends $m$ to $((c_1, \dots, c_k) \mapsto \iHom_\cM(c_1 \otimes \dots \otimes c_k, m)) \in \Fun(\Sym (\rC)^{\op}, \cV)$, is conservative.
	\end{enumerate}
\end{lemma}
\begin{proof}
	The equivalence $(2) \Leftrightarrow (3)$ follows from \marked{Cor.}{3.13}. Also $(2) \Rightarrow (1)$ since the image of $\rC \to \PSh \Sym \rC \otimes \cV$ clearly generates under colimit, $\cV$-tensoring and symmetric monoidal structure. Finally, $(1) \Rightarrow (2)$ follows since the colimit-closed image of $Y$ contains the image of $y$ and is closed under colimits, $\cV$-tensoring and symmetric monoidal structure by \cref{lem:CImclosed}
\end{proof}

\begin{lemma}
	\label{lem:iLcdomdiagram}
	Let $\cM, \cN, \cP \in \CAlg(\PrV)$ together with functors
	\[
	\begin{tikzcd}
		\cM \arrow[rr, "F"] & & \cP \\
		& \cN \arrow[ul, "I"', "\mathrm{iL,cdom}"] \arrow[ur, "G", "\mathrm{iL}"']
	\end{tikzcd}
	\]
	in $\CAlg(\PrV)$ where $I$ is internally left adjoint and colimit-dominant, and $G$ is internally left adjoint. Then $F$ is internally left adjoint.
\end{lemma}
\begin{proof}
	Consider the full subcategory $\widetilde{\cM} \subseteq \cM$ on those $m \in \cM$ such that for all diagrams $\rI \to \cP$, and all objects $p, p' \in \cP$ and $v\in \cV$ the morphisms
	\[ \Map_{\cM} (m, \colim_i F^\rR p_i) \to  \Map_{\cM}(m, F^\rR \colim_i p_i)  
	\]
	\[\Map_{\cM}(m, F^\rR p \otimes v) \to \Map_{\cM}(m, F^\rR (p \otimes v))
	\]
	\[\Map_{\cM}(m, F^\rR p \otimes F^\rR p') \to \Map_{\cM}(m, F^\rR (p \otimes p'))
	\]
	are isomorphisms. By assumption $\widetilde{\cM}$ contains the full image of $I$, since $\Map_\cM(I(n), -) \simeq \Map_\cM(n, I^\rR -)$ and both $I^\rR$ and $I^\rR F^\rR$ belong to $\CAlg(\PrV)$. Moreover, $\widetilde{\cM}$ is closed under colimits since we can pull them out as limits, so by colimit-dominance we deduce $\widetilde{\cM} = \cM$. By Yoneda, we learn that $F^\rR$ preserves colimits, symmetric monoidal structure and tensoring so $F$ is internally left adjoint.
\end{proof}

\begin{lemma}[Adjoint lifting theorem in $\CAlg(\IPrV)$]
	\label{lem:iRconsdiagram}
	Let $\cM, \cN, \cP \in \CAlg(\PrV)$ together with functors
	\[
	\begin{tikzcd}
		\cM \arrow[dr,"U", "\mathrm{iR}"'] \arrow[rr, "F"] & & \cP \arrow[dl, "V"', "\mathrm{iR, cons}"] \\
		& \cN &
	\end{tikzcd}
	\]
	in $\CAlg(\PrV)$ where $U$ is internally right adjoint, and $V$ is internally right adjoint and conservative. Then $F$ is internally right adjoint.
\end{lemma}
\begin{proof}
	First, note that the full subcategory of $\cN$ spanned by those $n$ such that the copresheaf $\Map_\cM(n , F -) : \cM\to \Spaces$ is representable contains the image of $V^\rL$ by assumption that $U$ is internally right adjoint, also it is closed under colimits since the coYoneda embedding preserves colimits. Hence we deduce that $F$ admits a left adjoint, which automatically inherits an oplax $\cV$-linear oplax symmetric monoidal structure. Once again using that $U$ is internally right adjoint, we learn that $F^\rL$ is strong $\cV$-linear and symmetric monoidal on the image of $V^\rL$, so by colimit-dominance of $V^\rL$ it is so everywhere (similarly to \cref{lem:iLcdomdiagram}).
\end{proof}

\begin{prop}
	\label{prop:iLpreserve}
	Let $F: \cM \to \cN$ be a morphism in $\CAlg(\PrV)$, and $\rC \to \cM$ any $\otimes$-atomically generating marking. Then, $F$ is internally left adjoint iff
	the composition $\rC \to \cM \to \cN$ is a $\otimes$-atomic marking.
\end{prop}
\begin{proof}
	Consider the composition of the unique extension:
	\[
	\PSh \Sym (\rC) \otimes \cV \to \cM \to \cN
	\]
	Internally left adjoint functors compose, so the \emph{if} direction is immediate. For the only if direction, apply \cref{lem:iLcdomdiagram} to this composition.
\end{proof}

\begin{cor}
	\label{cor:univisag}
	Let $y: \rC \to \cM$ be a functor with $\cM \in \CAlg(\PrV)$, and $f: \rC' \to \rC$ a surjective functor. Then $y$ is a $\otimes$-atomically generating marking iff the composition $y \circ f : \rC' \to \cM$ is a $\otimes$-atomically generating marking.
	
	In particular, $y: \rC \to \cM$ is a $\otimes$-atomically generating marking iff the fully faithful inclusion $\operatorname{Im}(y) \hookrightarrow \cM$ is, iff the subcategory inclusion $\operatorname{Im}(y)^{\simeq} \hookrightarrow \cM$ is.
\end{cor}
\begin{proof}
	Since $f: \rC \to \rC'$ is surjective, the induced $\PSh \Sym (\rC) \to \PSh \Sym (\rC')$ is colimit-dominant since it hits all representables, and internally left adjoint by \cref{lem:PSymisiL}. Now the result follows by applying \cref{lem:iLcdomdiagram}, together with composition and cancellation properties of colimit-dominant and internally left adjoint functors.
\end{proof}

\begin{rem}
	We will show in \cref{cor:tensoragfissilechar} that in fact, $\cM \in \CAlg(\PrV)$ is $\otimes$-atomically generated iff the full subcategory inclusion $\cM^{\otimes\mathrm{at}} \subseteq \cM$ on \emph{all} $\otimes$-atomic objects is a system of $\otimes$-atomic generators.
\end{rem}

Even for general $\cM$, we give the following weaker criterion for detecting $\otimes$-atomic markings:

\begin{prop}
	\label{prop:iLonsequ}
	A functor $y: \rC \to \cM$ into $\cM \in \CAlg(\PrV)$ is a $\otimes$-atomic marking iff for any finite tuple $(c_1, \dots, c_n)$ of objects in $\rC$, the images $(y(c_1), \dots,  y(c_n))$ form a $\otimes$-atomic tuple in $\cM$.
\end{prop}
\begin{proof}
	Let $Y : \PSh \Sym (\rC) \otimes \cV \to \cM$ be the unique extension of $y$, and consider the full subcategory ${\cX} \subseteq \PSh \Sym (\rC) \otimes \cV$ on those $x$ 
	such that for all diagrams $m_\bullet: \rI \to \cM$, and all objects $m, m' \in \cP$ and $v\in \cV$ the canonical morphisms
	\[ \iHom_{\PSh \Sym (\rC) \otimes \cV} (x, \colim_i Y^\rR m_i) \to  \iHom_{\PSh \Sym (\rC) \otimes \cV}(x, Y^\rR \colim_i m_i)  
	\]
	\[\iHom_{\PSh \Sym (\rC) \otimes \cV}(x, Y^\rR m \otimes v) \to \iHom_{\PSh \Sym (\rC) \otimes \cV}(x, Y^\rR (m \otimes v))
	\]
	\[\iHom_{\PSh \Sym (\rC) \otimes \cV}(x, Y^\rR m \otimes Y^\rR m') \to \iHom_{\PSh \Sym (\rC) \otimes \cV}(x, Y^\rR (m \otimes m'))
	\]
	are isomorphisms in $\cV$. Since ${\cX} \subseteq \PSh \Sym (\rC) \otimes \cV$ is closed under colimits and $\cV$-tensoring by construction, it suffices to show that it contains the image of $\Sym \rC \subseteq \PSh \Sym(\rC) \to \PSh \Sym (\rC) \otimes \cV$.  So let us choose $c_1, \dots, c_n \in \rC$ and denote $Z: \PSh \Sym \{1, \dots, n\} \to \PSh \Sym (\rC)$ the induced internally left adjoint functor; then for instance for the last map we may rewrite
	\[
	\iHom(Z(1 \otimes \dots \otimes n), Y^\rR m \otimes Y^\rR m') \simeq \iHom(1 \otimes \dots \otimes n, (Y \circ Z)^\rR m \otimes (Y \circ Z)^\rR m')
	\]
	and similarly on the right side, so it is an isomorphism as $Y \circ Z$ is internally left adjoint by assumption.
\end{proof}

\begin{prop}
	\label{prop:fflemma}
	Let $y: \rC \to \cM$ be a $\otimes$-atomically generating marking. An internally left adjoint morphism $F: \cM \to \cN$ in $\CAlg(\PrV)$ is fully faithful if and only if for all $c_1, \dots, c_n, c' \in \rC$, the induced morphism 
	\begin{equation*}
		F: \iHom_{\mathcal{M}}(y c_1 \otimes \dots \otimes y c_n, y c') \overset{}{\to} \iHom_{\mathcal{N}} (F(y c_1) \otimes \dots \otimes F(y c_n), F(y c_m'))
	\end{equation*}
	is an isomorphism in $\mathcal{V}$.	
\end{prop}
\begin{proof}
	The \emph{only if} statement is immediate from the definition of $\iHom$. Conversely assuming the given morphisms are isomorphisms, our first step will be to replace the target $y c'$ by an arbitrary $m' \in \cM$.
	
	Let $Y: \PSh \Sym (\rC) \otimes \cV \to \cM$ be the unique extension of $y$ to a morphism in $\CAlg(\PrV)$, then the above comparison map is obtained by evaluating the transformation $Y^{\rR}(y c') \to (F \circ Y)^\rR (F y c')$ induced by the unit $\id \to F^\rR F$ at the tuple $(c_1, \dots, c_n)$. Since both $Y$ and $F$ are internally left adjoint by assumption, the full subcategory of $m' \in \cM$ where $Y^\rR (m') \to Y^\rR F^\rR F (m')$ is an isomorphism is closed under colimits, $\cV$-tensoring and symmetric monoidal structure, so if it contains the image of $y$ it agrees with $\cM$.
	
	As our next step, for fixed $m' \in \cM$, consider the full subcategory on those $m \in \cM$ such that $\iHom_{\cM}(m, m') \to \iHom_{\cM}(Fm, Fm')$ is an isomorphism. It contains the image of the symmetric monoidal extension $\Sym \rC \to \cM$ by assumption, and is closed under colimits and $\cV$-tensoring since we can pull them out of the expression. Hence it contains the colimit-closed image of $\PSh \Sym(\rC) \otimes \cV$ which is all of $\cM$.
\end{proof}
	%

\begin{prop}
	\label{lem:iLclosedsiftedcolim}
	The subcategory $\CAlg(\PrV)^{\mathrm{iL}} \hookrightarrow \CAlg(\PrV)$ is closed under sifted colimits.
\end{prop}
\begin{proof}
	Since the forgetful functor $\CAlg(\PrV) \to \PrV$ preserves sifted colimits by \HA{Cor.}{3.2.3.2}, this is analogous to {\cite[Corollarly 1.32]{ramzi2024dualizable}}.
\end{proof}

\begin{warning}
	The subcategory $\CAlg(\PrV)^{\mathrm{iL}} \hookrightarrow \CAlg(\PrV)$ is not closed under coproducts: By \HA{Prop.}{3.2.4.7} the coproduct is the pointwise tensor product $\otimes_\cV$ in $\PrV$. Internally left adjoint maps from $\PSh \Sym(*) \otimes \cV$ into $\cM\in \CAlg(\PrV)$ correspond to $\otimes$-atomic objects, while internally left adjoint maps from $(\PSh \Sym(*) \otimes \cV) \otimes_\cV (\PSh \Sym(*) \otimes \cV) \simeq \PSh \Sym(* \sqcup *) \otimes \cV$ correspond to $\otimes$-atomic pairs, which are not the same as pairs of $\otimes$-atomic objects by \cref{warning:atomicpair}. This is remedied by restricting to either fissile or $\otimes$-atomically generated categories, see \cref{prop:colimitfissile} and \cref{prop:colimitiLtensorag}.
\end{warning}

\subsection{\texorpdfstring{$\otimes$}{Tensor}-disjunctive functors}

The extensive list of conditions in \cref{obs:tensoratomic} indicates that checking explicitly whether a given object or pair is $\otimes$-atomic can be difficult. However the following discussion ensures an ample supply of internally left adjoint functors and $\otimes$-atomic objects:

\begin{notat}
	Given a functor $F: \rC \to \rD$ and $d \in \rD$, we write $\rC_{/d}$ for the pullback $\rC \times_{\rD} \rD_{/d}$.
\end{notat}

\begin{defin}[{compare \cite[Def. 3.2.14]{equifib}}]
	\label{defin:tensordisjunctive}
	A symmetric monoidal functor $F: \rC \to \rD$ is \emph{$\otimes$-disjunctive} if for any $a, b \in \rD$ the induced functor
	\[
	\otimes : \rC_{/a} \times \rC_{/b} \to  \rC_{/a \otimes b}
	\]
	is an equivalence. More generally $F$ is called \emph{weakly $\otimes$-disjunctive} if this induced functor is colimit-cofinal for all $a, b \in \rD$.
\end{defin}

\begin{obs}
	\label{obs:tensordisjmultiple}
	Inductively, note that $F$ is (weakly) $\otimes$-disjunctive iff for any $n \in \N_0$ and $d_1, \dots, d_n \in \rD$ the functor
	\[
	\otimes : \rC_{/d_1} \times \dots \times \rC_{/d_n} \to  \rC_{/d_1 \otimes \dots \otimes d_n}
	\]
	induced by the symmetric monoidal structure is an equivalence (colimit-cofinal).
\end{obs}

\begin{obs}
	\label{obs:tensordisj}
	This generalizes the notion of $\otimes$-disjunctive symmetric monoidal categories from \cite[Def. 3.2.14]{equifib}, in the sense that $\rC$ is $\otimes$-disjunctive iff the identity functor $\id_\rC$ is. It follows from \cite[Cor. 3.2.16]{equifib} that the envelope $\Env(\rO)$ of any operad $\rO$ is $\otimes$-disjunctive, i.e.\ for tuples $\underline{o}, \underline{o'} \in \Env(\rO)$ we have
	\[
	\Env(\rO)_{/\underline{o}} \times \Env(\rO)_{/\underline{o'}} \overset{\simeq}{\longrightarrow} \Env(\rO)_{/\underline{o} \otimes \underline{o'}}
	\]
	where $\underline{o} \otimes \underline{o'}$ is the concatenation of tuples. Our next goal is to extend this statement to maps of operads.
\end{obs}

\begin{defin}[{\cite[Def. 2.2.1]{equifib}}]
	\label{defin:equifibered}
	A symmetric monoidal functor $F: \rC \to \rD$ is called \emph{equifibered} if the following commutative square is a pullback:
	\[
	\begin{tikzcd}
		\rC \times \rC \arrow[r, "F \times F"] \arrow[d, "\otimes"] & \rD \times \rD \arrow[d, "\otimes"] \\
		\rC \arrow[r, "F"] & \rD
	\end{tikzcd}
	\]
\end{defin}

\begin{lemma}
	\label{lem:odisjequif}
	A symmetric monoidal functor $F: \rC \to \rD$ is $\otimes$-disjunctive iff the target projection $\rC \times_\rD \Arr(\rD) \to \rD$, regarded as a symmetric monoidal functor by equipping $\Arr(\rD)$ with the pointwise symmetric monoidal structure, is equifibered.
\end{lemma}
\begin{proof}
	As $(1) \Leftrightarrow (2)$ in \cite[Lem. 3.2.15]{equifib}.
\end{proof}

\begin{lemma}[{\cite[above Cor. 2.2.28]{equifib}}]
	\label{lem:equifclosed}
	Equifibered symmetric monoidal functors are closed under composition and pullbacks in $\CAlg(\Cat)$, since they form the right side of a factorization system.
\end{lemma}

\begin{lemma}
	\label{lem:pastingequif}
	Let $\rC, \rD, \rE$ be symmetric monoidal categories, $F: \rC \to \rD$ an equifibered symmetric monoidal functor and $G: \rD \to \rE$ a $\otimes$-disjunctive symmetric monoidal functor. Then the composite $G \circ F$ is $\otimes$-disjunctive.
\end{lemma}
\begin{rem}
	Compare this statement to pasting for (lax) pullback squares: Pullback squares paste (just as equifibered functors compose), but also a pullback square can be pasted from the left to a lax pullback square to obtain another lax pullback.
\end{rem}
\begin{proof}
	By \cref{lem:odisjequif} we must show that the horizontal composition in the commutative diagram
	\[
	\begin{tikzcd}
		\rC \times_{\rE} \Arr(\rE) \arrow[r] \arrow[d] \arrow[dr, phantom, "\scalebox{1}{$\lrcorner$}" , very near start, color=black] & \rD \times_{\rE} \Arr(\rE) \arrow[d, "\mathrm{src}"] \arrow[r, "\mathrm{trg}"] & \rE \\
		\rC \arrow[r, "F"] & \rD & 
	\end{tikzcd}
	\]
	is equifibered. But $\rD \times_{\rE} \Arr(\rE) \to \rE$ is equifibered by the same Lemma, so since the left square is a pullback this follows from \cref{lem:equifclosed}.
\end{proof}

\begin{prop}
	Let $f: \rO \to \rP$ be a map of Lurie-operads. Then the induced symmetric monoidal functor $\Env(f) : \Env(\rO) \to \Env(\rP)$ is $\otimes$-disjunctive.
\end{prop}
\begin{proof}
	By \cref{lem:pastingequif}, writing $\Env(f) = \id_{\Env(\rP)} \circ \Env(f)$ it suffices to note that $\Env(f)$ is equifibered by \cite[Thm. 3.2.13]{equifib} and $\id_{\Env(\rP)}$ is $\otimes$-disjunctive by \cref{obs:tensordisj}.
\end{proof} 

\begin{rem}
	We will see in \cref{cor:tensordisjchar} that in fact, a symmetric monoidal functor $F: \Env(\rO) \to \Env(\rP)$ is of the form $\Env(f)$ for some map of operads $f: \rO \to \rP$ iff it is $\otimes$-disjunctive iff it is weakly $\otimes$-disjunctive.
\end{rem}

\begin{obs}
	\label{obs:precomplax}
	For $F: \rC \to \rD$ a symmetric monoidal functor, since $\PSh : \Cat \to \Pr$ is symmetric monoidal by \HA{Rem.}{4.8.1.8} the induced $F_! : \PSh(\rC) \to \PSh(\rD)$ is symmetric monoidal with respect to the Day convolution product. Thus its right adjoint precomposition functor $F^* : \PSh(\rD) \to \PSh(\rC)$ inherits a canonical lax symmetric monoidal structure.
\end{obs}

We thank Jan Steinebrunner for communicating to us the following statement:

\begin{theorem}
	\label{thm:dayweaklydisj}
	Let $F: \rC \to \rD$ be a symmetric monoidal functor, then the lax symmetric monoidal structure on $F^* : \PSh(\rD) \to \PSh(\rC)$ from \cref{obs:precomplax} is strong symmetric monoidal if and only if $F$ is weakly $\otimes$-disjunctive.
\end{theorem}
\begin{proof}
	To show that the lax symmetric monoidal functor $F^*$ is monoidal, it suffices to prove that for any active morphism $\alpha : \fset{n} \to \fset{m}$ and $w_1, \dots, w_n \in \PSh(\rD)$ the commutative (by symmetric monoidality of $F_!$) diagram
	\[
	\begin{tikzcd}
		\PSh(\rC^{\times n}) \arrow[r, "F_!"] \arrow[d, "\otimes_\alpha"] & \PSh(\rD^{\times n}) \arrow[d, "\otimes_\alpha"] \\
		\PSh(\rC^{\times m}) \arrow[r, "F_!"] & \PSh(\rD^{\times m})
	\end{tikzcd}
	\]
	is  horizontally right adjointable. But by the definition of a Lurie-operad, the operation $\otimes_\alpha : \PSh(\rD^{\times n}) \to \PSh(\rD^{\times m})$ splits into a product of $m$ operations $\otimes_{\alpha_i}: \PSh(\rD^{\times f^{-1}(\{i\})}) \to \PSh(\rD^{\times \{i\}})$, so it suffices to consider the case where $m = 1$ and $\alpha : \fset{n} \to \fset{1}$ is the unique active map. Horizontal right adjointability of the diagram
	\[
	\begin{tikzcd}
		\PSh(\rC^{\times n}) \arrow[r] \arrow[d] & \PSh(\rD^{\times n}) \arrow[d] \\
		\PSh(\rC) \arrow[r] & \PSh(\rD)
	\end{tikzcd}
	\]
	can be checked on representable presheaves $\yo_{(d_1, d_2)} \in \PSh(\rD \times \rD)$, since all of the involved functors (including the horizontal right adjoints) preserve colimits. This presheaf is represented by the right fibration $(\rD^{\times n})_{/(d_1, \dots, d_n)} \to \rD^{\times n}$, and its image $\otimes_! \yo_{(d_1, d_2)} \simeq \yo_{d_1 \otimes d_2}$ by the right fibration $\rD_{/d_1 \otimes d_2} \to \rD$. They fit into a commutative diagram
	\[
	\begin{tikzcd}[sep=.4cm]
		&[+11pt] (\rC \times \rC)_{/(d_1,d_2)} \arrow[rr, "\otimes"] \arrow[dd] \arrow[dl, "F \times F"]  & &[-13pt] \rC_{/d_1 \otimes d_2} \arrow[dd] \arrow[dl, "F"]\\
		(\rD \times \rD)_{/(d_1, d_2)} \arrow[rr, crossing over, "\otimes"] \arrow[dd]  &   & \rD_{/d_1 \otimes d_2}  & \\
		& \rC \times \rC \arrow[rr, "\otimes" {pos=0.5}] \arrow[dl, "F \times F"] &  & \rC \arrow[dl, "F"]  \\
		\rD \times \rD \arrow[rr, "\otimes"] &   & \rD \arrow[from=uu, crossing over] &
	\end{tikzcd}
	\]
	where the vertical morphisms are right fibrations, and the left and right squares are pullbacks. Recall that for a presheaf $W \in \PSh(\rB)$ represented by a right fibration $\rE \to \rB$ and a functor $g: \rB \to \rB'$, if we factor the composite map $\rE \to \rB'$ into a colimit-cofinal map $\rE \to \rE'$ and a right fibration $\rE' \to \rB'$ using that those classes form a factorization system, then $\rE' \to \rB'$ represents the left Kan extension $F_! W \in \PSh(\rB')$: This follows from the composing the adjunction $ g \circ -: \Cat_{/\rB} \to \Cat_{/\rB'} : g^*$ and the reflection of $\mathrm{RFib}_{/\rB'} \subseteq \Cat_{/\rB'}$ by factoring through a right fibration. In other words, to show that the Beck-Chevalley map is an isomorphism it suffices to check that the horizontal map $(\rC \times \rC)_{/(d_1, d_2)} \to \rC_{/d_1 \otimes d_2}$ is colimit-cofinal, thereby exhibiting the presheaf $F^* \yo_{/d_1 \otimes d_2}$ as $\otimes_! F^* \yo_{(d_1, d_2)}$. But this is precisely the condition of weak $\otimes$-disjunctivity.
\end{proof}

This generalizes \cref{lem:PSymisiL} from symmetric algebras of categories, i.e.\ envelops of trivial operads, to envelopes of arbitrary operads:

\begin{cor}
	If $f: \rO\to \rP$ is a map of Lurie-operads, then the induced functor $\PSh \Env (f) \otimes \cV : \PSh \Env (\rO) \otimes \cV \to \PSh \Env (\rP) \otimes \cV$ is internally left adjoint in $\CAlg(\IPrV)$.
\end{cor}
\begin{proof}
	Since the free module functor $- \otimes \cV : \CAlg(\Pr) \to \CAlg(\PrV)$ is a $2$-functor by \cref{cor:freeforgetfulfunctors}, and thus preserves internal left adjoints, it suffices to consider the case $\cV = \Spaces$, which is part of \cref{thm:dayweaklydisj}.
\end{proof}

\begin{cor}
	\label{cor:colorstensoratomic}
	If $\rO$ is a Lurie-operad, then $\PSh \Env (\rO) \otimes \cV \in \CAlg(\PrV)$ is $\otimes$-atomically generated via the marking given by the Yoneda functor $\underline{\rO}^\simeq \hookrightarrow \Env(\rO) \subseteq \PSh \Env(\rO) \to \PSh \Env(\rO) \otimes \cV$. In particular for any color $o \in \rO$, the representable presheaf $\yo_{(o)} \otimes 1_\cV$ on the singlett $(o)$ is $\otimes$-atomic in $\PSh \Env \rO \otimes \cV$.
\end{cor}

\section{Marked algebras and fissile categories}
\label{sec:markedalgebras}

We have defined enriched operads as certain monads in $\CAlg(\IPrV)$. Next, we give a description and characterization of the associated monadic $1$-morphisms, leading to the notion of marked algebras.

\begin{prop}
	\label{prop:monadicmorphisms}
	An adjunction $F: \cM\rightleftarrows \cN : G$ internal to $\CAlg (\IPrV)$ is monadic iff it satifies either of the following equivalent conditions:
	\begin{itemize}
		\item The underlying adjunction in $\IPrV$ is monadic,
		\item The underlying adjunction in $\IPr$ is monadic,
		\item The underlying adjunction in $\widehat{\ICat}$ is monadic,
		\item $F$ is colimit-dominant,
		\item $G$ is conservative.
	\end{itemize}
\end{prop}
\begin{proof}
	The forgetful $2$-functors $\CAlg(\IPrV) \to \IPrV \to \IPr \to \widehat{\ICat}$ create Eilenberg-Moore-objects and monadic $1$-morphisms by \cref{thm:EMPrV}. In $\widehat{\ICat}$ we may apply Barr-Beck-Lurie \HA{Thm.}{4.7.3.5}, where the colimit conditions are automatically satisfied so only conservativity of $G$ remains. This is equivalent to colimit-dominance of $F$ by \marked{Cor.}{3.15}.
\end{proof}

Given a monad $T$ on an object $c$ in a $2$-category $\C$, we should always be able to recover $T$ from the associated monadic adjunction
\[
\mathrm{free} : c \rightleftarrows \LMod_T(c) : \mathrm{fgt}
\]
as the composition $\mathrm{fgt} \circ \mathrm{free} : c \to c$, assuming the Eilenberg-Moore object $\LMod_T(c)$ exists. The following statement makes this precise:

\begin{prop}[{\cite{heinemonadicity, haugsengmonads, stockall2025large}}]
	\label{prop:hmonadsff}
	Let $\C$ be a $2$-category that admits Eilenberg-Moore objects, with underlying $1$-category $\C^{\leq 1}$ and $c \in \C$. Then there exists a fully faithful {left adjoint} functor
	\[
	\EM: \Alg(\Endo_\C(c))^{\op} \hookrightarrow (\C^{\leq 1})_{/c}
	\]
	sending a monad $T$ on $c$ to its Eilenberg-Moore object $\LMod_T(c) \in \C$, equipped with the monadic right adjoint $\LMod_T(c) \to c$ in the corresponding monadic adjunction. Its right adjoint sends $f: d \to c$ to the endomorphism object $\Endo_{\Hom_{\C}(d,c)}(f)$ for the left $\Endo_{\C}(c)$-action on $\Hom_{\C}(d,c)$ by postcomposition. 
	
	Passing to opposites, if $\C$ admits Kleisli-objects, there exists a fully faithful left adjoint
	\[\mathcal{KL} : \Alg(\Endo_\C(c)) \hookrightarrow (\C^{\leq 1})_{c/} \]
	sending a monad $T$ on $c$ to its Kleisli-object $\Kl_T(c) \in \C$, equipped with the free module functor $c \to \Kl_T(c)$.
\end{prop}

\begin{cor}
	\label{cor:monadsff}
	For any $\cM \in \CAlg(\PrV)$, there exists a fully faithful functor
	\[
	\Alg(\Endo^{\otimes, \rL}_\cV(\cM)) \hookrightarrow \CAlg(\PrV)_{\cM/}
	\]
	sending a colimit-preserving $\cV$-linear symmetric monoidal monad $T$ on $\cM$ to its Eilenberg-Moore object $\LMod_T(\cM)$, equipped with the free module functor $\cM \to \LMod_T(\cM)$.
\end{cor}
\begin{proof}
	Apply the second part of \cref{prop:hmonadsff}, and use that Eilenberg-Moore objects and Kleisli-objects in $\CAlg(\IPrV)$ agree in a way that identifies the free module functors by \cref{prop:locgr}.
\end{proof}
\begin{rem}
	Alternatively we could  have also used the first part of \cref{prop:hmonadsff}: By \cref{lem:iRconsdiagram}, the functor $\EM$ sends any map of algebras to an internal right adjoint in $\CAlg(\IPrV)$. Hence, we can pass to left adjoints, which by \cite[Thm. 5.3.8]{fernandomate} induces an equivalence $\CAlg(\PrV)^{\mathrm{iL}} \simeq (\CAlg(\PrV)^{\mathrm{iR}})^{\op}$, so passing to slice categories and using \cref{lem:iLcdomdiagram} we obtain
	\begin{align*}
	\Alg(\Endo^{\otimes, \rL}_\cV(\cM)) &\simeq (\CAlg(\PrV)_{/^{\mathrm{monra}}\cM})^{\op} \simeq (\CAlg(\PrV)^{\mathrm{iR}}_{/^{\mathrm{monra}}\cM})^{\op} \simeq \\
		&\simeq \CAlg(\PrV)^{\mathrm{iL}}_{\cM/^{\mathrm{monla}}} \simeq \CAlg(\PrV)_{\cM/^{\mathrm{monla}}}
	\end{align*}
	where $\CAlg(\PrV)_{\cM/^{\mathrm{monra}}}$ denotes the full subcategory on the monadic right adjoints, and similarly for left adjoints.

	We expect a version of the adjoint lifting statements \cref{lem:iLcdomdiagram} and \cref{lem:iRconsdiagram} to hold in any $2$-category that admits Eilenberg-Moore objects and locally admits geometric realizations. Note that only the second Lemma holds for monadic right adjoints in $\ICat$.
\end{rem}

\begin{notat}
	\label{notat:slicesequivalent}
	For $X \in \Cat$, let $\CAlg(\PrV)_{X/}$ denote the pullback $\CAlg(\PrV) \times_{\lCat} \lCat_{X/}$. By the same argument as in \marked{Lem.}{5.4}, the adjunction $(\PSh^{\tav\mathrm{-rex}} \Sym - ) \otimes \cV : \lCat \rightleftarrows \CAlg(\RMod_\cV(\Catcolim))$ exhibits this as equivalent to $\CAlg(\PrV)_{\PSh \Sym X \otimes \cV /}$.
\end{notat}

\begin{theorem}
	\label{thm:markedalgebras}
	Fix a space $X$ and an enrichment category $\cV \in \CAlg(\PrV)$. 
	The assignment $\cO \mapsto (X \to \PShO(\cO))$ sending a valent enriched $\cV$-operad with space of objects $X$ to its operadic Yoneda functor \cref{defin:operadicpsh} induces a fully faithful embedding
	\[
	\vOp_X(\cV) = \Alg(\Endo^{\rL, \otimes}_\cV((\PSh(\Sym X) \otimes \cV))) \hookrightarrow \CAlg (\PrV)_{X/}
	\]
	whose image consists those functors $\yo: X \to \cM$ with $\cM \in \CAlg (\RMod_\cV(\Pr))$ such that:
	\begin{itemize}
		\item The unique extension $Y: \PSh \Sym (X) \otimes \cV \to \cM$ is an internally left adjoint $1$-morphism in $\CAlg(\RMod_\cV(\IPr))$,
		\item The image of $\yo$ generates $\cM$ under colimits, symmetric monoidal structure and $\cV$-tensoring.
	\end{itemize}
\end{theorem}
\begin{proof}
	By the characterization in \cref{prop:monadicmorphisms}, an internally left adjoint morphism in $\CAlg(\IPrV)$ is a monadic left adjoint iff it is colimit-dominant, which by \cref{lem:cdomgeneration} is equivalent to the second condition. Thus, the statement follows immediately from \cref{cor:monadsff} using \cref{notat:slicesequivalent}.
\end{proof}

\begin{defin}
	\label{defin:markedalgebras}
	We refer to objects $(X \to \cM) \in \CAlg(\PrV)_{X/}$ in the image of the embedding in \cref{thm:markedalgebras} as \emph{marked $\cV$-algebras}.
\end{defin}


This is similar to the description of enriched categories as marked modules in \marked{Thm.}{5.5}. The main difference is that $y: X \to \cM$ being a $\otimes$-atomic marking is not equivalent to $y(x)$ being $\otimes$-atomic for every $x \in X$, see \cref{warning:atomicpair}. We now introduce an extra condition on $\cM$ that remidies this issue. 

Recall that $\CAlg(\PrV) \simeq \CAlg(\CAlg(\PrV))$, i.e.\ the unit $\cV \to \cM$ and the multiplication map $\otimes: \cM \otimes_\cV \cM \to \cM$ are themselves morphisms in $\CAlg(\PrV)$.

\begin{defin}
	\label{def:fissile}
	A presentably symmetric monoidal $\cV$-module $\cM \in \CAlg(\PrV)$ is called \emph{fissile} if the unit $\cV \to \cM$ and the multiplication map $\otimes: \cM \otimes_\cV \cM \to \cM$ are internal left adjoints in $\CAlg(\IPrV)$.
\end{defin}

\begin{obs}
	\label{obs:fissilemultiple}
	Inductively, we see that $\cM \in \CAlg(\PrV)$ is fissile iff for any $n \geq 0$, the $n$-fold multiplication map $\otimes^n : \cM^{\otimes_\cV n} \to \cM$ is internally left adjoint in $\CAlg(\IPrV)$. In other words, $\cM$ lies in $\CAlg(\CAlg(\PrV)^{\mathrm{iL}}) \hookrightarrow \CAlg(\CAlg(\PrV)) \simeq \CAlg(\PrV)$, which is the subcategory of $\CAlg(\PrV)$ on fissile categories and internally left adjoint functors.
\end{obs}

\begin{obs}
	\label{obs:smfissile}
	Any symmetric monoidal $2$-functor $\CAlg(\IPrV) \to \CAlg(\IPr_\cW)$, for instance extension-of-scalars along some $f: \cV \to \cW$, preserves fissile categories.
\end{obs}

\begin{rem}
	While this definition is reminiscent of the notion of rigid categories \cite[Chapter 1, §9]{gr}, its implications are very different and neither implies the other. Our naming was suggested by Nikolaus Betker and stems from the intuition that $\otimes$-atomic objects are separated enough such that atomic objects can be split apart into a product of them, compare \cref{prop:iLfissile}.
\end{rem}

\begin{ex}	
	Given a small symmetric monoidal category $\rC \in \CAlg(\Cat)$, we deduce from \cref{thm:dayweaklydisj} that $\PSh \rC \in \CAlg(\Pr)$ is fissile iff
	\begin{itemize}
		\item The unit $1_\rC \in \rC$ is terminal, and
		\item The multiplication map
		\[
		(\rC \times \rC)_{/c} \times (\rC \times \rC)_{/c'} \to (\rC \times \rC)_{/c \otimes c'}
		\]
		sending $(c_1 \otimes c_2 \to c, c'_1 \otimes c'_2 \to c')$ to $(c_1 \otimes c'_1) \otimes (c_2 \otimes c'_2) \to c \otimes c'$ is colimit-cofinal. 
	\end{itemize}
\end{ex}

\begin{ex}
	Given a category $X \in \Cat$, the unit $* \to \Sym X$ of its symmetric algebra is obtained by applying $\Sym: \Cat \to \CAlg(\Cat)$ to $\emptyset \to X$, and the multiplication $\Sym X \times \Sym X \to \Sym X$ is obtained by applying $\Sym$ to the fold map $X \sqcup X \to X$. Hence by \cref{lem:PSymisiL}, we deduce that $\PSh \Sym X \otimes \cV$ is fissile.
\end{ex}

\begin{prop}
	\label{prop:tensoragdisj}
	Let $F: \cM \to \cN$ be a colimit-dominant morphism in $\CAlg(\PrV)$, and $\cM$ be fissile. Then $\cN$ is fissile. In particular, any $\otimes$-atomically generated category is fissile.
\end{prop}
\begin{proof}
	Since the functor $F$ is a map in $\CAlg(\PrV) \simeq \CAlg(\CAlg(\PrV))$, we obtain commutative squares
	\[
	\begin{tikzcd}
		\cV \arrow[r] \arrow[d, "\id"] & \cM \arrow[d, "F"] \\
		\cV  \arrow[r]  & \cN
	\end{tikzcd} \quad
	\begin{tikzcd}
		\cM \otimes_\cV \cM \arrow[r, "\otimes"] \arrow[d, "F \otimes_\cV F"] & \cM \arrow[d, "F"] \\
		\cN \otimes_\cV \cN \arrow[r, "\otimes"]  & \cN
	\end{tikzcd}
	\]
	in $\CAlg(\PrV)$, where the top horizontal maps are internally left adjoint, and the vertical maps are internally left adjoint since $\otimes_\cV$ is a $2$-functor and colimit-dominant by \marked{Prop.}{3.18}. Hence we conclude using \cref{lem:iLcdomdiagram}.
\end{proof}

The main upshot of this definition is the following:

\begin{prop}
	\label{prop:iLfissile}
	Let $F: \cM \to \cN$ be a morphism in $\CAlg(\PrV)$ such that $\cM$ is $\otimes$-atomically generated and $\cN$ is fissile. Then the following are equivalent: 
	\begin{enumerate}[(1)]
		\item $F$ is internally left adjoint,
		\item $F$ preserves $\otimes$-atomic objects,
		\item Given any/ some system of $\otimes$-atomic generators $\rC \to \cM$, the composition $\rC \to \cM \to \cN$ factors through the $\otimes$-atomic objects of $\cN$,
		\item For $\cM_0 \subseteq \cM^{\otimes\mathrm{at}}$ any/ some set of $\otimes$-atomic generators of $\cM$, the image $F(\cM_0)$ consists of $\otimes$-atomic objects in $\cN$.
	\end{enumerate}
\end{prop}
\begin{proof}
	The implications $(1) \Rightarrow (2) \Rightarrow (3) \Rightarrow (4)$ are straightforward. For $(4) \Rightarrow (1)$ we use the criterion from \cref{prop:iLonsequ}: We must show that for any finite set of objects $m_1, \dots, m_n \in \cM_0$ with $n \geq 0$, the induced functor $\PSh \Sym \{m_1, \dots, m_n\} \otimes \cV \to \cM$ is internally left adjoint. But this functor is obtained by composing the tensor product $\bigotimes_\cV (\PSh \Sym \{m_i\} \otimes \cV \to \cM)$, which is internally left adjoint since $\otimes_\cV$ is a $2$-functor, with the $n$-fold multiplication map $\otimes^n : \cM^{\otimes_\cV n} \to \cM$ which is internally left adjoint by \cref{obs:fissilemultiple}.
\end{proof}
\begin{theorem}
	\label{thm:markedalgfissile}
	Given a space $X$ and $\cM \in \CAlg(\PrV)$, then a functor $y: X \to \cM$ is a marked algebra, i.e.\ in the image of the functor from \cref{introthm:markedalgebras}, iff \begin{itemize}
		\item $\cM$ is fissile,
		\item Every $x \in X$ is sent to a $\otimes$-atomic object in $\cM$, and
		\item The image of $y$ generates $\cM$ under colimits, $\cV$-tensoring and symmetric monoidal structure.
	\end{itemize} 
\end{theorem}
\begin{proof}
	Marked algebras satisfy the given conditions by \cref{prop:tensoragdisj} and \cref{obs:iLpresatomic}. The converse follows from \cref{prop:iLfissile}. 
\end{proof}

\begin{cor}
	\label{cor:tensoragfissilechar}
	Given a presentably symmetric monoidal $\cV$-module $\cM \in \CAlg(\PrV)$, the following are equivalent:
	\begin{itemize}
		\item $\cM$ is $\otimes$-atomically generated.
		\item $\cM$ is fissile, and it is generated under colimits, $\cV$-tensorings and symmetric monoidal structure by its $\otimes$-atomic objects,
		\item The inclusion $\cM^{\otimes\text{at}} \subseteq \cM$ of all $\otimes$-atomic objects into $\cM$ is a $\otimes$-atomic marking.
		\item The inclusion $\cM^{\otimes\text{at}, \simeq} \subseteq \cM$ of the space of all $\otimes$-atomic objects is a $\otimes$-atomic marking.
	\end{itemize}
\end{cor}

\begin{rem}
	\label{rem:cauchyfissile}
	In particular if $y: X \to \cM$ is a marked algebra describing an enriched operad $\cO$, then the inclusion $\cM^{\otimes, \simeq} \hookrightarrow \cM$ is a marked algebra as well, and we call its associated operad the \emph{Cauchy-completion} $\hat{\cO}$ of $\cO$. Compare \cref{rem:cauchy}.
\end{rem}

Finally, we rectify the issues we had in \cref{lem:iLclosedsiftedcolim}, preparing a description of colimits of operads in \cref{prop:operadslimcolim}:

\begin{prop}
	\label{prop:colimitfissile}
	The subcategory $\CAlg(\CAlg(\PrV)^{\mathrm{iL}}) \hookrightarrow \CAlg(\PrV)$ on fissile categories and internally left adjoint morphisms is closed under colimits.
\end{prop}
\begin{proof}
	First, we prove that is closed under coproducts, which in $\CAlg(\PrV)$ are given by $\otimes_\cV$. It is a $2$-functor, so the coproduct of internal left adjoints is an internal left adjoint. In particular if $\cM$ and $\cN$ are fissile, then the multiplication $\otimes^n : (\cM \otimes_\cV \cN)^{\otimes_\cV n} \simeq \cM^{\otimes_\cV n} \otimes_\cV \cN^{\otimes_\cV n} \to \cM \otimes_\cV \cN$ is internally left adjoint as well. 
	
	We must also verify that the universal property of the coproduct restricts, i.e.\ 
	\begin{itemize}
		\item The inclusions $\cM \simeq \cM \otimes_\cV \cV \to \cM \otimes_\cV \cN$ and similarly for $\cN$ are internally left adjoint,
		\item Given $F: \cM \to \cN$ and $G: \cM' \to \cN$ in $\CAlg(\PrV)^{\mathrm{iL}}_{\otimes\text{-ag}}$, the induced map $\cM \otimes_\cV \cM' \to \cN \otimes_\cV \cN \to \cN$ is internally left adjoint, 
	\end{itemize}
	which follow since $\cN$ is fissile.
	
	Finally by \cref{lem:iLclosedsiftedcolim}, $\CAlg(\PrV)^{\mathrm{iL}} \hookrightarrow \CAlg(\PrV)$ is closed under sifted colimits, and the (fully faithful) forgetful functor $\CAlg(\CAlg(\PrV)^{\mathrm{iL}}) \to \CAlg(\PrV)^{\mathrm{iL}}$ preserves sifted colimits since the symmetric monoidal structure $\otimes_cV$ on $\CAlg(\PrV)^{\mathrm{iL}}$ is compatible with sifted colimits as they are calculated in $\CAlg(\PrV)$.
\end{proof}

\begin{prop}
	\label{prop:colimitiLtensorag}
	The subcategory $\CAlg(\PrV)^{\mathrm{iL}}_{\otimes\text{-ag}} \hookrightarrow \CAlg(\PrV)$ on $\otimes$-atomically generated categories and internally left adjoint morphisms is closed under colimits.
\end{prop}
\begin{proof}
	For coproducts we argue as in \cref{prop:colimitfissile}, using that $\otimes$-atomically generated categories are fissile by \cref{prop:tensoragdisj}. Note that if $\rC \to \cM$ is a system of $\otimes$-atomic generators for $\cM$ and $\rD \to \cN$ is one for $\cN$, then $\rC \sqcup \rD \to \cM \otimes_\cV \cN$ exhibits $\cM \otimes_\cV \cN$ as $\otimes$-atomically generated. 
	
	We know from \cref{lem:iLclosedsiftedcolim} that $\CAlg(\PrV)^{\mathrm{iL}} \hookrightarrow \CAlg(\PrV)$ is closed under sifted colimits, so it remains to show $\CAlg(\PrV)^{\mathrm{iL}}_{\otimes\text{-ag}} \subseteq \CAlg(\PrV)^{\mathrm{iL}}$ is closed under sifted colimits. Since internal left adjoints preserve $\otimes$-atomics, any sifted diagram $\cM_\bullet: \rK \to \CAlg(\PrV)^{\mathrm{iL}}_{\otimes\text{-ag}}$ restricts to a diagram $\cM^{\otimes\mathrm{at}}_\bullet : \rK \to \Cat$ on $\otimes$-atomic objects. Now, the sifted colimit of the internally left adjoint (by \cref{prop:iLfissile}) and colimit-dominant arrows $(\PSh \Sym \cM^{\otimes\mathrm{at}}_\bullet \otimes \cV \to \cM_\bullet)$ in $\Arr(\CAlg(\PrV)^{\mathrm{iL}})$ recovers an arrow
	\[
	\PSh \Sym (\colim \cM^{\otimes\mathrm{at}}_\bullet) \otimes \cV \to \colim \cM_\bullet
	\]
	that is internally left adjoint by \cref{lem:iLclosedsiftedcolim} and colimit-dominant by \cref{lem:cdomiLFS}, thus exhibiting $\colim \cM_\bullet$ as $\otimes$-atomically generated.
\end{proof}

\begin{defin}
	The \emph{category of valent $\cV$-enriched operads} is defined as the pullback
	\[
	\vOp(\cV) := \Spaces \times_{\CAlg(\PrV)} \Arr^{\mathrm{iL,cdom}}(\CAlg(\PrV))
	\]
	of the source projection along the free functor $\PSh \Sym (-) \otimes \cV : \Spaces \to \CAlg(\PrV)$. We refer to its morphisms as \emph{maps of $\cV$-enriched operads}. 
	
	Denote by $\col: \vOp(\cV) \to \Spaces$ the projection to the first component, and call $\col(\cO)$ the \emph{space of colors} of $\cO$. Similarly, denote the target projection by $\PShO: \vOp(\cV) \to \CAlg(\PrV)$ as it sends $\cO$ to its operadic enriched presheaf category.
\end{defin}

\begin{lemma}
	\label{lem:cdomiLFS}
	The classes of colimit-dominant and fully faithful morphisms form a factorization system on $\CAlg(\PrV)$. This restricts to a factorization system on the wide sub-$2$-category $\CAlg(\PrV)^{\mathrm{iL}}$ on internally left adjoint morphisms.
\end{lemma}
\begin{proof}
	The first claim follows analogously to \marked{Prop.}{3.16}. The second claim can be deduced from \cref{lem:iLcdomdiagram} analogously to \marked{Cor.}{3.35}.
\end{proof}

\begin{prop}
	\label{prop:operadslimcolim}
	The space-of-colors functor $\col: \vOp(\cV) \to \Spaces$ is a Cartesian and coCartesian fibration, where a map of $\cV$-operads $F: \cO\to \cP$ encoded by a commutative square
	\[
	\begin{tikzcd}			\PSh \Sym(\col \cO) \otimes \cV \arrow[d] \arrow[r] & \PSh \Sym (\col \cP) \otimes \cV \arrow[d]\\
		\PShO(\cO) \arrow[r] & \PShO(\cP)
	\end{tikzcd}
	\]
	\begin{itemize}
		\item is a $\col$-Cartesian morphism iff $\PShO(\cO) \to \PShO(\cP)$ is fully faithful,
		\item is a $\col$-coCartesian morphism iff this is a pushout square, in other words the induced map \[ \left( \PSh \Sym(\col \cP) \otimes \cV \right) \otimes_{\PSh \Sym(\col \cO) \otimes \cV} \PShO(\cO) \to \PShO(\cP) \] is an isomorphism.
	\end{itemize}
	The category $\vOp(\cV)$ of valent $\cV$-operads admits limits and colimits, and $\col$ preserves them. Its fiber over a space $X$ is equivalent to the category $\vOp_X(\cV)$ defined in \cref{def:vOpX}. 
\end{prop}
\begin{proof}
	Completely analogous to \marked{Cor.}{6.6}, \marked{Obs.}{6.8} and \marked{Cor.}{6.9} using \cref{lem:cdomiLFS}. The only subtlety is that in order to apply \cref{prop:colimitiLtensorag}, we need to observe that all involved categories are $\otimes$-atomically generated.
\end{proof}

\begin{cor}
	\label{cor:PShOcolim}
	The operadic presheaf category functor $\PShO: \vOp(\cV) \to \CAlg(\PrV)$ preserves colimits.
\end{cor}
\begin{proof}
	We have defined $\PShO$ as the pullback of the target projection $\Arr^{\mathrm{cdom}}(\CAlg(\Pr_\cV)^{\mathrm{iL}}_{\otimes\mathrm{ag}}) \to \CAlg(\Pr_\cV)^{\mathrm{iL}}_{\otimes\mathrm{ag}}$ along the functor $\PSh \Sym (-) \otimes \cV : \Spaces \to \CAlg(\Pr_\cV)^{\mathrm{iL}}_{\otimes\mathrm{ag}}$. Colimit-dominant arrows are closed under colimits by \cref{lem:cdomiLFS}, also the inclusion $\Pr^{\mathrm{iL}_{\cV, \mathrm{ag}}} \hookrightarrow \CAlg(\PrV)$ preserves colimits by \cref{prop:colimitiLtensorag}, so we conclude since $\Catcolim \hookrightarrow \lCat$ is closed under limits.
\end{proof}

\begin{constr}
	\label{constr:changeofenrichment}
	Given a morphism $f: \cV \to \cW$ in $\CAlg(\Pr)$, we define the \emph{change-of-enrichment} functor $f_!: \vOp(\cV) \to \vOp(\cW)$ as the map of pullbacks
	\[
	\Spaces \times_{\CAlg(\PrV)} \Arr^{\mathrm{iL,cdom}}(\CAlg(\PrV)) \longrightarrow \Spaces \times_{\CAlg(\PrV)} \Arr^{\mathrm{iL,cdom}}(\CAlg(\PrV))
	\]
	involving the identity on $\Spaces$, and the extension-of-scalars functor $f_{\mathrm{ext}} : \CAlg(\PrV) \to \CAlg(\PrW)$. Note that it preserves colimit-dominant functors by \marked{Cor.}{3.18}, and internal left adjoints as it is a $2$-functor by \cref{prop:eosstar}\footnote{Alternatively, as a $2$-functor that locally preserves geometric realizations it preserves monadic morphisms by \cref{prop:locgr}.}.
	
	As a left adjoint $2$-functor, $f_{\mathrm{ext}}$ preserves Eilenberg-Moore objects, so we learn that fiberwise the above agrees with the construction of \cref{constr:coespace}. Further, one can show analogously to \marked{Cor.}{6.19} that $f_!$ preserves colimits and admits a right adjoint $f^\rR_!$. 
\end{constr}

\section{Comparison and Univalence}
\label{sec:comparison}

Our next task will be to investigate how $\Spaces$-enriched operads compare to Lurie's category $\Op$ of $\infty$-operads, which we call Lurie-operads for distinction. To be precise, we prove that $\vOp(\Spaces)$ is equivalent to the category of \emph{flagged operads}, and we will introduce a notion of univalence to obtain an equivalence with $\Op$.

\begin{reminder}
	The functor $\underline{(-)}: \Op \to \Cat$, sending a Lurie-operad $\rP$ to the category $\rP^\otimes_{\fset{1}}$ obtained by discarding all $k$-ary multimorphisms with $k \neq 1$, admits a left adjoint $\mathrm{Triv}_{(-)}: \Cat \to \Op$ sending a category to the corresponding \emph{trivial operad} admitting no $k$-ary multimorphisms for $k \neq 1$, see  \HA{Prop.}{2.1.4.11}.
	
	Also, the inclusion $\CAlg(\Cat) \hookrightarrow \Op$ on those operads $\rO^\otimes \to \Fin$ that are coCartesian fibrations admits a left adjoint by \HA{Prop.}{2.2.4.9}, called the \emph{symmetric monoidal envelope} $\Env: \Op \to \CAlg(\Cat)$. Composition of left adjoints implies that $\Sym \simeq \Env \circ \mathrm{Triv}_{(-)}$.
\end{reminder}

\begin{defin}
	A map of Lurie-operads $f: \rO \to \rP$ is called \emph{surjective} if the underlying functor $f: \rO_{\fset{1}} \to \rP_{\fset{1}}$ on categories of colors is surjective. We denote by $\Arr^{\text{surj}}(\Op) \subseteq \Arr(\Op)$ the full subcategory on surjective maps of Lurie-operads.
	
	A map of Lurie-operads $f: \rO \to \rP$ is called \emph{fully faithful} if it induces isomorphisms on multimorphism spaces, i.e.\ for any $n \in \N$ and $o_1, \dots, o_n, o \in \rO$ the induced map $\Mul_\rO(o_1, \dots, o_n ; o) \to \Mul_\rP(f(o_1), \dots, f(o_n) ; f(o))$ is an isomorphism.
\end{defin}

\begin{obs}
	The classes of surjective and fully faithful maps of operads form a factorization system on $\Op$.
\end{obs}

\begin{lemma}
	\label{lem:PShEnvcdom}
	If $f: \rO \to \rP$ in $\Op$ is surjective, then the induced symmetric monoidal functor $\Env(f): \Env(\rO) \to \Env(\rP)$ is surjective, and $\PSh \Env(f) \otimes \cV : \PSh \Env(\rO)  \otimes \cV \to \PSh \Env(\rP)  \otimes \cV$ in $\CAlg(\PrV)$ is colimit-dominant.
\end{lemma}
\begin{proof}
	The first claim is clear, since the image of $\rO \hookrightarrow \Env \rO$ generates $\Env \rO$ under the symmetric monoidal structure, which is preserved by $\Env(f)$. Then $\PSh \Env(f)$ is colimit-dominant since it hits all representables, and so is $\PSh \Env(f) \otimes \cV$ by \marked{Lem.}{3.9}.
\end{proof}

\begin{defin}
	A \emph{flagged operad} consists of a Lurie-operad $\rO$, a space $X$ and a surjective map $X \to \rO_{\fset{1}}$ into the colors of $\rO$. We will simply denote it as $X \to \rO$.
	
	The category $\FOp$ of flagged operads is the full subcategory of the pullback $\Arr(\Cat) \times_{\Cat} \Op$, along the target projection and $\underline{(-)}$, 
	on those pairs $(X \to \underline\rO, \rO)$ where $X$ is a space and the functor is surjective. The adjunction $\Cat \rightleftarrows \Op$ exhibits this as equivalent to the pullback $\Spaces \times_{\Op} \Arr^{\text{surj}}(\Op)$, using \marked{Lem.}{A.7}.
\end{defin}

\begin{constr}\label{con:associatedmarked}
	The \emph{associated marked $\Spaces$-algebra} of a flagged operad $X \to \rO$ is the marked algebra $X \to \rO \to \PSh(\Env \rO)$. Note that this is a $\otimes$-atomic marking by \cref{cor:colorstensoratomic}, and generates under colimits by \cref{lem:PShEnvcdom}. This construction assembles into a functor $\FOp \to \vOp(\Spaces)$ via the following map of pullback squares:		\[
	\begin{tikzcd}[sep=.4cm]
		&[+11pt] \vOp(\Spaces) \arrow[rr] \arrow[dd] \arrow[dr, phantom, "\scalebox{1}{$\lrcorner$}" , very near start, color=black] & &[-13pt] \Arr^{\mathrm{iL,cdom}}(\CAlg(\Pr)) \arrow[dd, "\mathrm{src}" {pos=0.6}]\\
		\FOp \arrow[rr, crossing over] \arrow[dd] \arrow[ur, dashed] \arrow[dr, phantom, "\scalebox{1}{$\lrcorner$}" , very near start, color=black] &   & \Arr^{\mathrm{surj}}(\Op) \arrow[ur, outer sep=-2pt, "\PSh \Env"']  & \\
		& \Spaces \arrow[rr] &  & \CAlg(\Pr)  \\
		\Spaces \arrow[ur, Rightarrow, no head]\arrow[rr] &   & \Op \arrow[from=uu, crossing over, "\mathrm{src}" {pos=0.6}] \arrow[ur, outer sep=-2pt, "\PSh \Env"'] &
	\end{tikzcd}
	\]
\end{constr}

\begin{defin}
	\label{def:underlyingflagged}
	The \emph{underlying flagged operad} of a marked $\cV$-algebra $(\col \cO \to \PShO(\cO)) \in \vOp(\cV)$ is the surjective functor $\col \cO \to \operatorname{Im}\left(\col \cO \to \PShO(\cO)\right) \in \FOp$, where we view the full image as an operad by virtue of it being a full subcategory of the symmetric monoidal category $\PShO(\cO)$.
\end{defin}

\begin{prop}
	\label{prop:freeunderlyingadj}
	Let $\cV \in \Alg(\Pr)$ and let $\iota: \Spaces \to \cV$ in $\Alg(\Pr)$ denote its unit. The composite \[ \FOp \to \vOp(\Spaces) \overset{\iota_!}{\to} \vOp(\cV)\]  admits a right adjoint, which sends a valent $\cV$-enriched operad $\cO$ to its underlying flagged operad $\col \cO \to \operatorname{Im}\left(\yo_{\cO}: \col \cO \to \PShO(\cO)\right)$ .	
\end{prop}

\begin{proof}
	Recall from \cref{constr:changeofenrichment} that the change-of-enrichment functor $\iota_!: \vOp(\Spaces) \to \vOp(\cV)$ sends a marked $\cV$-algebra $X\to \cM$ to $X \to \cM \otimes_{\Spaces}\cV \simeq \cM \otimes \cV$. For $X\to \rO$ a flagged operad 
	and $\yo: Y \to \cM$ a marked $\cV$-algebra, consider the following diagram:	
	\[
	\begin{tikzcd}
		\Map_{\FOp}\left(\oArr{X}{\rO}, \oArr{Y}{\operatorname{Im}(\yo)}\right) \arrow[r] \arrow[d] & \Map_{\widehat{\Op}}(\rO, \operatorname{Im}(\yo)) \arrow[r] \arrow[d] & \Map_{\widehat{\Op}}(\rO, \cM) \arrow[d] \\
		\Map_{\Spaces}(X, Y) \arrow[r] & \Map_{\widehat{\Op}}(X, \operatorname{Im}(\yo)) \arrow[r] & \Map_{\widehat{\Op}}(X, \cM) 
	\end{tikzcd}
	\]
	The left square is a pullback by definition of the mapping spaces of $\FOp$. Orthogonality of surjective and fully faithful maps of operads shows that the right square is a pullback, since $X\to \rO$ is surjective and $\operatorname{Im}(\yo) \to \cM$ fully faithful. Thus, the total square is a pullback and
	\begin{align*} &\Map_{\FOp} \left( \oArr{X}{\rO}, \oArr{Y}{\operatorname{Im}(\yo)} \right) \simeq \Map_{\Spaces}(X, Y) \times_{\Map_{\widehat{\Op}}(X, \cM)} \Map_{\widehat{\Op}}(\rO, \cM) \simeq  \\
		&\quad {\simeq}  \Map_{\Spaces}(X, Y) \times_{\Map_{\CAlg(\PrV)}(\PSh \Sym X \otimes \cV, \cM)} \Map_{\CAlg(\PrV)}(\PSh \Env \rO \otimes \cV, \cM)  \end{align*}
	which by definition agrees with the mapping space $\Map_{\vOp(\cV)} \left( \oArr{X}{\PSh \Env \rO \otimes \cV}, \oArr{Y}{\cM} \right)$.
\end{proof}

Inspecting the unit and counit, we conclude that for $\cV = \Spaces$, this adjunction is in fact an equivalence:

\begin{theorem}\label{thm:vOpS} The associated-marked-$\Spaces$-algebra functor $\FOp \to \vOp(\Spaces)$ from \cref{con:associatedmarked} is an equivalence.
\end{theorem}		
\begin{proof} Let $(X\to \rO)\in \FOp$; our goal is to prove that both unit and counit of the adjunction from~\cref{prop:freeunderlyingadj} are isomorphisms. The unit is the map of flagged operads $(X\to \rO) \to (X\to \operatorname{Im}(\yo:\rO \to \PSh \Env \rO))$. Since $\rO \to \Env(\rO) \to \PSh \Env(\rO)$ is fully faithful, this is an isomorphism.
	
	Conversely, for $(\yo:Y \to \cM)\in \vOp(\Spaces)$, the counit is the map of marked $\Spaces$-algebras $(Y \to \PSh \Env (\operatorname{Im}(\yo))) \to (Y \to \cM)$. The functor $\PSh \Env (\operatorname{Im}(\yo)) \to \cM$ is colimit-dominant by cancellation; it remains to show it is fully faithful. We may apply the criterion in \cref{prop:fflemma} to the $\otimes$-atomic marking $Y \to \PSh \Env (\operatorname{Im}(\yo))$, so we must prove that
	\begin{align*}
		&\Map_{\Env (\operatorname{Im}(\yo))}(\yo_{y_1} \otimes \dots \otimes \yo_{y_n}, \yo_{y}) \simeq \Mul_{\operatorname{Im}(\yo)}(\yo_{y_1}, \dots, \yo_{y_n} ; \yo_{y}) \to \\
		&\quad \to \Map_\cM(\yo_{y_1} \otimes \dots \otimes \yo_{y_n}, \yo_{y})
	\end{align*}
	is an equivalence for any $y_1, \dots, y_n, y \in Y$. This is immediate since we have constructed $\operatorname{Im}(\yo) \subseteq \cM$ as a full suboperad.
\end{proof}

\begin{obs}
	\label{obs:operadicpshspaces}
	By construction of the associated-marked-$\Spaces$-algebra functor \cref{con:associatedmarked}, if  $X \to \rO$ is the flagged Lurie-operad associated to a $\Spaces$-enriched operad $\cO \in \vOp(\Spaces)$ via the equivalence from \cref{thm:vOpS}, then $\PSh^\otimes_{\Spaces}(\cO) \simeq \PSh \Env \rO$.
\end{obs}

\begin{rem}
	In \cref{obs:envagree} we will show that the same is true for enriched operads, namely $\PShO(\cO) \simeq \PShV \vEnv (\cO)$ agrees with the enriched presheaf category of the enriched envelope in the sense of \cref{constr:env}.
\end{rem}

\begin{defin}
	A $\cV$-enriched operad $\cO \in \vOp(\cV)$, with associated marked $\cV$-algebra $(\yo: \col \cO \to \PShO(\cO))$, is called \emph{univalent} if the functor $\yo$ is a subcategory inclusion. In other words, it exhibits $\col \cO \simeq \operatorname{Im}(\yo)^\simeq$. Denote by $\Op(\cV) \subseteq \vOp(\cV)$ the full subcategory on univalent $\cV$-operads.
\end{defin}

\begin{rem}
	\label{rem:univalentonVcat}
	We will see in \cref{obs:PShVinPShO} that the operadic Yoneda functor $\col \cO \to \PShO(\cO)$ factors through the full inclusion $\PShV(\ColV(\cO)) \subseteq \PShO(\cO)$, so $\cO$ is univalent iff its underlying $\cV$-category $\ColV(\cO)$ is.
\end{rem}

\begin{rem}
	\label{rem:comparisonunivalence}
	This agrees with the notion of univalence introduced in \cite{haugsengchu}, either by \cref{rem:univalentonVcat} or \cref{thm:univloc}.
\end{rem}

\begin{defin}
	\label{obs:embedOp}
	The target projection $\mathrm{trg} : \FOp \to \Op$ admits a fully faithful right adjoint $\rO \mapsto (\underline{\rO}^\simeq \to \rO)$. We call a flagged operad $(X \to \rO)$ \emph{univalent} it lies in the image of this right adjoint, i.e.\ the induced map $X \to \underline{\cO}^{\simeq}$ is an isomorphism, or equivalently $X \to \underline{\rO}$ is a subcategory inclusion.
\end{defin}

\begin{cor}
	\label{cor:OpS}
	The equivalence from \cref{thm:vOpS} restricts to an equivalence $\Op(\Spaces) \simeq \Op$, where we embed $\Op \subseteq \FOp$ as in \cref{obs:embedOp}.
\end{cor}
\begin{proof}
	Clearly if a marked algebra $(\yo: X \to \cM)$ is univalent, then so is its associated flagged operad $(X \to \mathrm{Im}(\yo))$. Conversely, given a flagged operad $(X \to \rO)$ such that $X \to \underline{\rO}$ is a subcategory inclusion, then so is the composition $X \to \underline{\rO} \subseteq \Env(\rO) \subseteq \PSh \Env(\rO)$.
\end{proof}

\begin{prop}
	\label{prop:ffchar}
	For a map of valent $\cV$-operads $F: \cO\to \cP$, the following are equivalent:
	\begin{itemize}[(1)]
		\item The underlying morphism $\PShO(F) : \PShO(\cO) \to \PShO(\cP)$ in $\CAlg(\PrV)$ is fully faithful,
		\item For any $o_1, \dots, o_n, o \in \col \cO$, the induced map on multigraphs
		\begin{align*}
			&\Mul_\cO(o_1, \dots, o_n; o) \simeq \iHom_{\PShO(\cO)}(\yo(o_1) \otimes \dots \otimes \yo(o_n), \yo(o)) \to \\
			&\to \iHom_{\PShO(\cP)}(\yo(F o_1) \otimes \dots \otimes \yo(F o_n), \yo(F o)) \simeq \Mul_\cO(F o_1, \dots, F o_n; F o)
		\end{align*}
		using the operadic Yoneda Lemma \cref{prop:operadicyoneda} is an isomorphism in $\cV$,
		\item $F$ is $\col$-Cartesian for the Cartesian fibration $\col : \vOp(\cV) \to \Spaces$.
	\end{itemize}
\end{prop}
\begin{proof}
	$(1) \Leftrightarrow (2)$ follows from \cref{prop:fflemma} by definition of marked algebras. Further, $(1) \Leftrightarrow (3)$ is part of \cref{prop:operadslimcolim}.
\end{proof}

\begin{defin}
	A map of valent $\cV$-operads $F: \cO\to \cP$ is called \emph{fully faithful} if either of the equivalent conditions of \cref{prop:ffchar} is satisfied.
\end{defin}

\begin{defin}
	A map of valent $\cV$-operads $F: \cO\to \cP$ is called \emph{surjective-on-colors} if the underlying map $\col \cO \to \col \cP$ in $\Spaces$ is surjective on connected components.
	
	On the other hand, $F: \cO\to \cP$ is called \emph{surjective} if the induced map $\mathrm{Im}(\yo_\cO) \to \mathrm{Im}(\yo_\cP)$ in $\Cat$ is surjective, i.e.\ if the induced map on univalizations $u\cO \to u\cP$ is surjective-on-colors. 
\end{defin}

\begin{obs}
	\label{obs:surjectiveandcdom}
	If $F: \cO\to \cP$ is a surjective map of $\cV$-operads, then the induced functor $F_! : \PShO(\cO) \to \PShO(\cP)$ is colimit-dominant by cancellation, So if $F$ is additionally fully faithful, then $F_!$ is an equivalence.
\end{obs}

\begin{theorem}
	\label{thm:univloc}
	The full inclusion $\Op(\cV) \subseteq \vOp(\cV)$ admits a left adjoint $u$, called \emph{univalization}, that sends a marked algebra $\yo: \col \cO \to \PShO(\cO)$ to the subcategory inclusion $\mathrm{Im}(\yo)^\simeq \hookrightarrow \PShO(\cO)$.
	This left adjoint exhibits $\Op(\cV)$ as the localization of $\vOp(\cV)$ at the single morphism $(* \sqcup * \to \PSh \Sym(*)) \to (* \to \PSh \Sym(*))$. It inverts precisely those morphisms in $\vOp(\cV)$ that are both fully faithful and surjective.
\end{theorem}
\begin{proof}
	The functor $\mathrm{Im}(\yo)^\simeq \hookrightarrow \PShO(\cO)$ is indeed a $\otimes$-atomically generating marking by \cref{cor:univisag}, and thus defines a marked algebra. Now, the first claim is analogous to \marked{Thm.}{7.27}. Further, note that for any marked algebra $X \to \cM$, being local with respect to the above morphism is equivalent to any diagram
	\[
	\begin{tikzcd}
		* \sqcup * \arrow[r] \arrow[d] & X \arrow[d]\\
		* \arrow[r] \arrow[ur, dashed] & \cM
	\end{tikzcd}
	\]
	admitting a unique filler, which is equivalent to the right map being a monomorphism in $\Cat$, i.e.\ a subcategory inclusion.
	
	For the last claim, note that a fully faithful and surjective morphism $\cO\to \cP$ induces an isomorphism on univalizations since by definition the induced map $\operatorname{Im}(\yo_\cO) \to \operatorname{Im}(\yo_\cP)$ is an equivalence, as is the induced map on operadic presheaf categories by \cref{obs:surjectiveandcdom}. Conversely if $uF: u\cO \to u\cP$ is an isomorphism, then it is an equivalence on operadic presheaf categories so $F$ must have been fully faithful. Also the map $\col(u\cO) \simeq \operatorname{Im}(\yo_\cO) \to \operatorname{Im}(\yo_\cP) \simeq \col(u \cP)$ is an equivalence implying that $F$ is surjective.
\end{proof}
\begin{rem}
	Since $(* \sqcup * \to \PSh \Sym(*)) \to (* \to \PSh \Sym(*))$ is fully faithful and surjective-on-objects, we deduce that $\Op(\cV)$ can also be obtained as the localization at the intermediate class of fully faithful and surjective-on-objects morphisms.
\end{rem}

\begin{rem}
	As in \marked{Thm.}{7.15} one shows that the univalization functor induces an equivalence between $\vOp(\cV)$ and the category of \emph{flagged $\cV$-operads} 
	\[
	\Spaces \times_{\Op(\cV)} \Arr^{\mathrm{surj}}(\Op(\cV)) \simeq \Arr^{\mathrm{surj}}(\Spaces) \times_{\Spaces} \Op(\cV)
	\]
	whose objects are pairs of a univalent $\cV$-operad $\cO$ and a surjective map $X \to \col \cO$.
\end{rem}

\section{Enriched envelopes}
\label{sec:envelopes}

For Lurie-operads, the forgetful functor $\CAlg(\Cat) \to \Op$ admits a left adjoint, called the \emph{envelope functor} $\Env$. In this section, we extend it to enriched operads and investigate some of its basic properties.

\begin{obs}
	\label{obs:smVcat}
	Recall that a valent $\cV$-enriched category $\cC$ is described by its associated marked module $\ob \cC \to \PShV(\cC)$, i.e.\ a map from a space $\ob \cC$ to some $\PShV(\cC) \in \PrV$ whose image consists of $\cV$-atomic objects and generates under colimits and $\cV$-tensoring \marked{Def.}{6.1}. Hence, a symmetric monoidal valent $\cV$-category is an object of
	\[
	\CAlg(\vCatV) \subseteq \CAlg(\Fun([1], \lCat) \times_{\lCat} \PrV) \simeq \Fun([1], \CAlg(\lCat)) \times_{\CAlg(\lCat)} \CAlg(\PrV) \; ,
	\]
	i.e.\ it consists of $\ob \cC \in \CAlg(\Spaces), \PShV(\cC) \in \CAlg(\PrV)$ as well as a symmetric monoidal functor $\ob\cC \to \PShV(\cC)$ that is a marked module. We have used that the $2$-functor $\CAlg = \Hom_{\CAlg(\ICat)}( \Env \E_\infty, - ) : \CAlg(\ICat) \to \ICat$ is compatible with limits and cotensorings.
\end{obs}

\begin{defin}
	\label{defin:envelopefun}
	Given a valent $\cV$-enriched operad $\cO$ represented by the marked algeba $\col \cO \to \PShO(\cO)$, the unique symmetric monoidal extension $\Sym \col \cO \to \PShO(\cO)$ is a marked module since a tensor product of finitely many $\otimes$-atomic objects is atomic by \cref{obs:tensoratomic}. 
	Using  \cref{obs:smVcat}, we define the \emph{valent $\cV$-enriched envelope functor} $\vEnv_\cV: \vOp(\cV) \to \vCat(\cV)$ through the following map of pullback squares:
	\[
	\begin{tikzcd}[sep=.4cm]
		&[+11pt] \CAlg(\vCatV) \arrow[rr] \arrow[dd] \arrow[dr, phantom, "\scalebox{1}{$\lrcorner$}" , very near start, color=black] & &[-13pt] \Arr^{\mathrm{\cV\text{-iL},cdom}}(\CAlg(\PrV)) \arrow[dd, "\mathrm{src}" {pos=0.6}]\\
		\vOp(\cV) \arrow[rr, crossing over] \arrow[dd] \arrow[ur, dashed] \arrow[dr, phantom, "\scalebox{1}{$\lrcorner$}" , very near start, color=black] &   & \Arr^{\mathrm{iL,cdom}}(\CAlg(\PrV)) \arrow[ur, outer sep=-2pt, hook]  & \\
		& \CAlg(\Spaces) \arrow[rr] &  & \CAlg(\PrV)  \\
		\Spaces \arrow[ur, outer sep=-2pt, "\Sym"] \arrow[rr] &   & \CAlg(\PrV) \arrow[from=uu, crossing over, "\mathrm{src}" {pos=0.6}] \arrow[ur, Rightarrow, no head] &
	\end{tikzcd}
	\]
	Here, we use that a morphism that is internally left adjoint in $\CAlg(\IPrV)$ is also internally left adjoint in $\IPrV$, and denote the latter as $\cV\text{-iL}$ for distinction.
\end{defin}

\begin{rem}
	\label{rem:envagrees}
	It follows from \cref{obs:operadicpshspaces} that this agrees with the envelope from \HAsubsec{2.2.4}. However our model for enriched operads makes it difficult to construct a forgetful functor $\CAlg(\vCatV) \to \vOp(\cV)$ right adjoint to $\vEnv_\cV$; we do this in \cref{rem:envelopela}. 
\end{rem}

\begin{obs}
	\label{obs:envagree}
	As a right adjoint $2$-functor, the forgetful functor $\CAlg(\IPrV) \to \IPrV$ preserves weighted limits, in particular Eilenberg-Moore objects. We deduce that the above definition of the valent envelope agrees with \cref{constr:env}.
	
	Similarly, the left adjoint $2$-functor $\Sym^{\otimes_\cV}: \IPrV \to \CAlg(\IPrV)$ preserves Kleisli-objects which agree with Eilenberg-Moore objects on both sides. We deduce that the functor $\mathrm{Triv}_\cV : \vCatXV \to \vOp_X(\cV)$ sends the marked module $\ob \cC \to \PShV(\cC)$ to the marked algebra $\ob \cC \to \Sym^{\otimes_\cV}(\PShV(\cC))$. In fact:
\end{obs}

\begin{prop}
	The adjunction $\ColV: \vCatXV \rightleftarrows \vOp_X(\cV) : \mathrm{Triv}_\cV$ from \cref{constr:modulesvsalgebras} extends to an adjunction
	\[
	\mathrm{Triv}_\cV: \vCatV \rightleftarrows \vOp(\cV) : \ColV
	\]
	relative to $\Spaces$, where $\mathrm{Triv}_\cV(\ob \cC \to \PShV(\cC)) \simeq (\ob \cC \to \Sym^{\otimes_\cV}(\PShV(\cC)))$ and $\ColV(\col \cO \to \PShO(\cO)) \simeq (\col \cO \to \CIm(\PSh \col\cO \otimes \cV \to \PShO(\cO)))$.
\end{prop}
\begin{proof}
	Since $\CIm(\PSh \col\cO \to \PShO(\cO)) \subseteq \PShO(\cO)$ is a full sub-$\cV$-module category, the images of $\col \cO$ are still $\cV$-atomic in it, so this is indeed a marked module. Also, the first expression can be made functorial in $X$ using a map of pullback squares similar to \cref{defin:envelopefun}. It remains to show that the second expression is indeed right adjoint to it, so we obtain an adjunction relative to $\Spaces$, in particular fiberwise adjunctions. This follows from a straightforward calculation.
\end{proof}

\begin{obs}
	\label{obs:PShVinPShO}
	By this description, $\PShO(\cO)$ contains $\PShV\ColV(\cO)$ as the smallest full subcategory containing the image of the operadic Yoneda functor that is closed under colimits and $\cV$-tensoring. In particular, the $\cV$-category $\mathrm{vEnv}_\cV (\cO)$ contains $\ColV(\cO)$ as a full sub-$\cV$-category. 
\end{obs}

\begin{reminder}
	\label{reminder:equifibfree}
	By \cite[Obs. 2.1.17]{equifib}, the symmetric algebra functor $\Sym: \Spaces \to \CAlg(\Spaces)$ is a subcategory inclusion, and becomes fully faithful when passing to the wide subcategory $\CAlg(\Spaces)^{\mathrm{equf}}$ of equifibered maps in the sense of \cref{defin:equifibered}.
\end{reminder}	



\begin{prop}
	\label{prop:valentenv}
	The valent envelope functor $\vEnv_\cV : \vOp(\cV) \to \CAlg(\vCatV)$ is a subcategory inclusion. Its image consists of:
	\begin{itemize}
		\item Those symmetric monoidal valent $\cV$-categories $\cC$ whose underlying $\E_\infty$-space $\ob \cC$ is equifibered, and
		\item Those symmetric monoidal $\cV$-enriched functors $F: \cC \to \cD$ such that the induced map of spaces $\ob \cC \to \ob \cD$ is equifibered, and the induced functor $F_!: \PShV(\cC) \to \PShV(\cD)$ between enriched presheaf categories is internally left adjoint in $\CAlg(\PrV)$.
	\end{itemize}
\end{prop}
\begin{proof}
	Using \cref{reminder:equifibfree}, we have defined the functor $\vEnr_\cV$ as a pullback of subcategory inclusions in \cref{defin:envelopefun}.
\end{proof}

\begin{warning}
	\label{warning:envuniv}
	If $\cO \in \Op(\cV)$ is univalent, the valent envelope $\vEnv_\cV (\cO)$ is not necessarily univalent as an enriched category. This is true for $\cV = \Spaces$ by \cite[Prop. 2.4.11]{haugsengenv}; also the converse statement is generally true since $X \to \Sym X$ is a (full) subcategory inclusion.
	
	However if $\cV = *$ is the category of $(-2)$-categories, i.e.\ the point, then $\Pr_{*} = \{*\}$ so there exist precisely two univalent $*$-operads, the empty operad and the terminal operad described by the marked algebra $(* \to *)$. But $\vEnv_\cV(* \to *) \simeq (\Sym (*) \to *)$, which is not univalent.
	
	As was pointed out to us by Thomas Blom, this example immediately extends to any $\cV$ admitting a zero object, in particular vector spaces, derived categories or $\Sp$.
\end{warning}

\begin{rem}
	For this theorem, it is very important that we work with valent $\cV$-operads and the valent envelope functor. Extending it to the univalent envelope $\Env_\cV := u \vEnv_\cV : \Op(\cV) \to \CatV$ is difficult, unless they agree like for $\Spaces$, c.f.\ \cref{warning:envuniv}. In future work \cite{cauchyoperads}, we show that one can recover $\cO$ from $\Env_\cV(\cO)$ iff $\cO \subseteq \Env_\cV(\cO)$ is closed under absolute colimits for $\cV$-enrichment; see also the related statement in \cref{rem:cauchy}. 
\end{rem}


\begin{cor}
	\label{cor:enrichedtensordisj}
	Let $\cO, \cP$ be $\cV$-enriched operads. A symmetric monoidal $\cV$-functor $F: \vEnv_\cV(\cO) \to \vEnv_\cV(\cP)$ is of the form $\vEnv_\cV(f)$ for some map of $\cV$-operads $f: \cO \to \cP$, iff it is weakly $\otimes$-disjunctive and the underlying map of $\E_\infty$-spaces $f_0 : \col \cO \to \col \cP$ is equifibered.
\end{cor}
%

In case $\cV = \Spaces$, we can prove a similar statement for the univalent envelope, where the condition that the underlying map on $\E_\infty$-spaces of object is equifibered can be dropped. Extending this result to general enrichment fails due to \cref{warning:envuniv}, we will discuss the extent of this failure in \cite{cauchyoperads}.

\begin{cor}
	\label{cor:tensordisjchar}
	Let $\rO, \rP \in \Op$ be Lurie-operads and $F: \Env \rO\to \Env \rP$ a symmetric monoidal functor. Then the following are equivalent:
	\begin{enumerate}[(1)]
		\item There exists a map of operads $f: \rO \to \rP$ such that $F \simeq \Env(f)$,
		\item $F$ is equifibered,
		\item $F$ is $\otimes$-disjunctive,
		\item $F$ is weakly $\otimes$-disjunctive,
		\item The induced map $F_!: \PSh \Env \rP \to \PSh \Env \rO$ in $\CAlg(\Pr)$ is internally left adjoint.
	\end{enumerate}
\end{cor}
\begin{proof}
	$(1) \Rightarrow (2)$ is \cite[Thm. 3.2.13]{equifib}, $(2) \Rightarrow (3)$ follows from \cref{lem:pastingequif}, $(3) \Rightarrow (4)$ by definition and $(4) \Leftrightarrow (5)$ by \cref{thm:dayweaklydisj}. 
	
	It remains to prove $(5) \Rightarrow (1)$. Consider the marked algebra $\yo_\rO: \col \rO \to \PSh \Env \rO$ asssociated to $\rO$, then the composition $F_! \yo_\rO : \col \rO \to \PSh \Env \rO \to \PSh \Env \rP$ is still internally left adjoint, and thereby restricts to a marked algebra $\col \cO\to \CIm(F_! \yo_\rO)$ to which we associate a Lurie-operad $\rP'$. Since $\Sym \col \rO \to \PSh \Env \rP$ factors through the representables by construction, $\Env(\rP') \subseteq \Env(\rP)$ is a full symmetric monoidal subcategory.
	
	By \cite[Prop. 2.4.11]{haugsengenv} we know $\Env(\rP')^\simeq \simeq \Sym(\rP'^\simeq)$ and similarly for $\rP$, so we obtain a full inclusion of $\E_\infty$-spaces $\Sym(\rP'^\simeq) \subseteq \Sym(\rP^\simeq)$. This restricts to a full inclusion $\rP'^\simeq \subseteq \rP^\simeq$ on free objects, since any map of $\E_\infty$-spaces on $\pi_0$ preserves indecomposables, compare \cite[Lem. 2.1.3]{equifib}. In other words, $F_! \circ \yo_\rO : \col \rO \to \PSh \Env \rP$ factors through $\col \rP$, and thereby induces a map of univalent marked algebras, i.e.\ a map of operads $f$.
\end{proof}

\section{Algebras over enriched operads}
\label{sec:algebras}

In this section, given a $\cV$-enriched operad $\cO$ and $\cM \in \CAlg(\PrV)$, we construct a category of $\cO$-algebras $\Alg_\cO(\cM)$. We discuss limits and colimits of $\cO$-algebras as well free $\cO$-algebras, enhancing several results of \cite[\HAsec{3}]{HA} to the enriched setting.

\begin{constr}
	\label{constr:algebratensoring}
	Given any $\cM \in \CAlg(\Pr)$, by \cref{prop:CAlgPrVclosed} the category \\ \mbox{$\Fun^{\rL, \otimes}_{\cV}(\PSh(\Sym X) \otimes \cV, \cM) \simeq \Fun(X, \cM)$} admits a sifted-colimit-preserving right tensoring by $\Endo^{\rL, \otimes}_{\cV}(\PSh(\Sym X) \otimes \cV)$ via precomposition.
\end{constr}

\begin{defin}
	\label{defin:alg}
	Given a $\cV$-enriched operad $\cO \in \vOp_X(\cV)$ and $\cM \in \CAlg(\Pr)$, we define the category of \emph{$\cO$-algebras} in $\cM$ as
	\[
	\Alg_\cO(\cM) := \RMod_\cO(\Fun(X, \cM)) \; .
	\]
\end{defin}

\begin{obs}
	\label{prop:AlgPShO}
	Since by \cref{thm:EMPrV}, the operadic presheaf category $\PShO(\cO)$ is the Kleisli-object of $\cO$ regarded as a monad on $\PSh \Sym(\col \cO) \otimes \cV$, we can rewrite
	\[
	\Alg_\cO(\cM) \simeq \RMod_\cO(\Fun^{\rL, \otimes}_\cV(\PSh \Sym(\col \cO) \otimes \cV, \cM)) \simeq \Fun^{\rL, \otimes}_\cV(\PShO(\cO), \cM)  \; .
	\] 
\end{obs}

\begin{defin}
	\label{defin:alggeneral}
	More generally for $\cO\in \vOp_X(\cV)$ and $\cP \in \vOp_Y(\cV)$, we define a category $\Alg_{\cO}(\cP)$ of \emph{$\cO$-algebras in $\cP$} as the pullback
	\[
	\Map_{\Spaces}(X, Y) \times_{ \Fun(X, \PShO(\cP)) } \RMod_\cO(\Fun(X, \PShO(\cP))) \; .
	\]
	Note that $\Alg_\cO(\cP)^\simeq \simeq \Map_{\vOp(\cV)}(\cO, \cP)$ by \cref{prop:AlgPShO}.
\end{defin}

\begin{prop}
	\label{prop:algspaces}
	Let $\cO \in \vOp(\Spaces)$ be a $\Spaces$-enriched operad and $X \to \rO$ its associated flagged Lurie-operad via \cref{thm:vOpS}. Then for $\cM\in \CAlg(\Pr)$, we have $\Alg_\cO(\cM) \simeq \Alg_{\rO}(\cM)$, i.e.\ \cref{defin:alg} is compatible with the corresponding notion in \cite{HA}. 
	
	If $\cP \in \vOp(\Spaces)$ with associated flagged Lurie-operad $Y \to \rP$, we similarly have $\Alg_{\cO}(\cP) \simeq \Alg_{\rO}(\rP) \times_{\Fun(X, \underline{\rP})} \Fun(X, Y)$, in particular if $\cP$ is univalent then $\Alg_{\cO}(\cP) \simeq \Alg_{\rO}(\rP)$.
\end{prop}
\begin{proof}
	Considering how mapping spaces in flagged Lurie-operads and marked algebras are calculated, and our definition of $\Alg_{\cO}(\cP)$ in \cref{defin:alggeneral}, it suffices to prove the first statement. But since we already know from \cref{obs:operadicpshspaces} that $\PSh^{\otimes}_{\Spaces}(\cO) \simeq \PSh \Env \rO$ in $\CAlg(\Pr)$, this is immediate from \cref{prop:AlgPShO}.
\end{proof}

%
	
\begin{rem}
	For general $\cV$, we do not expect $\Alg_\cO(\cM)$ to be a $\cV$-enriched category, e.g.\ the category $\CAlg(\Ab)$ of commutative rings is not $\Ab$-enriched. Also, we do not expect there to be a pointwise symmetric monoidal structure on algebras, since constructing an action $\cO(n) \otimes (A \otimes B)^{\otimes n} \to A \otimes B$ for $\cO$-algebras $A, B$ requires a comultiplication $\cO(n) \to \cO(n) \otimes \cO(n)$, e.g.\ through a Hopf operad structure on $\cO$ in the sense of \cite{heinemilnor}. We thank Thomas Blom and Jonas Linßen for pointing these issues out to us. 
	
	The situation is better in case $\cV$ is Cartesian symmetric monoidal (so every operad canonically enhances to a Hopf operad). Then $\vCatV$ is also Cartesian closed symmetric monoidal by \cite[Prop. 2.5.8]{presn}, and by \cref{prop:CMonproduct} there is a unique (presentably) symmetric monoidal structure $\star_\cV$ on $\CAlg(\vCatV)$ such that the symmetric algebra functor is symmetric monoidal. Its internal hom $\Fun^{\otimes, \cV}$, which can be described analogously to the proof of \cref{prop:starmoduleclosed}, allows us to equip \[\Alg_\cO(\cM) \simeq \Fun^{\rL, \otimes}_\cV(\PShV \vEnv_\cV \cO, \cM) \simeq \Fun^{\otimes, \cV}(\vEnv_\cV \cO, \cM)\] with a symmetric monoidal $\cV$-enriched structure. In fact, this can be used to make $\vOp(\cV)$ into a $\cV$-enriched $2$-category as follows:
	
	Similarly to \marked{Prop.}{7.20} once can show that $\vOp(\cV)$ can be described as the full subcategory of marked symmetric monoidal $\cV$-categories $(X \to \cC) \in \Spaces \times_{\CAlg(\CatV)} \Arr(\CAlg(\CatV))$ such that the induced morphism $X \to \PShV(\cC)$ in $\CAlg(\IPrV)$ is a marked algebra. Via this description, $\star_\cV$ induces a \emph{Boardman-Vogt symmetric monoidal structure} on $\cV$-enriched operads\footnote{Compare the construction of the Boardman-Vogt product in \cite{segalification}.}, that is closed with internal hom $\Alg_\cO(\cP)$ and thus induces a $\CatV$-enrichment by \cref{reminder:tensoredenrichment}. 
	
	Proving this involves checking that if $X \to \cC$ and $Y \to \cD$ model enriched operads, then $X \times Y \to \PShV(\cC \star_{\cV} \cD)$ is a marked algebra. We resolve $\cC$ and $\cD$ as geometric realizations of free symmetric monoidal $\cV$-categories on valent $\Spaces$-enriched categories $(X_k \to \rC_k), (Y_k \to \rD_k) \in \vCat(\Spaces)$, such that $\PSh \Sym (X \times Y) \otimes \cV \to \PShV  (\cC \star_{\cV} \cD) $ is the geometric realization of
	\[
	\PSh \Sym (X_k \times Y_k) \otimes \cV \to \PShV(\Sym \FreeV \rC_k \star_{\cV} \Sym \FreeV \rD_k) \simeq \PSh \Sym (\rC_k \times \rD_k) \otimes \cV
	\]
	all of which are internally left adjoint in $\CAlg(\IPrV)$ by \cref{lem:PSymisiL} and colimit-dominant. Also the transition maps in this geometric realization are internally left adjoint by an analogous argument, so we conclude by \cref{lem:iLclosedsiftedcolim} and \cref{lem:cdomiLFS}.
\end{rem}

\begin{defin}
	The \emph{algebra-functor} $\Alg_{(-)}(-) : \vOp(\cV)^{\op} \times \CAlg(\PrV) \to \lCat$ is defined as the composition of the operadic presheaf category functor and the hom-functor of $\CAlg(\IPrV)$:
	\[
	\Alg_{(-)}(-) := \Fun^{\rL, \otimes}_\cV( \PShO(-) , -)
	\]
\end{defin}

\begin{prop}
	\label{prop:propertiesalg}
	For any $\cO \in \vOp(\cV)$ and $\cM \in \CAlg(\PrV)$, the category $\Alg_\cO(\cM)$ is presentable.
	
	The functor $\Alg_{(-)}(\cM) : \vOp(\cV)^{\op} \to \lCat$ preserves small limits and factors through $\Pr$; in particular any morphism $f: \cM \to \cN$ in $\CAlg(\PrV)$ induces an adjunction
	\[
	f \circ - : \Alg_{\cO}(\cM) \rightleftarrows \Alg_{\cO}(\cN) : f^\rR \circ - \; .
	\]
	
	Similarly given a map $F: \cO\to \cP$ in $\vOp(\cV)$, the associated \emph{restriction functor}
	\[ - \circ F_! : \Alg_\cP(\cM) \simeq \iHom_{\CAlg(\PrV)}(\PShO(\cP), \cM) \to \iHom_{\CAlg(\PrV)}(\PShO(\cO), \cM) \simeq \Alg_\cO(\cM)\]
	preserves sifted colimits, and admits a left adjoint \emph{free algebra functor} $\mathrm{Free}_{\cO}^{\cP} : \Alg_\cO(\cM) \to \Alg_\cO(\cP)$.
\end{prop}
\begin{proof}
	Since the operadic presheaf category functor $\PShO: \vOp(\cV) \to \CAlg(\PrV)$ preserves colimits by \cref{cor:PShOcolim} and factors through the internal left adjoints, all of these statements follow from \cref{prop:CAlgPrVclosed}.
\end{proof}

\begin{cor}
	The forgetful functors $\Alg_\cO(\cM) \to \coPShV(\mathrm{Col}_\cV(\cO); \cM)$ to $\cM$-valued enriched copresheaves and $\Alg_\cO(\cM) \to \Fun(\col \cO, \cM)$ to ordinary functors are both monadic and create sifted colimits.
\end{cor}
\begin{proof}
	For preservation of sifted colimits, apply the last part of \cref{prop:propertiesalg} to the counit $\mathrm{Triv}_\cV(\mathrm{Col}_\cV(\cO)) \to \cO$ or $\mathrm{Triv}_\cV \FreeV \col (\cO) \to \cO$, where $\FreeV$ is the free $\cV$-enriched category on an ordinary category. Existence of sifted colimits also follows from \cref{prop:propertiesalg}, so to apply Barr-Beck-Lurie it remains to show conservativity. But the forgetful functor $\Alg_\cO(\cM) \simeq \RMod_\cO(\Fun(\col \cO, \cM)) \to \Fun(\col \cO, \cM)$ is conservative, and factors through $\Alg_\cO(\cM) \to \coPShV(\mathrm{Col}_\cV(\cO); \cM) = \RMod_{\mathrm{Col}_\cV(\cO)}(\Fun(\col \cO,\cM)) \simeq \RMod_{\mathrm{Triv}_\cV \mathrm{Col}_\cV(\cO)}(\Fun(\col \cO,\cM))$ where modules are taken with respect to the composition in $\IPrV$ and $\CAlg(\IPrV)$ respectively.
\end{proof}

\begin{prop}
	For fixed $\cO\in \vOp(\cV)$, the functor $\Alg_{\cO}(-)$ enhances to a $2$-functor $\CAlg(\IPrV) \to \IPr$ that preserves weighted limits.
\end{prop}
\begin{proof}
	We have defined $\Alg_{\cO}(-) := \Fun_\cV^{\otimes}(\PShO(\cO), -) : \CAlg(\IPrV) \to \IlCat$, and any corepresentable $\lCat$-enriched functor preserves weighted limits. It factors through $\IPr$ by \cref{prop:propertiesalg}, and $\IPr \hookrightarrow \ICat$ is closed under weighted limits \HTT{Prop.}{5.5.3.6} and \HTT{Prop.}{5.5.3.13}.
\end{proof}

\begin{rem}
	We deduce that the functor $\chi:= \Alg_{(-)}(-)^\simeq : \CAlg(\PrV) \to \Fun( \vOp(\cV)^{\op},\lSpaces)$ sending $\cM$ to $\chi \cM : \cO \mapsto \left(\Alg_\cO(\cM)\right)^\simeq$ factors through the $\tav$-Ind-completion $\lInd(\vOp(\cV))$. This can be shown to agree with the category of \emph{large $\cV$-operads} $\widehat{\vOp}(\cV)$, which is obtained by mimicking our definition of enriched operads in a larger universe, for the large-presentably symmetric monoidal category $\lInd(\cV)$. 
	
	Thus, we have constructed a functor $\chi : \CAlg(\PrV)  \to  \lInd(\vOp(\cV)) \simeq \widehat{\vOp}(\cV)$ such that \[ \Map_{\widehat{\vOp}(\cV)}(\cO, \chi \cM) \simeq \Map_{\Fun( \vOp(\cV)^{\op},\lSpaces)}(\yo_\cO , \Alg_{(-)}(-)^\simeq) \simeq \Alg_{\cO}(\cM)^\simeq \] functorially in $\cO$ and $\cM$. For $\cM=\cV$ we recover \cite[Thm. 5.12]{haugsengalg}, deducing $\chi \cV$ agrees with their $\overline{\cV}$. Also note that $\col(\chi \cM) \simeq \Map_{\widehat{vOp}(\cV)}(\mathrm{Triv}_\cV(*), \chi \cM) \simeq \Alg_{\mathrm{Triv}_\cV(*)}(\cM)^\simeq \simeq \cM^\simeq$.
\end{rem}

\begin{rem}
	\label{rem:envelopela}
	The envelope functor $\vEnv_\cV: \vOp(\cV) \to \CAlg(\vCat(\cV))$ admits a right adjoint, which sends a symmetric monoidal $\cV$-category $\cC$ to the full suboperad of $\chi \PShV(\cC)$ spanned by the objects of $\cC$, i.e.\ the Cartesian transport along the map of spaces $\yo: \ob \cC \to \ob \chi \PShV(\cC) \simeq \PShV(\cC)^\simeq$:
	\begin{align*}
	&\Map_{\vOp(\cV)}(\cO, \yo^! \chi \PShV(\cC)) \simeq \Map_{\widehat{\vOp}(\cV)} (\cO, \chi \PShV(\cC)) \times_{\Map_{\Spaces}(\ob \cO, \PShV(\cC)^\simeq)} \{\yo\} \simeq \\
	&\quad \simeq \Alg_{\cO}(\PShV(\cC))^\simeq \times_{\Map_{\Spaces}(\ob \cO, \PShV(\cC)^\simeq)} \{\yo\} \simeq \\
	&\quad \simeq \Map_{\CAlg(\PrV)}(\PShV \vEnv_\cV \cO, \PShV \cC) \times_{\Map_{\Spaces}(\ob \cO, \PShV(\cC)^\simeq)} \{\yo\} \simeq \Map_{\CAlg(\vCat(\cV))}(\vEnv_\cV \cO, \cC)
	\end{align*}
	using the description of symmetric monoidal $\cV$-categories in \cref{obs:smVcat} for the last step. 
\end{rem}

We conclude with the following recognition criteria for isomorphisms of enriched operads:

\begin{prop}
	\label{prop:easyrecog}
	For a map of valent $\cV$-enriched operads $f: \cO \to \cP$, the following are equivalent:
	\begin{itemize}
		\item $f$ is an isomorphism,
		\item The underlying map on spaces of colors $\col \cO\to \col \cP$ as well as the induced map $\PShO(\cO) \to \PShO(\cP)$ are equivalences,
		\item The underlying map on spaces of colors $\col \cO\to \col \cP$ as well as the restriction functor $f^* : \Alg_\cP(\cM) \to \Alg_\cO(\cM)$ for any $\cM\in \CAlg(\PrV)$ are equivalences,
		\item The underlying map on spaces of colors $\col \cO\to \col \cP$ as well as the restriction map $f^* : \Alg_\cP(\cM)^\simeq \to \Alg_\cO(\cM)^\simeq$ for any $\cM\in \CAlg(\PrV)$ are equivalences.
	\end{itemize} 
\end{prop}
\begin{proof}
	We have defined $\vOp(\cV)$ as a full subcategory of $\Spaces \times_{\CAlg(\PrV)} \Arr(\CAlg(\PrV))$, and the projections from there to $\Spaces$ together with the target projection to $\CAlg(\PrV)$ are jointly conservative, implying $(1) \Leftrightarrow (2)$. Further,  $\PShO(\cO) \to \PShO(\cP)$ is an equivalence iff the map between corepresentables $\Map_{\CAlg(\PrV)}(\PShO(\cP), -) \to \Map_{\CAlg(\PrV)}(\PShO(\cO), -)$ is an isomorphism, so we conclude the other equivalences by \cref{prop:AlgPShO}.
\end{proof}

%
%
%

\begin{rem}
	\label{rem:cauchy}
	It is not true that the mere existence of a natural isomorphism $\Alg_\cP(-) \simeq \Alg_\cO(-)$, or equivalently an isomorphism $\PShO(\cO) \simeq \PShO(\cP)$, implies that $\cO\simeq \cP$. Consider for instance the trivial operads on a category $\rC$ and on its idempotent completion. However, we can always equip $\PShO(\cO)$ with the maximal marking $(\PShO(\cO)^{\otimes\mathrm{at}, \simeq} \to \PShO(\cO))$ and thereby recover from it the Cauchy-completion $\hat{\cO}$ introduced in \cref{rem:cauchyfissile}. 
	In a followup paper \cite{cauchyoperads}, we further investigate this Cauchy-completion, and thus determine how much information about $\cO$ can be recovered from $\PShO(\cO)$. 
\end{rem}

\appendix
\crefalias{section}{appendix}

\section{The \texorpdfstring{$(\infty, 2)$}{(\unicodeinfty, 2)}-category of presentably symmetric monoidal module categories}

\label{sec:appstar}


\subsection{The tensor product of symmetric monoidal categories}

In this appendix, we construct the $2$-category $\CAlg(\IPrV)$ of presentably symmetric monoidal $\cV$-module categories, and prove the $2$-categorical claims made about it in the beginning of \cref{sec:enrichedoperads}. Our main technical ingredient is a symmetric monoidal structure on $\CAlg(\Cat)$ similar to the Boardman-Vogt-product on symmetric operads: Given symmetric monoidal categories $\rC$ and $\rD$ we consider a tensor product $\rC \star \rD$ such that
\[
\Fun^\otimes(\rC \star \rD, \rE) \simeq \Fun^{\otimes, \otimes}(\rC \times \rD, \rE) \simeq \Fun^\otimes(\rC , \Fun^\otimes(\rD, \rE))
\]
where $\Fun^\otimes(\rD, \rE)$ carries the pointwise tensor product.

\begin{notat}
	For $\rC \in \CAlg(\Cat)$, denote by $\underline{\rC}$ the associated Segal object $\Fin \to \Cat$ and by $\rC^\otimes \to \Fin$ its unstraightening.
\end{notat}

\begin{reminder}
	\label{reminder:pointwise}
	By \HA{Prop.}{3.2.4.3}, if $\rC$ is symmetric monoidal and $\rO$ an operad, then $\Alg_\rO(\rC)$ inherits a symmetric monoidal structure for which the forgetful functor $\Alg_\rO(\rC) \to \rC$ is symmetric monoidal; we refer to this as the \emph{pointwise tensor product}. Concretely, for $\rP$ another operad,
	\[
	\Map_{\Op}(\rP, \Alg_\rO(\rC)) \simeq \Map_{/\Fin}^{\mathrm{inert}}(\mu_!(\rP^\otimes \times \rO^\otimes), \rC^\otimes) \simeq \Map_{/\Fin \times \Fin}^{\mathrm{inert}}(\rP^\otimes \times \rQ^\otimes, \mu^* \rC^\otimes)
	\]
	where $\mu : \Fin \times \Fin \to \Fin$ is the smash product of finite pointed sets, inducing adjoint functors $\mu_! : \Cat_{/\Fin \times \Fin} \rightleftarrows \Cat_{/\Fin} : \mu^*$ of postcomposition and pullback along it. Also $\Map^{\mathrm{inert}}$ denotes those functors that send coCartesian lifts of inert maps to coCartesian morphisms, where a morphisms in $\Fin \times \Fin$ is called inert if both its components are. A morphism $[1] \to \Alg_{\rO}(\rC)$, viewing $[1]$ as a trivial operad, is coCartesian precisely if for any color $o \in \rO$ (or equivalently, any tuple $(o_1, \dots, o_n) \in \rO^\otimes$) the corresponding morphism $[1] \times \{o\} \to \rC^\otimes$ is coCartesian.
\end{reminder}

\begin{lemma}
	\label{lem:pointwisesm}
	Let $\rC,\rD$ be symmetric monoidal categories. Then the full subcategory $\Fun^\otimes(\rD, \rC) \subseteq \Alg_\rD(\rC)$ is closed under the pointwise tensor product, i.e.\ the pointwise product of symmetric monoidal functors is again symmetric monoidal. Moreover, for any symmetric monoidal category $\rE$,
	\[
	\Map_{\CAlg(\Cat)}(\rE, \Fun^\otimes(\rD, \rC)) \simeq \Map_{/\Fin \times \Fin}^{\mathrm{coCart}}(\rE^\otimes \times \rD^\otimes, \mu^* \rC^\otimes) \; .
	\]
\end{lemma}
\begin{proof}
	It suffices to show that if $F \in \Alg_{\rD}(\rC)$ corresponds to a symmetric monoidal functor $\rD \to \rC$ and $[1] \to \Alg_{\rD}(\rC)$ is a coCartesian morphism starting at $F$, then its target also corresponds to a symmetric monoidal functor. Equivalently if $H: \rD^\otimes \times [1] \to \rC^\otimes$ is a natural transformation whose restriction to $\{0\}$ preserves coCartesian morphisms and whose restriction $\{d\} \times [1] \to \rC^\otimes$ marks a coCartesian morphism for any $d \in \rD^\otimes$, then $H_{\rD^\otimes \times \{1\}}$ also preserves coCartesian morphisms. This follows from the cancellation property of coCartesian morphisms \kerodon{01UK}.
\end{proof}

Unlike the forgetful functor, its left adjoint free algebra functor $\Free_{\rO} : \rC \to \Alg_{\rO}(\rC)$ (if it exists) will almost never be symmetric monoidal for the pointwise product. For $\rO$ the commutative operad, one can sometimes find another symmetric monoidal structure for which $\Free_{\E_\infty} =: \Sym$ becomes symmetric monoidal:

\begin{notat}
	For $\rC$ a category with finite products, denote by $\CMon(\rC) \subseteq \Fun(\Fin, \rC)$ the category of \emph{commutative monoid objects} in $\rC$, i.e.\ functors $X : \Fin \to \rC$ such that for any $\underline{n}_+ \in \Fin$ the $n$ active morphisms $\underline{n}_+ \to \underline{1}_+$ induce an equivalence $X(\underline{n}_+) \simeq X(\underline{1}_+)^{\times n}$. By \HA{Prop.}{2.4.2.5}, this is equivalent to commutative algebras in $\rC$ for its Cartesian symmetric monoidal structure.
	
	If $\rC$ is presentable, the full inclusion $\CMon(\rC) \subseteq \Fun(\Fin, \rC)$ is a localization at a small set of morphisms, and thus admits a left adjoint (the \emph{Segalification functor}).
\end{notat}

\begin{prop}[{\cite[Thm. 5.1]{loop}, \cite[Prop. 2.28]{equifib}}]
	\label{prop:CMonproduct}
	Let $(\cC, \otimes) \in \CAlg(\Pr)$ be a presentably symmetric monoidal category. There exists a unique symmetric monoidal structure $\star$ on $\CMon(\cC)$ such that the free algebra functor $\Sym : (\cC, \otimes) \to (\CMon(\cC), \star)$ is symmetric monoidal. 
	
	To construct it, note that by \SAG{Prop.}{C.4.1.9} the functor $\CMon(\Spaces) \otimes - : \Pr \to \Pr$ is a smashing localization so in particular lax symmetric monoidal, meaning that $\CMon(\cC) \simeq \CMon(\Spaces) \otimes \cC$ inherits a commutative algebra structure from $\cC$.
	
	Explicitly (see \cite[Prop. 2.28]{segalification}), this structure is obtained by localizing the Day convolution product $\Day$ on $\Fun(\Fin, \rC)$, for the smash product on $\Fin$ and the symmetric monoidal structure on $\rC$, to the full subcategory on functors $\Fin \to \rC$ satisfying the Segal conditions.
\end{prop}

\begin{ex}
	This recovers some well-known examples:
	\begin{itemize}
			\item If $\rC$ is equipped with the coCartesian symmetric monoidal structure, then since $\Sym$ is a left adjoint the resulting tensor product on $\CMon(\rC)$ must also be the coproduct, hence agrees with the pointwise tensor product by \HA{Prop.}{3.2.4.7}.
			\item For $\Set$ equipped with the Cartesian product, the induced symmetric monoidal structure on $\CMon(\Set)$ is the tensor product of commutative monoids.
			\item For $\Spaces$ with the Cartesian product, $\star$ restricts on grouplike monoids $\CMon^{\mathrm{grp}}(\Spaces) \simeq \Spcn$ to the smash product on connective spectra \cite[Ex. 5.3]{loop}.
			\item For $\Cat_\infty$ the category of $(\infty, \infty)$-categories equipped with the Gray tensor product\footnote{Note this is only monoidal, but the above proposition still applies.}, the induced tensor product on $\CMon(\Cat_\infty) \simeq \mathrm{CatSp}^{\mathrm{cn}}$ is the smash product of connective categorical spectra by \cite[Prop. 4.3.9]{narukismash}.
		\end{itemize}
\end{ex}

\begin{prop}[{c.f.\ \cite[Cor. 2.29]{segalification}}]
	\label{prop:stariHom}
	Equip the category $\CMon(\Cat)$ of symmetric monoidal categories with the tensor product $\star$ induced by the Cartesian product via \cref{prop:CMonproduct}. This symmetric monoidal structure is closed, with internal hom between $\cC$ and $\cD$ the category $\Fun^{\otimes}(\cC, \cD)$ of symmetric monoidal functors equipped with the pointwise tensor product.
\end{prop}
\begin{proof}
	Using the definition of Day convolution $\Day$ as a left Kan extension, we calculate
	\[
	\Map_{\CAlg(\Cat)}(\rC \star \rD, \rE) \simeq \Map_{\Fun(\Fin, \Cat)}(\underline{\rC} \Day \underline\rD, \underline\rE) \simeq \Map_{\Fun(\Fin \times \Fin, \Cat)}(\underline\rC \times \underline\rD, \underline\rE \circ \mu) \; .
	\]
	Unstraightening $\Fun(\Fin \times \Fin, \Cat) \simeq \coCart_{/\Fin \times \Fin}$ identifies this with
	\[
	\Map_{/\Fin \times \Fin}^{\mathrm{coCart}}(\rC^\otimes \times \rD^\otimes, \mu^* \rE^\otimes)
	\]
	which by \cref{lem:pointwisesm} is equivalent to $\Map_{\CAlg(\Cat)}(\rC, \Fun^\otimes(\rD, \rE))$.
\end{proof}

\begin{cor}
	\label{cor:funotimesfunctorial}
	The assignment $(\rC , \rD) \mapsto \Fun^\otimes (\rC, \rD)$ assembles to a functor
	\[
	\Fun^\otimes : \CAlg(\Cat)^{\op} \times \CAlg(\Cat) \to \CAlg(\Cat) \; .
	\]
\end{cor}

\begin{notat}
	Recall from \HA{Lem.}{4.8.3.15} that the forgetful functor $\RMod(\Cat) \to \Alg(\Cat)$ is a coCartesian fibration of symmetric monoidal categories in the sense of \marked{Cor.}{B.7} with respect to the Cartesian product, so it unstraightens into a lax symmetric monoidal functor $\RMod_{(-)}(\Cat) : \Alg(\Cat) \to \lCat$. In particular if we plug in a commutative algebra $\rV \in \CAlg(\Alg(\Cat)) \simeq \CAlg(\Cat)$, the resulting category $\RMod_\rV(\Cat)$ inherits a symmetric monoidal structure, called the \emph{relative tensor product} $\times_\rV$.
\end{notat}

\begin{defin}
	Let $\rV \in \CAlg(\Cat)$, then we denote by $\star_\rV$ the unique symmetric monoidal structure on $\CMon(\RMod_\rV(\Cat))$ such that the symmetric algebra functor $\Sym : \RMod_\rV(\Cat) \to \CMon(\RMod_\rV(\Cat))$ is symmetric monoidal for the relative tensor product $\times_\rV$ on $\RMod_\rV(\Cat)$, compare \cref{prop:CMonproduct}.
\end{defin}

\begin{warning}
	\label{warning:rigcats}
	Despite the notation, $\star_\rV$ is \emph{not} obtained by straightening the coCartesian fibration of symmetric monoidal categories 
	$(\CMon(\RMod(\Cat)), \star_{\RM}) \to (\CMon(\Cat), \star)$ in which both sides are equipped with symmetric monoidal structures rendering the respective symmetric algebra functors symmetric monoidal. Such an unstraightening would instead produce a lax symmetric monoidal functor $\CMon(\RMod_{(-)}(\Cat)) : (\CMon(\Cat), \star) \to (\lCat, \times)$ sending commutative algebras in $(\CMon(\Cat),\star)$, which are known as \emph{rig categories}, to symmetric monoidal categories.
\end{warning}

\begin{prop}
	\label{prop:starmoduleclosed}
	The symmetric monoidal structure $\star_\rV$ on $\CMon(\RMod_\rV(\Cat))$ is closed. Its internal Hom $\Fun^\otimes_{\rV}(\rM, \rN) \in \CMon(\RMod_\rV(\Cat))$ equips the category of symmetric monoidal $\rV$-linear functors between $\rM, \rN \in \CMon(\RMod_\rV(\Cat))$ with a pointwise symmetric monoidal structure. Moreover, there is a canonical equivalence of symmetric monoidal $\rV$-module categories
	\[
	\Fun_{\rV}^\otimes(\rM,\Fun^\otimes_{\rV}(\rN,\rP)) \simeq \Fun_{\rV}^\otimes(\rM \star_\rV \rN,\rP) \; .
	\]
\end{prop}
\begin{proof}
	We begin by calculating on mapping spaces:
	\begin{align*}
			&\Map_{\CMon(\RMod_\rV(\Cat))}(\rM \star_{\rV} \rN, \rP) \simeq \Map_{\Fun(\Fin \times \Fin, \RMod_\rV(\Cat))}( \underline{\rM} \times \underline{\rN}, \underline{\rP} \circ \mu ) \simeq \\
			&\quad \simeq \Map_{\Fun(\Fin, \RMod_\rV(\Cat))} \left(\underline{\rM}, \oint_{\underline{n}_+ \in \Fin} \Fun_\rV(\underline{N}(\underline{n}_+), \underline{P}(\underline{n}_+ \wedge -)) \right)
		\end{align*}
	where $\Fun_\rV$ denotes the internal hom in $\RMod_\rV(\Cat)$, compare \cite[Prop. 3.89]{heine} or \cite[§4.2]{bienriched}, and $\oint$ denotes an end. We need to show that the right argument is a Segal object in $\RMod_\rV(\Cat)$, which follows since this internal hom commutes with products in the second argument and $P(\underline{n}_+ \wedge -)$ is a Segal object by the proof of \cite[Prop. 2.28]{segalification}. 
	
	This equivalence of spaces enhances to the stated equivalence in $\CMon(\RMod_\rV(\Cat))$ by associativity of $\star_\rV$.
\end{proof}

\begin{rem}
	For any $\cC \in \CAlg(\Pr)$, we have defined $\CMon(\cC) \simeq \CMon(\Spaces) \otimes \cC \in \CAlg(\Pr)$ so its symmetric monoidal structure $\star_{\cC}$ is always closed. In fact, following the proof of \cref{prop:starmoduleclosed} gives an explicit formula for the internal hom.
\end{rem}

\begin{obs}
	\label{obs:functorialitystarV}
	Since $\RMod_{(-)}(\Cat): \Alg(\Cat) \to \lCat$ is lax symmetric monoidal, a symmetric monoidal functor $f: \rV \to \rW$ induces a symmetric monoidal functor $f_{\mathrm{ext}} : (\RMod_\rV(\Cat), \times_\rV) \to (\RMod_\rW(\Cat), \times_\rW)$, and thus applying $\CMon(\Spaces) \otimes -$ we obtain a symmetric monoidal functor
	\[
	f_{\mathrm{ext}} : \CMon(\RMod_\rV(\Cat)) \to \CMon(\RMod_\rW(\Cat)) \; .
	\]
\end{obs}

\subsection{The \texorpdfstring{$(\infty, 2)$}{(\unicodeinfty, 2)}-category of symmetric monoidal categories}

Finally, we use the symmetric monoidal structure $\star$ to construct the $2$-category $\CAlg(\ICat)$.

\begin{notat}
	We will model $2$-categories as $\Cat$-enriched categories, and freely use the notions of weighted colimits, enriched functor categories and enriched adjunctions in this context, compare \cite{heineweighted}, \cite{bienriched}, \cite{stefanich}. Recall that an enriched category admits, and an enriched functor preserves, all weighted colimits if it admits/ preserves tensorings and conical colimits \cite[Thm. 5.6.1]{stefanich}, \cite[Prop. 3.64]{heineweighted}. Similarly weighted limits are generated by cotensorings and conical limits. By a \emph{conical colimit} we mean a weighted colimit with terminal weight; loosely speaking a conical limit of a functor $p: \rK \to \bC$ from a category $\rK$ $2$-category to a $2$-category $\bC$ is an object $x \in \bC$ such that for any $c \in \bC$, we have $\Hom_{\bC}(c, x) \simeq \lim_{k \in \rK} \Hom_{\bC}(c, p(k))$. Put differently, it is a limit in the underlying $1$-category of $\bC$ whose universal property enhance to the hom-categories in $\bC$.
\end{notat}

\begin{reminder}
	\label{reminder:tensoredenrichment}
	Let $\cV \in \Alg(\Pr)$. To any presentable $\cV$-module $\cM \in \PrV$, by \cite[Thm. 1.2]{heine} we can associate a (univalent) $\cV$-enriched category $\chi \cM$ with large space of objects $\cM^\simeq$ and hom-objects determined by the internal hom on $\cM$. This assembles into a subcategory inclusion $\chi: \PrV \to \lCat_{\text{large space}}(\cV)$, with image those $\cV$-categories with large space of objects that admit all weighted colimits and whose underlying category is presentable, and those enriched functors that preserve weighted colimits.
	
	In particular, to any presentable $\Cat$-module, we may associate a locally small $2$-category admitting all $\Cat$-weighted colimits whose maximal sub-$1$-category is presentable; they are called \emph{presentable $2$-categories} in \cite{secondaryK} 
	(to be distinguished from presentable $2$-categories in the sense of \cite{presn}). By \cite[Lem. 3.69]{heineweighted}, these automatically admit all $\Cat$-weighted limits as well. 
\end{reminder}

\begin{reminder}[{\cite[Lem. 2.90]{heineweighted}}]
	\label{reminder:twoadj}
	Let $\bC$ and $\bD$ be $2$-categories with underlying catgories $\rC$ and $\rD$. Then an adjunction $F: \rC \rightleftarrows \rD : G$ enhances to a \emph{$2$-adjunction}, i.e.\ $\Cat$-enriched adjunction if either
	\begin{itemize}
		\item $\rC$ admits $\Cat$-tensorings and $F$ preserves them,
		\item $\rD$ admits $\Cat$-cotensorings and $G$ preserves them.
	\end{itemize}
\end{reminder}

\begin{defin}
	Restricting scalars along the symmetric monoidal functor $\Sym : (\Cat, \times) \to (\CAlg(\Cat), \star))$, we equip the category $\CAlg(\Cat)$ with a presentable $\Cat$-tensoring. The presentable $2$-category of symmetric monoidal categories $\CAlg(\ICat)$ is associated to this presentable $\Cat$-module via \cref{reminder:tensoredenrichment}.
	
	Similarly, we define the presentable $2$-category of symmetric monoidal $\rV$-module categories $\CAlg(\RMod_\rV(\ICat))$ by restriction of scalars along the symmetric monoidal functor $(\Cat, \times) \to (\RMod_\rV(\Cat), \times_\rV) \to (\CAlg(\RMod_\rV(\Cat)), \star_\rV)$\footnote{Another approach would be to define $\CAlg(\RMod_\rV(\ICat))$ as the slice $2$-category $\CAlg(\ICat)_{\rV/}$.}. 
\end{defin}

\begin{obs}
	Using the expression for the internal hom for the $\star$-product \cref{prop:stariHom}, we learn that the hom-categories of $\CAlg(\ICat)$ are given by the categories of symmetric monoidal functors $\Fun^\otimes(\rM, \rN)$. Similarly by \cref{prop:starmoduleclosed}, the hom-categories of $\CAlg(\RMod_\rV(\ICat))$ are given by $\Fun^{\otimes}_{\rV}(\rM, \rN)$.
\end{obs}


\begin{obs}
	\label{obs:CAlgCattwofunctors}
	Since they are symmetric monoidal, the functors \[\ICat \overset{- \otimes \rV}{\to} \RMod_\rV(\ICat) \overset{\Sym}{\to} \CAlg(\RMod_\rV(\ICat))\] are compatible with the respective $\Cat$-tensorings, i.e.\ maps in $\CAlg(\Pr_{\Cat})$ so they enhance to symmetric monoidal $2$-functors. Also by \cref{obs:functorialitystarV}, given a symmetric monoidal functor $f: \rV \to \rW$ we obtain a symmetric monoidal $2$-functor $f_{\mathrm{ext}} : \CMon(\RMod_\rV(\ICat)) \to \CMon(\RMod_\rW(\ICat))$. All of these are left adjoint $1$-functors, so by \cref{reminder:twoadj} their right adjoint forgetful functors enhance to $2$-functors $\CAlg(\RMod_\rV(\ICat)) \to \RMod_{\rV}(\ICat) \to \ICat$ and $f^{\mathrm{res}}$ as well, and the respective adjunctions enhance to $2$-adjunctions.
\end{obs}

\begin{obs}
	\label{obs:cotensoring}
	Recall that for $\rK \in \Cat$ and $\rC \in \CAlg(\Cat)$, we can equip $\Fun(K, \rC)$ with the {pointwise  symmetric monoidal structure}. This satisfies the universal property of a $\Cat$-cotensoring in $\CAlg(\ICat)$, since
	\[
	\Fun^{\otimes} (\rD, \Fun(\rK, \rC)) \simeq \Nat(\underline{\rD}, \Fun(\rK, \underline{\rC}(-))) \simeq \Fun(\rK, \Nat(\underline{\rD}, \underline{\rC})) \simeq \Fun(\rK, \Fun^{\otimes}(\rD, \rC)) \; .
	\]
	Similarly for $\rM \in \CAlg(\RMod_\rV(\Cat))$ the pointwise symmetric monoidal $\rV$-module structure\footnote{To be precise, this can be defined by applying $\Fun(\rK, -)$ to the $\Fin \times \RM^\otimes$-Segal-object describing $\rM$, and restricting scalars along the diagonal $\Delta: \rV \to \Fun(\rK, \rV)$ afterwards. See also \cite[Not. 3.69]{heine}.} exhibits $\Fun(\rK, \rM)$ as a cotensoring in $\CAlg(\RMod_\rV(\ICat))$. Note that the forgetful functors to $\ICat$ preserve the cotensorings, as they should being $2$-right adjoints.
\end{obs}

\begin{prop}
	\label{prop:localsiftedcat}
	The forgetful $2$-functors $\CAlg(\ICat) \to \ICat$ and $\CAlg(\RMod_\rV(\ICat))) \to \ICat$ locally preserve sifted colimits. To be precise, for $\rC, \rD \in \CAlg(\RMod_\rV(\Cat)))$, the forgetful functor $\Fun^{\otimes}(\rC, \rD) \to \Fun(\rC, \rD)$ preserves sifted colimits if they exist; which they do if $\rD$ admits sifted colimits. 
\end{prop}
\begin{proof}
	For any symmetric monoidal categories $\rC, \rD \in \CAlg(\Cat)$, the forgetful functor $\Fun^{\otimes}(\rC, \rD) \to \Fun^{\mathrm{lax}\otimes}(\rC, \rD) \to \Fun(\rC, \rD)$ factors through lax symmetric monoidal functors. The second functor preserves sifted colimits by \HA{Prop.}{3.2.3.1} since $\Fun^{\mathrm{lax}\otimes}(\rC, \rD) \simeq \Alg_{/ \rC}(\rC \times_{\Fin} \rD)$, in particular sifted colimits of lax symmetric monoidal functors are calculated pointwise\footnote{Alternatively, use that lax symmetric monoidal functors are algebras for the Day convolution product.}. But from this pointwise description we immediately conclude that strong monoidal functors are closed under sifted colimits.
	
	The case of $\CAlg(\RMod_\rV(\ICat)))$ is similar: Either embed into $\CAlg(\RMod(\ICat))$ and write out the same argument with $\E_\infty$ replaced by the Boardman-Vogt-product $\E_\infty \otimes_{\mathrm{BV}} \RM$, using that sifted colimits are weakly contractible so we never leave the fiber of $\id_\rV$ in $\Fun^{\otimes}(\rV, \rV)$. Alternatively, use a similar argument as \HA{Lem.}{4.8.4.12}.
\end{proof}

\subsection{Presentably symmetric monoidal categories}

Next, we consider presentable variants of the $2$-categories $\CAlg(\ICat)$ and $\CAlg(\RMod_\rV(\ICat))$ that we have constructed.

\begin{defin}
	Given $\cV \in \CAlg(\Pr)$, let $\CAlg(\IPrV) \hookrightarrow \CAlg(\RMod_\cV(\IlCat))$ be the sub-$2$-category on presentably symmetric monoidal $\cV$-modules and symmetric monoidal $\cV$-linear functors preserving colimits. Similarly we define $\CAlg(\RMod_\cV(\ICatcolim))$. 
\end{defin}

\begin{obs}
	\label{obs:CAlgPrVemb}
	The full inclusion $\CAlg(\Cat) \subseteq \Fun(\Fin , \Cat)$ enhances to a fully faithful $2$-functor: It is lax $\Cat$-linear since its left adjoint Segalification functor $\Fun(\Fin, \Cat) \to \CAlg(\Cat)$ is $\Cat$-linear for the pointwise $\Cat$-tensoring on either side. To see this, note that for $\rC, \rD \in \CAlg(\Cat)$ and $\rK \in \Cat$ we have
	\[
	\Map_{\Fun(\Fin, \Cat)}(\underline{\rC} \times \rK , \underline{\rD}) \simeq \Map_{\Cat}(\rK, \Nat(\underline{\rC}, \underline{\rD})) \simeq \Map_{\CAlg(\Cat)}(\Sym \rK, \Fun^\otimes(\rC, \rD))
	\]
	so the Segalification of $\underline{C} \times \rK$ must agree with $\rC \star \Sym \rK$. Similarly, the full inclusions $\CAlg(\RMod(\Cat)) \subseteq \Fun(\Fin, \RMod(\Cat)) \subseteq \Fun(\Fin \times \RM^\otimes , \Cat)$ enhance to fully faithful $2$-functors.
	
	This means that we could have alternatively defined $\CAlg(\RMod_\cV(\Catcolim))$ as the locally full sub-$2$-category of the fiber
	\[ \Fun(\Fin \times \RM^\otimes, \lCat) \times_{\Fun(\Fin \times \Ass^\otimes, \lCat)} \{\underline{\cV}\}\] 
	on Segal-objects whose structure maps are cocontinuous. While this description is more straightforward, it would make the study of universal constructions in it much more difficult.
\end{obs}

\begin{lemma}
	\label{lem:conicallimitsub}
	Let $\bC \hookrightarrow \bD$ be a locally full sub-$2$-category, and $p: \rK^{\triangleleft} \to \bC$ a diagram that is a conical limit diagram in $\bC$ and a limit diagram in $\bD$. Then $p$ is automatically a conical limit diagram in $\bD$. Similarly for conical colimits.
\end{lemma}
\begin{proof}
	Since $\bC$ is a locally full subcategory of $\bD$, for any $c, c' \in \bC$ the square of hom-categories and hom-spaces
	\[
	\begin{tikzcd}
		\Hom_{\bC}(c, c') \arrow[r, hook] \arrow[d, hook] & \Hom_{\bC}(c, c')^{\simeq} \arrow[d, hook] \\
		\Hom_{\bD}(c, c') \arrow[r, hook] & \Hom_{\bD}(c, c')^{\simeq}
	\end{tikzcd}
	\]
	is a pullback as the right vertical map is a full subcategory inclusion. Since we know $p$ satisfies the correct universal property for $\Hom_{\bC}(-, -)^\simeq$ and $\Hom_\bD(-, -)$, it does so for the pullback $\Hom_{\bC}(-, -)$ and we conclude.
\end{proof}

\begin{prop}
	\label{prop:CAlgPrVweightedlimits}
	The $2$-category $\CAlg(\IPrV)$ admits $\Cat$-weighted limits, which are created by the forgetful $2$-functors $\CAlg(\IPrV) \hookrightarrow \CAlg(\RMod_\rV(\IlCat)) \to \CAlg(\IlCat) \to \IlCat$.
\end{prop}
\begin{proof}
	Note that $\CAlg(\PrV) \hookrightarrow \CAlg(\RMod_\cV(\lCat))$ is closed under ordinary limits, since $\Pr \hookrightarrow \lCat$ is closed under limits by \HTT{Prop.}{5.5.3.13} and the forgetful functors to there create limits. But conical $2$-limits agree with the underlying ordinary limits if they exist, so since they exist in the presentable $2$-category $\CAlg(\RMod_\rV(\lCat))$ they also exist in the sub-$2$-category on presentable modules by \cref{lem:conicallimitsub}.
	
	It remains to consider cotensorings. From \cref{obs:cotensoring} we know $\CAlg(\RMod_\rV(\ICat))$ admits $\Cat$-cotensorings, which are created by the forgetful functor to $\ICat$. It remains to show that for $\cM \in \CAlg(\PrV)$ and $\Cat$, the cotensoring $\Fun(K, \cM)$ is still in $\CAlg(\PrV)$; and that given a functor $\rK \to \Fun^{\rL, \otimes}_\cV(\cM, \cN)$ the induced symmetric monoidal $\cV$-linear functor $\cM \to \Fun(\rK, \cN)$ preserves colimits. The latter statement is obvious since colimits in $\Fun(\rK, \cN)$ are calculated pointwise. Also $\Fun(\rK, \cN)$ is presentable by \HTT{Porp.}{5.5.3.6}, and since its symmetric monoidal structure and $\cV$-tensoring are defined pointwise they are still compatible with colimits.
\end{proof}

\begin{rem}
	In fact, $\CAlg(\IPrV)$ also admits $\Cat$-weighted colimits. For conical sifted colimits we argue as for limits by embedding $\CAlg(\IPrV) \hookrightarrow \CAlg(\RMod_\cV(\ICat))$; and since any object is a geometric realization of free objects $(\Sym^\otimes \cN) \otimes \cV$ with $\cN \in \Pr$ for conical coproducts it suffices to describe coproducts of free objects, which are given by $(\Sym^\otimes (\cN_1 \oplus \cN_2)) \otimes \cV$ using that $\Sym^\otimes$ is a $2$-functor as we show in \cref{cor:freeforgetfulfunctors}. Similarly it suffices to ensure the existence of $\Cat$-tensorings of free objects, which for $\rK \in \Cat$ are given by $(\Sym^\otimes (\PSh \rK \otimes \cN)) \otimes \cV$.
\end{rem}

\begin{warning}
	Unlike $\CAlg(\RMod_\rV(\ICat))$, the large $2$-category $\CAlg(\IPrV)$ is not presentable, since it is not locally small. It also is not presentable in the sense of \cite{presn}, since precomposition does not preserve colimits \cref{prop:CAlgPrVclosed}.
\end{warning}

\begin{warning}
	The symmetric monoidal structure $\star$ on $\CAlg(\ICat)$ suggests the existence of a symmetric monoidal structure $\ostar$ on $\CAlg(\IPr)$ such that for $\cM, \cN, \cP \in \CAlg(\Pr)$ we have
	\[
	\Fun^{\otimes, \rL} (\cM \ostar \cN,\cP) \simeq \Fun^{\otimes,\otimes, \rL,  \rL} (\cM\times\cN, \cP) \; .
	\]
	If this existed, it could be expressed as the full subcategory of the free cocompletion $\PSh^{\tav\mathrm{-rex}}(\cM \star \cN)$ on presheaves $\cM \star \cN \to \lSpaces$ such that precomposition with the canonical inclusions $\cM \simeq \cM \star \Sym(*) \to \cM \star \cN$ and $\cN \to \cM \star \cN$ yields functors preserving small colimits. However, we expect that this localization of $\PSh^{\tav\mathrm{-rex}}(\cM \star \cN)$ is generally not compatible with the symmetric monoidal structure, and hence merely forms a lax symmetric monoidal category.
	
	Further, even if $\ostar$ existed, it would not be closed because the pointwise tensor product of colimit-preserving functors generally only preserves sifted colimits.
\end{warning}

\subsection{Eilenberg-Moore objects and free \texorpdfstring{$2$}{2}-functors}
\label{subsec:EM}

As a special case of weighted limits, $\CAlg(\IPrV)$ admits Eilenberg-Moore objects in the sense of \cite{haugsengmonads}, \cite{heinemonadicity}, \cite{stefanich}, \cite{stockall2025large}. As a reminder, for $\bD$ a $2$-category and $d, d' \in \bD$, the category $\Hom_\bD(d,d')$ is a left module over $\Endo_\bD(d')$ and a right module over $\Endo_\bD(d)$ via post- and precomposition respectively. Given a monad $T \in \Alg(\Endo_\bD(d'))$, a $1$-morphism $f: x \to d'$ is said to exhibit $x$ as \emph{Eilenberg-Moore object} of $T$ if composition with it lifts to an equivalence
\[
\Hom_\bD(d, x) \simeq \LMod_T( \Hom_\bD(d, d') )
\]
for any $d \in \bD$. In other words, $x$ represents the $\Cat$-enriched presheaf $\LMod_T(\Hom_\bD(-, d'))$. In this case $f$ is called a \emph{monadic $1$-morphism}. Similarly for $S \in \Alg (\Endo_\bD(d))$, a morphism $g: d \to y$ exibits $y$ as a \emph{Kleisli object} of $S$ if composition with it lifts to an equivalence
\[
\Hom_\bD(y, d') \simeq \RMod_S( \Hom_\bD(d, d') )
\]
for any $d' \in \bD$. See the proof of \cref{prop:locgr} for another perspective. 

\begin{prop}
	\label{prop:locgr}
	Let $\C,\bD$ be $2$-categories that locally admit geometric realizations (in particular, their compositions preserve geometric realizations in each argument). Given a monad $T$ on an object $c \in \C$, then an object $e \in \C$ is an Eilenberg-Moore object of $T$ iff it is a Kleisli object of $T$. A $1$-morphism $e \to c$ is monadic iff it admits a right adjoint $c \to e$ exhibiting $e$ as a Kleisli object.
	
	If $F: \C \to \bD$ is a $2$-functor which locally preserves geometric realizations, then also $Fe$ is an Eilenberg-Moore and Kleisli object of the monad $FT$ in $\bD$. In particular, $F$ preserves monadic $1$-morphisms.
\end{prop}
\begin{proof}
	Follows from lax matrix calculus applied to $1 \times 1$-matrices, see \cite{laxadd} and \cite{angus}. We spell out a simplified version of their argument: To any $2$-category $\C$ that locally admits geometric realizations, we can associate its Morita $2$-category $\Morita(\C)$ whose objects are pairs $(c, T)$ of an object $c \in \C$ and a monad $T \in \Endo_\C(c)$, with morphism categories given by
	\[
	\Hom_{\Morita(\C)}((c, T) , (c', T')) \simeq \BMod{T}{T'}(\Hom_\C(c, c'))
	\]
	and composition given by relative tensor products (explaining our assumption that geometric realizations must exist). See \cite{blomstraight} for an explicit construction. 
	
	This construction is functorial in $2$-functors that locally preserve geometric realizations, so we obtain a $2$-functor $\Morita(\C) \to \Morita(\bD)$ sending $(c, T) \mapsto (Fc, FT)$. Also, $e \in \C$ is an Eilenberg-Moore object of a monad $T$ iff there exists an isomorphism $(c, T) \simeq (e, \id_e)$ in $\Morita(\C)$ iff $e$ is a Kleisli-object of $T$; this becomes apparent by applying $\Hom_{\Morita(\C)}((d, \id_d) , -)$ and $\Hom_{\Morita(\C)}(-, (d, \id_d))$ for arbitrary $d \in \C$ to this isomorphism which by the above description of morphism categories in $\Morita(\C)$ recovers the defining universal properties of Eilenberg-Moore and Kleisli-objects respectively.
	
	Finally a $1$-morphism $f: c' \to c$ in $\bC$ is monadic iff the associated $1$-morphism $(c' , \id_{c'}) \to (c, \id_c)$ in $\Morita(\C)$ is isomorphic to the canonical morphism $(c, T) \to (c, \id_c)$ induced by regarding $T$ as a left $T$-module; while a morphism $g: c \to c'$ exhibits $c'$ as a Kleisli-object if it is isomorphic to the morphism $(c, \id_c) \to (c, T)$ regarding $T$ as a right $T$-module (which is right adjoint to the former).
\end{proof}

\begin{cor}
	\label{cor:reflgr}
	Let $\C,\bD$ be $2$-categories that locally admit geometric realizations and $F: \C \to \bD$ a conservative $2$-functor which locally preserves geometric realizations. Then $F$ creates Eilenberg-Moore objects.
\end{cor}

\begin{lemma}
	\label{lem:fgtlocallycocont}
	The hom-categories of $\CAlg(\IPrV)$ admit sifted colimits, which are calculated pointwise. The forgetful $2$-functor $\CAlg (\RMod_\cV (\IPr)) \to \IlCat$ is conservative and locally preserves sifted colimits. Similarly for $\CAlg(\RMod_\cV(\ICatcolim))$.
\end{lemma}
\begin{proof}
	Conservativity is clear from \HA{Prop.}{3.2.2.5} and \HA{Prop.}{4.2.3.2} as it only depends on the underlying $1$-functor. 
	Then, the statement follows from the analogous statement \cref{prop:localsiftedcat} in $\CAlg(\RMod_\cV(\lCat))$, since colimit-preserving functors are closed under pointwise colimits.
\end{proof}

\begin{prop}
	\label{prop:CAlgPrVclosed}
	The $2$-category $\CAlg(\PrV)$ locally admits sifted colimits, i.e.\ is enriched over $\lCat {}^{\mathrm{sift}}$. In fact, its hom-categories are presentable, and for
	$\cM_1, \cM_2, \cM_3 \in \RMod_\cV(\Catcolim)$ and a morphism $F: \cM_2 \to \cM_3$ the postcomposition functor
	\[
	F \circ - : \Fun^{\rL, \otimes}_\cV(\cM_1, \cM_2) \to \Fun^{\rL, \otimes}_\cV(\cM_1, \cM_3)
	\]
	preserves small colimits so it admits a right adjoint $\iNat_{\cM_3}(F, -)$\footnote{Classically, $\iNat_{\cM_3}(F, G)$ is called the \emph{right Kan lifting} of $G$ along $F$; in our case it is closely related to operadic weighted limits.}. 
	Similarly, the $2$-category $\CAlg(\RMod_\cV(\ICatcolim))$ locally admits sifted colimits, its hom-categories are cocomplete and postcomposition preserves small colimits.
\end{prop}
\begin{proof}
	Composition preserves sifted colimits in both arguments, since we already know by \cref{lem:fgtlocallycocont} that they exist and are calculated pointwise. 
	
	For the case of general colimits, assume that $\cM_1 = (\Sym^\otimes \cP) \otimes \cV$ for some $\cP \in \Pr$. Then $\Fun^{\rL, \otimes}_\cV(\cM_1, \cM_2) \simeq \Fun^\rL(\cP , \cM_2)$ is presentable, and $F \circ -$ agrees with the postcomposition
	\[
	F \circ - : \Fun^{\rL}(\cP, \cM_2) \to \Fun^{\rL}(\cP, \cM_3)
	\]
	which preserves colimits since it admits a right adjoint $F^\rR \circ -$. A general $\cM_1$ can be resolved as a geometric realization of $(\Sym^\otimes \cP_n) \otimes \cV$ in $\CAlg(\Pr)$, and on each component $F \circ -$ admits a right adjoint $F^\rR \circ -$.
	Hence, the postcomposition functor
	\[
	F \circ - : \Fun^{\rL, \otimes}_\cV(\cM_1, \cM_2) \simeq \lim_{[n] \in \Delta^\op} \Fun^{\rL}(\cP_n, \cM_2) \to \lim_{[n] \in \Delta^\op} \Fun^{\rL}(\cP_n, \cM_n) \simeq \Fun^{\rL, \otimes}_\cV(\cM_1, \cM_3)
	\]
	is induced by a transformation between limit diagrams in $\Pr$. So we conclude since $\Pr \hookrightarrow \lCat$ is closed under limits by \HTT{Prop.}{5.5.3.13}, and similarly $\Catcolim \hookrightarrow \lCat$ is closed under limits.
\end{proof}

\begin{warning}
	However, $\CAlg(\IPrV)$ is not left closed, i.e.\ internal right Kan extensions generally do not exist. For instance, the composition product from \cref{cor:compositionproduct} only preserves sifted colimits in the left argument, not coproducts.
\end{warning}

\begin{theorem}
	\label{thm:EMPrV}
	The categories $\CAlg(\RMod_\cV(\Pr))$ and $\CAlg(\RMod_\cV(\Catcolim))$ admit Eilenberg-Moore objects, which automatically satisfy the universal property of Kleisli-objects as well. The forgetful $2$-functors to $\IlCat$ create Eilenberg-Moore objects and monadic $1$-morphisms. 
\end{theorem}
\begin{proof}
	Combine \cref{prop:locgr}, \cref{lem:fgtlocallycocont} and \cref{prop:CAlgPrVclosed}.
\end{proof}

Finally, we construct several $2$-adjunctions involving $\CAlg(\IPrV)$.

\begin{prop}
	\label{prop:freeforgetfulfunctors}
	The left adjoint functors
	\[
	\lCat \overset{\Sym}{\longrightarrow} \CAlg(\lCat) \overset{\PSh^{\tav\text{-rex}}}{\longrightarrow} \CAlg(\Catcolim) \overset{- \otimes \cV}{\longrightarrow} \CAlg(\RMod_\cV(\Catcolim))
	\]
	enhance to $2$-functors, giving rise to a sequence of $2$-adjunctions
	\[
	\widehat{\ICat} \rightleftarrows \CAlg(\widehat{\ICat}) \rightleftarrows \CAlg(\IPr) \rightleftarrows \CAlg(\IPrV)
	\]
	with the forgetful $2$-functors from \cref{prop:CAlgPrVweightedlimits}. Similarly, the adjunctions
	\[
	\Sym^{\otimes} : \Catcolim \rightleftarrows \CAlg(\Catcolim) \; , \quad \Sym^{\otimes_\cV} : \RMod_\cV(\Catcolim) \rightleftarrows \CAlg(\RMod_\cV(\Catcolim))
	\]
	enhance to $2$-adjunctions between $2$-functors.
\end{prop}
\begin{proof}
	By \cref{reminder:twoadj}, it suffices to prove that the $2$-categories $\CAlg(\IlCat), \CAlg(\ICatcolim)$ and $\CAlg(\RMod_\cV(\ICatcolim))$ are $\lCat$-cotensored and the forgetful functors preserve $\ICat$-cotensors, which follows from \cref{prop:CAlgPrVweightedlimits}.
\end{proof}
\begin{rem}
	An alternative proof that $\Sym$ and $\Sym^{\otimes}$ are $2$-functors, using flipped Segal categories, can be found in \cite[Prop. 2.2.15]{blomchain}.
\end{rem}

\begin{cor}
	\label{cor:freeforgetfulfunctors}
	There is a sequence of $2$-functors
	\[
	\IlCat \overset{\Sym}{\longrightarrow} \CAlg(\IlCat) \overset{\PSh}{\longrightarrow} \CAlg(\IPr) \overset{- \otimes \cV}{\longrightarrow} \CAlg(\IPrV)
	\]
	enhancing the respective $1$-functors, such that for $\rC \in \Cat$ and $\cM \in \CAlg(\PrV)$,
	\[
	\Fun^{\otimes, \rL}_\cV (\PSh ( \Sym - ) \otimes \cV , \cM ) \simeq \Fun(\rC, \cM) \; .
	\]
\end{cor}

\begin{prop}
	\label{prop:eosstar}
	Given a $f: \cV \to \cW$ in $\CAlg(\Pr)$, the extension-of-scalars and restriction-of-scalars along it enhance to a $2$-adjunction
	\[
	f_{\mathrm{ext}} : \CAlg(\PrV) \rightleftarrows \CAlg(\PrW) : f^{\mathrm{res}} \; .
	\]
\end{prop}
\begin{proof}
	We already know from \cref{obs:CAlgCattwofunctors} that \[f^{\mathrm{res}} : \CAlg(\RMod_{\cW}(\lCat)) \to \CAlg(\RMod_{\cV}(\lCat))\] is a $2$-functor, and it restricts to presentable categories. Thus, the statement once again follows from \cref{reminder:twoadj} since restriction-of-scalars preserves cotensors, as it factors the restriction of scalars $\CAlg(\IPrV) \to \CAlg(\Pr)$ along $\Spaces \to \cV$ which creates weighted colimits by \cref{prop:CAlgPrVweightedlimits}.
\end{proof}


\bibliographystyle{alpha}
\bibliography{enrichedoperadbib}

\end{document}